\documentclass[11pt]{article}

\usepackage[a4paper, margin=1in]{geometry}
\usepackage{setspace}
\usepackage{longtable}
\usepackage{parskip} 
\usepackage{multicol, upgreek}
\usepackage[OT1]{fontenc}
\usepackage{comment}
\usepackage{multicol, upgreek}

\usepackage{amsmath, amssymb, amsthm, mathtools, thmtools}
\usepackage{mathrsfs} 
\usepackage{bm}       
\usepackage{dsfont}   
\usepackage{enumitem} 

\theoremstyle{plain}
\newtheorem{theorem}{Theorem}[section]
\newtheorem{lemma}[theorem]{Lemma}
\newtheorem{corollary}[theorem]{Corollary}

\theoremstyle{definition}
\newtheorem{definition}[theorem]{Definition}
\newtheorem{example}[theorem]{Example}

\usepackage{graphicx}
\usepackage{float}
\usepackage{subcaption}
\usepackage{tikz}
\usepackage{pgfplots}
\pgfplotsset{compat=1.18}

\usepackage{hyperref} 
\usepackage{cleveref} 
\usepackage{xcolor}   
\usepackage{fancyhdr} 

\newcommand{\ra}{\rightarrow}

\newcommand{\Z}{\mathbb{Z}}

\title{A Classification of Small MSTD Sets in Arbitrary Fields}
\author{Yorick Herrmann \\
{\small\itshape Department of Mathematics, University of Arizona, Tucson, Arizona} \\
{\tt yherrmann@arizona.edu}}
\date{}
\begin{document}

\maketitle

\begin{abstract}
    A finite set $A$ in an additive abelian group is a More Sums Than Differences (MSTD) set if $|A+A| > |A-A|$. We prove that no MSTD set of size $5$ exists in an additive abelian group. Using a computer program inspired by Hegarty, we then obtain classifications of MSTD sets of sizes $6$, $7$, $8$, and $9$ in arbitrary fields. We also investigate the minimal cardinality of MSTD sets that are multiplicative subgroups of $\Z/p\Z$.
\end{abstract}

\section{Introduction}

For a finite set $A$ in an additive abelian group, the sumset $A+A$ and the difference set $A-A$ of $A$ are defined as \begin{align*}
    A+A &= \{a_1+a_2:a_1,a_2\in A\} \\
    A-A &= \{a_1-a_2:a_1,a_2\in A\}.
\end{align*}
Let $|A|$ denote the cardinality of $A$. We say that $A$ is a \textit{More Sums Than Differences} (MSTD) set if $|A+A|>|A-A|$, and we say that $A$ is \textit{sum-difference balanced} if $|A+A|=|A-A|$. The first MSTD sets in $\Z$ were given by Conway, who found the set $\{0,2,3,4,7,11,12,14\}$, and Marica \cite{Marica1969}, who gave the set $\{1,2,3,5,8,9,13,15,16\}$. Nathanson formalized the study of these sets, and provided several methods for constructing MSTD sets in $\Z$ \cite{Nathanson2007}. Hegarty then proved that no MSTD set of size 7 or less exists in $\Z$, and that up to affine transformation, Conway's set is the unique MSTD set of size 8 in $\Z$ \cite{Hegarty2007}. Further research has explored the density of MSTD sets in subsets of $\Z$ \cite{MartinOBryant2007}, explicit constructions of large families of MSTD sets \cite{MillerOroszScheinerman2010, Zhao2010b}, and restricted-sum-dominant sets \cite{PenmanWells2013}. More recent research on MSTD sets in $\Z$ includes \cite{Chu2020, CutlerPebodySarkar2024, HerrmannHillPhillipsFloresMillerSenger2024b, HerrmannHillPhillipsFloresMillerSenger2024a, KumarMohanPandey2024}.

MSTD sets have also been studied, though to a lesser extent, in the setting of
finite abelian groups. Penman and Wells
classified which finite abelian groups contain MSTD sets \cite{PenmanWells2014}, and Zhao determined the asymptotic number of
MSTD sets in finite abelian groups \cite{Zhao2010a}. Miller and
Vissuet showed that as $|G|\to\infty$, almost every uniformly random subset of a finite group is sum-difference balanced \cite{MillerVissuet2014}. 

\subsection*{Main Results}
Our goal in this paper is to prove the following six statements:
\begin{theorem}\label{Size5Thm}
    No MSTD set of size five exists in an additive abelian group.
\end{theorem}

\begin{theorem}\label{Size6Field}
    No MSTD set of size six exists in an arbitrary field $F$. 
\end{theorem}

\begin{theorem}\label{Size 7}
    Up to affine transformation, the only MSTD sets of size seven in arbitrary fields $F$ are $\{0, a, a+2b, a+3b, b, 2b, 4b\}$ when $\operatorname{char}(F)=7$ and $a,b$ are linearly independent over $\Z/7\Z$, $\{0, 1,2,3,5,6,14\}$ when $\operatorname{char}(F)=17$, and $\{0, 1, 2, 3, 8, 10, 13\}$ when $\operatorname{char}(F)=19$. 
\end{theorem}

\begin{theorem}\label{Size8}
    The following list gives every MSTD configuration of size eight in arbitrary fields $F$, up to affine transformation: \begin{itemize}
        \item Conway's set $(\{0,2,3,4,7,11,12,14\})$ when $\operatorname{char}(F)=0$ or $\operatorname{char}(F)=p$ where $p$ is prime and $p \ge 29$.
        \item $\operatorname{char}(F)=3$: $\{0, a, 2a+b+c, 2a+2b+c, 2a+2b+2c, b, 2b+c, c\}$, when $a,b,$ and $c$ are linearly independent over $\Z/3\Z$.
        \item $\operatorname{char}(F)=5$: $\{0, a, a+b, 3a+4b, 4a+3b, 4a+4b, b, 3b\}$, when $a$ and $b$ are linearly independent over $\Z/5\Z$. 
        \item $\operatorname{char}(F)=19: \{0, 1, 2, 3, 4, 6, 7, 16\}, \{0, 1, 2, 3, 4, 7, 8, 15\}$
        \item $\operatorname{char}(F)=23:\{0, 1, 2, 3, 4, 7, 8, 19\}, \{0, 1, 2, 3, 4, 10, 12, 16\}, \{0, 1, 2, 3, 5, 6, 11, 15\}, \{0, 1, 2, 3, 5, 6, 12, 19\}, \\ \{0, 1, 2, 3, 6, 8, 18, 19\}, \{0, 1, 2, 3, 9, 10, 12, 15\}, \{0, 1, 2, 4, 5, 7, 14, 18\}, \{0, 1, 2, 4, 10, 18, 19, 20\}$
        \item $\operatorname{char}(F)=29: \{0, 1, 2, 3, 4, 12, 16, 20\}, \{0, 1, 2, 5, 18, 22, 23, 25\}$
       \item $\operatorname{char}(F)=31: \{0, 1, 2, 3, 6, 7, 17, 27\}$ 
    \end{itemize}
\end{theorem}

\begin{theorem}\label{Size9}
    For all primes $p >47$, every MSTD set in $\Z/p\Z$ of size nine is affinely equivalent to one of the nine integer MSTD sets of size nine. The only MSTD sets of size nine in $\Z/p\Z$ that are not affinely equivalent to an integer MSTD set $A$ where $\max(A)-\min(A)<p/2$ occur when $p$ is between $23$ and $47$, inclusive. There are no MSTD sets of size nine in fields of characteristic $13$, $17$, or $19$. The following list gives every multivariable MSTD configuration of size nine, up to affine transformation, when the variables are linearly independent over $\Z/p\Z$: \begin{itemize}
        \item char$(F)=3:$ \begin{enumerate}
            \item $\{0, a, a+b, a+2b+c, a+2b+2c, 2a+2b+c, b, b+c, c\}$
            \item $\{0, a, a+2b+2c, a+c, 2a+b+2c, b, b+c, c, 2c\}$
            \item $\{0, a, a+b, a+c, 2a+2b, 2a+2b+2c, 2a+2c, b, c \}$
        \end{enumerate}
        \item char$(F)=5:$  \begin{enumerate}
            \item $\{0, a, a+b, a+2b, 2a+b, 3a+2b, 4a+b, 4a+4b, b\}$
            \item $\{0, a, 2a, 2a+b, 2a+3b, 3a+b, 3a+4b, b, 2b\}$ 
            \item $\{0, a, a+3b, 3a+2b, 4a, 4a+3b, b, 2b, 4b\}$ 
            \item $\{0, a, a+4b, 3a+b, 4a, 4a+3b, 4a+4b, b, 3b\}$ 
            \item $\{0, a, a+2b, a+4b, 3a+3b, 4a+b, 4a+3b, b, 2b\}$
        \end{enumerate}
        \item char$(F)=7: \{0, a, a+5b, 2a+4b, 2a+5b, 2a+6b, 3a+6b, 6a+6b, b\}$
        \item char$(F)=11: \{0, a, 2a+9b, 3a+8b, 3a+9b, 4a+7b, 5a+7b, 9a+2b, b\}$
    \end{itemize} 
\end{theorem}

\begin{theorem}\label{SubgroupThm}
    The smallest $p$ for which a multiplicative subgroup $A$ is MSTD in $\Z/p\Z$ is $3221$ $(|A|=n=161)$. The smallest odd order $n$ for which a multiplicative subgroup of order $n$ is MSTD is $n=141$ $(p=4231)$.
\end{theorem}

\subsection*{Background and Terminology}
We must first establish two basic facts about MSTD sets. Let $A=\{a_1,a_2,\dots a_n\}$ be a finite set in an additive abelian group $G$. If $x \in G$, then $B=A+x=\{a+x:a \in A\}$ satisfies  $|B+B|=|A+A|$ and $|B-B|=|A-A|$. Therefore, if $A$ is MSTD, then any shifted copy $B$ is MSTD as well \cite{Nathanson2007}. Next, for $x\in G$, define $x- A= \{x-a:a \in A\}$. If there exists $x \in G$ such that $x-A=A$, then we say that \textit{$A$ is symmetric with respect to $x$}, and we have that $A$ is sum-difference balanced. This is because $|A+A|=|A+(x-A)|=|x+(A-A)|=|A-A|$ \cite{Nathanson2007}.

We now define some terminology that will be used in our proofs. 

\begin{definition}
    A \textit{sum collision} occurs when there exists $a_1,a_2,a_3,a_4\in A$ such that $a_1+a_2=a_3+a_4$ and $\{a_1, a_2\}\ne \{a_3,a_4\}$.
\end{definition}

\begin{definition}
    A \textit{difference collision} occurs when there exist two pairs $(a_1, a_2), (a_3,a_4) \in A \times A$ such that $a_1-a_2=a_3-a_4\ne 0$, and $(a_1, a_2)\ne (a_3,a_4)$.
\end{definition}
By adding and subtracting terms, we see that every difference collision corresponds to no more than one sum collision, and that every sum collision creates at least one corresponding difference collision. However, the number of difference collisions created depends on the specific type of sum collision, of which there are three: 
\begin{definition}[Types of Sum Collisions]
    Let $a,b,c,d$ be distinct elements of a set $A$. A \textit{$C_1$ sum collision} contains two distinct elements and is of the form $2a=2b$. A \textit{$C_2$ sum collision} contains three distinct elements and is of the form $a+b=2c$. A \textit{$C_4$ sum collision} contains four distinct elements and is of the form $a+b=c+d$. 
\end{definition}
Similarly, there are three types of difference collisions:
\begin{definition}[Types of Difference Collision]
    Let $a,b,c,d$ be distinct elements of a set $A$. A \textit{$C_1$ difference collision} contains 2 distinct elements and is of the form $a-b=b-a$. A \textit{$C_2$ difference collision} contains 3 distinct elements and is of the form $a-c=c-b$. A \textit{$C_4$ difference collision} contains four distinct elements and is of the form $a-c=d-b$.
\end{definition}

A $C_1$ sum collision $2a=2b$ is possible only when the doubling endomorphism has nontrivial kernel, and it creates one $C_1$ difference collision, $a-b=b-a$. A $C_2$ sum collision $a+b=2c$ creates two $C_2$ difference collisions: $a-c=c-b$, and $c-a=b-c$. A $C_4$ sum collision $a+b=c+d$ creates four $C_4$ difference collisions: $a-c=d-b$, $c-a=b-d$, $a-d=c-b$, and $d-a= b-c$.

If $A=\{a_1,\dots ,a_n\}$ contains no difference collisions and $|A|=n$, then $|A-A|=n^2-n+1$ ($n^2-n$ nonzero differences). Similarly, if $A$ contains no sum collisions, then $|A+A|= n(n+1)/2$. In this case, $A$ is sometimes referred to as a Sidon set. Next, we define multiplicity-$k$ sums and differences: \begin{definition}
    Let $A=\{a_1,\dots, a_n\}$. A \textit{multiplicity-$k$ sum} is a sum $x$ for which there exists at least $k$ distinct pairs $(a_i, a_j) \in A\times A$ such that $a_i+a_j=x$ and $1 \le i \le j \le n$.
\end{definition} 

\begin{definition}
    A \textit{multiplicity-$k$ difference} is a difference $y \ne 0$ for which there exists at least $k$ many distinct pairs $(a_i, a_j) \in A\times A$ such that $a_i-a_j=y$.
\end{definition}

\begin{definition}
    If a sum/difference is multiplicity-$k$ but not multiplicity-$(k+1)$, then we say that sum/difference has \textit{exact multiplicity-$k$}.
\end{definition}

If a sum/difference has exact multiplicity-$k$, then that sum/difference creates $\binom{k}{2}$ many sum/difference collisions. Additionally, there are finitely many distinct forms in which a multiplicity-$k$ difference can occur. We establish the types of possible patterns present for a multiplicity$-k$ difference: \begin{definition}[Types of difference patterns]
    Let $a_1,\dots ,a_{n+1}$ be distinct elements of a set $A$. An \textit{$n-$cycle} is a difference of the form $a_1-a_2=a_2-a_3=a_3-a_4=\cdots = a_n-a_1$. An \textit{$n-$chain} is a difference of the form $a_1-a_2=a_2-a_3=a_3-a_4=\cdots=a_{n}-a_{n+1}$. A \textit{loose difference} is a nonzero difference pair that is not in an $n-$cycle or an $n-$chain.
\end{definition}
   
As an example, a difference of exact multiplicity-3 can be in the form of a 3-cycle, a 2-cycle and a loose difference, a 2-chain and a loose difference, a 3-chain, or 3 loose differences. Finally, we establish what it means for variables to be interchangeable.
\begin{definition}\label{def:interchangeable}
    Let $A=\{0,a_1,\dots ,a_{n-1}\}$ be a configuration of free variables and let $\mathcal{R}$ denote the set of sum collisions
    assumed to hold at a given stage of an
    argument. Two elements $x,y$ are \textit{interchangeable} if the
    permutation of the labels that swaps $x$ and $y$ and fixes every other
    element maps $\mathcal{R}$ to itself.
\end{definition}

Interchangeability will be a useful tool when examining potential difference collisions since if $a,b$ are interchangeable, then any difference collision through $a$ is essentially equivalent to the corresponding collision after swapping $a$ and $b$, so we need only consider one of them. Note that interchangeability is an equivalence relation, so if $x,y$ are interchangeable and $y,z$ are interchangeable, then $x,z$ are interchangeable, and we can say that $x,y,z$ are interchangeable. We provide an example of when variables are and are not interchangeable:

\begin{example}
    Let $A = \{0,a,b,c,d,e,f,g\}$, and assume $a+b=c+d$ is currently the only sum collision in $\mathcal{R}$. Then $a,b$ are interchangeable, $c,d$ are interchangeable, and $e,f,g,0$ are interchangeable. This example shows that 0 can be an interchangeable element even though it is not a variable. However, $a,e$ are not interchangeable, since this yields $e+b=c+d$, which is not in $\mathcal{R}$.
\end{example}
With these definitions in hand, we now describe the structure of this paper. 

\subsection*{Paper Outline}
Section 2 proves Theorem \ref{Size5Thm} by reaching a series of contradictions if we assume a MSTD set of five exists in an additive abelian group. Section 3 introduces a computer program, inspired by Hegarty's algorithmic approach in the integer
setting, to search for MSTD configurations in arbitrary fields. The program is used to prove Theorem \ref{Size6Field}, Theorem \ref{Size 7}, Theorem \ref{Size8}, and Theorem \ref{Size9}. Section 4 examines the cardinality of MSTD multiplicative subgroups in $\Z/p\Z$ for primes $p$ and proves Theorem \ref{SubgroupThm}.

\section{Smallest MSTD Set in an Arbitrary Group}

Although Hegarty proved that no MSTD set of size 7 or less exists in the integers, this result does not extend to finite abelian groups. For example, the sets $\{0,1,2,4,5,9\} \subset \Z/{12}\Z$ and $\{0,1,2,4,5,12\} \subset \Z/{15}\Z$ are both MSTD \cite{PenmanWells2014}. This motivates the question of what the smallest cardinality for an MSTD set in an arbitrary abelian group is. First, we recall the definition of a stabilizer:

\begin{definition}
    Let $A$ be a subset of an additive abelian group $G$. Then the stabilizer $S(A)$ of $A$ is defined as $$
    S(A) = \{g  \in G: g + A = A\}.
    $$
\end{definition}

We now recall the following result from standard group theory.

\begin{lemma}\label{StabSUBGROUP}
     Let $A$ be a subset of an additive abelian group $G$. Then the stabilizer $S(A)$ of $A$ is a subgroup of $G$. Therefore, when $G$ is finite, we have that $|S(A)|$ divides $|G|$. Additionally, since $A+S(A) = A$, we have that $A$ is a union of cosets of $S(A)$, so $|S(A)|$ divides $|A|$.
\end{lemma}

We will also use Kneser's Theorem, which is given below.

\begin{theorem}[Kneser's Theorem \cite{Kneser1953}]\label{kNESER'S}
    Let $G$ be an additive abelian group, and let $A,B$ be finite non-empty subsets of $G$. Let $S(A+B)$ denote the stabilizer of $A+B$. Then $$
    |A+B| \ \ge \ |A+S(A+B)| + |B+S(A+B)| - |S(A+B)|.
    $$ 
\end{theorem}

We now apply Kneser's Theorem to prove Theorem \ref{Size5Thm}.

\begin{proof}

Assume that $A$ is a MSTD set of size 5 in some arbitrary abelian group. Our proof follows this outline: \begin{enumerate}
    \item We use Kneser's Theorem to show $|A-A| \ge 8$ and $|A+A| \ge 9$.
    \item We show that no difference in $A-A$ occurs more than 3 times. The proof goes by contradiction and involves seven total cases.
    \item We show that there cannot be more than two multiplicity-3 differences. The proof goes by contradiction and involves three total cases.
    \item We use the fact that there cannot be more than two multiplicity-3 differences to show that $|A-A|\ge 10$ and $|A+A|\ge 11$. We then prove by contradiction that no multiplicity-3 difference exists, which forces $|A-A| \ge 11$ and $|A+A|\ge 12$.
    \item We prove that $|A-A|\ge 11$ leads to a contradiction. 
\end{enumerate}

We may assume that $A = \{0,a,b,c,d\}$ since the sizes of sumsets and difference sets are invariant under shifting. Since $A$ is MSTD, we know that $ |A+A|\le 15$, hence $5 \le |A-A| \le 14$.

\subsection*{Proving $|A-A| \ge 8$ and $|A+A|\ge 9$}

 If $|A-A|=5$, then Lemma \ref{StabSUBGROUP} and Kneser's Theorem force $|S(A-A)|=5$. Pick $g \in S(A-A)$. Since $0 \in A-A$, we know that $g \in A-A$. Therefore, $S(A-A) \subseteq A-A$, and since $|S(A-A)|=|A-A|=5$, we have that $S(A-A)=A-A$. Now, pick $a \in A$. Then, $A- a \subseteq A-A$, so $A \subseteq a + (A-A)$. Thus, $$A+A \subseteq 2a + (A-A)+(A-A)= 2a+(S(A-A)+(A-A)) = 2a+(A-A) $$ Hence, $|A+A| \le |2a+(A-A)|=|A-A|$, so $A$ is not MSTD. Since Lemma \ref{StabSUBGROUP} and Kneser's Theorem also force $|S(A-A)|=|A-A|$ when $|A-A|=6$ or $|A-A|=7$, we can follow the same logic to reach a contradiction in these cases. Thus, $|A-A|\ge 8$ and $|A+A|\ge 9$.

\subsection*{No Multiplicity-4 Difference Exists}

If a multiplicity-5 difference occurs, then WLOG, we have one of the following two cases: \begin{enumerate}
    \item  (5-cycle) $a-b=b-c=c-d=d-0=0-a$. Then $b=2a$, $c=3a$, $d=4a$, so $A=\{0,a,2a, 3a,4a\}$, which is symmetric with respect to $4a$ and hence $A$ is not MSTD.
    \item (2-cycle, 3-cycle) $a-b=b-a=c-d=d-0=0-c$. Then $d-c=d-0$, so $c=0$, a contradiction.
\end{enumerate}

If a multiplicity-4 difference exists, then WLOG, it must fall into one of the following cases:
\begin{enumerate}
    \item ($2$-cycle, $2$-chain) $a-b= b-a = c-0 = 0-d$. Then $c=-c=-d$, so $c=d$, a contradiction.
    \item ($4$-chain) $a-b = b-c = c-d= d-0$. Then $A = \{0,d,2d,3d, 4d\}$, so $A$ is symmetric with respect to $4d$.
    \item ($4$-cycle) $a-b=b-c=c-0=0-a$. Then we get the sum collisions $2b=a+c=0$, $a=b+c$, $2a=b=2c$, and $a+b=c$, so $|A+A| \le 9$. We also have $|A-A| \ge 9$ since $0, \pm (a-d), \pm (b-d), \pm (c-d)$, the 4-cycle and its negative are all distinct without assuming further collisions. Eliminating another difference forces another sum to be eliminated, so $|A+A| \le 8$ and $|A-A| \le 7$. 
    \item ($3$-cycle, $1$ loose difference) $a-b=b-0=0-a=c-d$. Then $a=2b$, $b=2a$, $a+b=0$, $a+d=b+c$, $b+d=c$, and $a+c=d$, so $|A+A| \le 9$. We also get the distinct differences $a-c=b-d$, $b-c=-d$, and $-c=a-d$ (and their negatives), so $|A-A|=9$ without further collisions. Eliminating another difference forces $|A+A| \le 8$ and $|A-A| \le 7$.
    \item (two $2$-cycles) $a-b = b-a= c-d = d-c$. This case is handled separately, below. 
\end{enumerate}

Assume a multiplicity-4 difference with two 2-cycles occurs ($a-b=b-a=c-d=d-c$), then if we assume no further collisions occur, we have the following 13 nonzero differences: \begin{itemize}
    \begin{multicols}{2}
        \item $a-b=b-a=c-d=d-c$
        \item $a-c=b-d$
        \item $c-a=d-b$
        \item $a-d=b-c$
        \item $d-a=c-b$
        \item $\pm a, \pm b, \pm c, \pm d$ 
    \end{multicols}
    \end{itemize}
    We also have the following 11 sums: \begin{itemize}
    \begin{multicols}{2}
        \item $2a = 2b$
        \item $2c=2d$
        \item $a+d=b+c$
        \item $a+c=b+d$
        \item $a+b,c+d$
        \item $a,b,c,d,0$
    \end{multicols}
    \end{itemize}

    Note that $a,b$ are interchangeable, and $c,d$ are interchangeable. If any of the differences currently of exact multiplicity-$2$ are in another collision, then due to interchangeability, it suffices to check the cases where $a-c=b-d$ is in another collision: \begin{itemize}
        \item If $a-c=b-d=c-a=d-b$, then we get $|A+A|=9$ while $|A-A|=12$, and another forced collision gives $|A+A|\le 8$.
        \item If $a-c=b-d=d-a=c-b$, then $|A+A|=9$ while $|A-A|=12$. 
        \item If $a-c=b-d=-a$, then $|A+A|=9$ while $|A-A|=10$. We get the same contradiction if $a-c=b-d$ is equal to $-b, c,$ or $d$.
    \end{itemize} 
    
    Next, we can have at most two difference collisions among $\{\pm a, \pm b, \pm c, \pm d\}$, and we now show that no such pair of possible difference collisions causes $A$ to be MSTD. Due to interchangeability, it suffices to consider the cases where $a$ and $c$ are in difference collisions: 
    \begin{itemize}
        \item If $a=-a$, then $b=-b$, so $|A+A|=10$ while $|A-A|=12$. If $c=-c$, then $a-c=c-a$. If $c=-d$, then $2a=c+d=0$, and $a-c=d-a$.
        \item If $a=-b$, then $|A+A|=10$ while $|A-A|=12$. If $c=-c$, then $a+b=2c=0$, so $a-c=c-b$. If $c=-d$, then $a+b=c+d=0$, so $a-c=d-b$.
        \item If $a=-c$, then $b=-d$, so $|A+A|=10$ while $|A-A|=10$. However, no further difference collision exists.
    \end{itemize}

    Since all cases are invalid, no multiplicity-4 difference occurs.

\subsection*{No More Than Two Multiplicity-3 Differences Exist}    

To prove this claim, we proceed by contradiction. There are three types of multiplicity-3 differences: \begin{enumerate}
    \item (3-chain) $a-b=b-c=c$
    \item (3-cycle) $a-b=b=-a$
    \item (2-chain, 1 loose difference) $a-b=b-c=d$
\end{enumerate}
Note that any multiplicity-3 difference containing a 2-cycle must also be a multiplicity-4 difference. 

For each case in this subsection, we will list out the current differences and sums given by each type of multiplicity-3 difference. We then show that if another multiplicity-3 difference appears, we must either have $|A+A| \le 8$ or a multiplicity-4 difference.

\textbf{Case 1:} If a 3-chain $a-b=b-c=c$ occurs, then we get the following 14 nonzero differences: \begin{itemize}\begin{multicols}{2} 
\item $a-b=b-c=c$
\item $b-a=c-b=-c$
\item $a-c=b$
\item $c-a=-b$
\item $ \pm a, \pm d$
\item $\pm (a-d), \pm (b-d), \pm (c-d)$
\end{multicols}
\end{itemize}

Additionally, $A+A$ contains the following 12 sums:
\begin{itemize}\begin{multicols}{2} 
\item $2b=a+c$
\item $a=b+c$
\item $b=2c$
\item $0,c,d$
\item $2a, 2d$
\item $a+b,a+d,b+d,c+d$
    \end{multicols}
\end{itemize}

We claim that if $a-c=b$ is in a difference collision with another nonzero difference, then we either get a multiplicity-4 difference or we eventually force $|A+A|\le 8$. This claim is verified in the appendix by checking through each potential difference collision involving $a-c=b$. For instance, if $a-c=b=d-b$, then $a-d=c-b$, so $c-b$, is a multiplicity-4 difference. 

Therefore, if a second multiplicity-3 difference occurs, then $a$ is a multiplicity-3 difference, since every other difference that has exact multiplicity-1 contains the variable $d$. We claim that when $a$ is a multiplicity-3 difference, then we either get a multiplicity-4 difference or we force $|A+A| \le 8$. This claim is verified in the appendix by checking through each potential multiplicity-3 difference involving $a$. For instance, if $a=-d=d-b$, then $b=d-a$, which gives that $a-c=b$ is a multiplicity-3 difference.

\textbf{Case 2:} If a 3-cycle $a-b=b=-a$ occurs, then we get the following 16 nonzero differences: \begin{itemize}\begin{multicols}{2} 
\item $a-b=b=-a$
\item $b-a=-b=a$
\item $\pm c, \pm d$
\item $\pm (a-c), \pm (b-c)$
\item $\pm (a-d), \pm (b-d)$
\item $\pm (d-c)$
\end{multicols}
\end{itemize}

Additionally, $A+A$ contains the following 12 sums:
\begin{itemize}\begin{multicols}{2} 
\item $2b=a$
\item $a+b=0$
\item $b=2a$
\item $2c, 2d$
\item $c, a+c,b+c, c+d$
\item $d,a+d, b+d$
\end{multicols}
\end{itemize}

We have that $a,b,0$ are interchangeable, and that $c,d$ are interchangeable. Therefore, if another multiplicity-3 difference occurs, then WLOG, we can assume that $c$ is in this difference. We now have the following possibilities: \begin{itemize}
    \item $c=-d=a-c$. Then $a=c-d$, creating a multiplicity-4 difference.
    \item $c=-d=d-a$. Then $a=d-c$.
    \item $c=-d=d-c$. Then $a+b=c+d=0$, $2d=c$, and $2c=d$, so $|A+A|=9$. We get the differences $-c=d=c-d$, $a-c=d-b$, $c-a=b-d$, $a-d=c-b$, and $d-a=b-c$, so $|A-A|=9$. 
\end{itemize}

\textbf{Case 3:} If a multiplicity-3 difference with a 2-chain and one loose difference occurs ($a-b=b-c=d$) then we get the following 12 nonzero differences:  
\begin{itemize}
    \begin{multicols}{2}
    \item $a-b = b-c = d$
    \item $b-a = c-b = -d$
    \item $a-d = b$
    \item $d-a = -b$
    \item $b-d=c$
    \item $d-b=-c$
    \item $\pm (a-c), \pm a$
    \item $\pm (c-d)$
    \end{multicols}
\end{itemize}
Additionally, $A+A$ contains the following 12 sums:
\begin{itemize}
    \begin{multicols}{2}
        \item $2b = a+c$
        \item $a=b+d$
        \item $b=c+d$ 
        \item $2a,2c, 2d$
        \item $a+b, a+d, b+c$
        \item $c,d,0$
    \end{multicols}
\end{itemize}

We now consider all new possible multiplicity-3 differences that do not fall in Case 1 or Case 2: \begin{itemize}
    \item $a-d=b=c-a$ gives $b-c=-a$, so $b-c$ is a multiplicity-4 difference.
    \item $a-d=b=d-c$ gives $b-d=-c$, so $b-d$ is a multiplicity-4 difference.
    \item $b-d=c=a-c$ gives $a-b=c-d$, so $a-b$ is a multiplicity-4 difference.
    \item $b-d=c=-a$ gives $a-d=-b$, so $a-d$ is a multiplicity-4 difference.
\end{itemize}

Every multiplicity-3 difference that can be formed between $\pm (a-c)$, $\pm (c-d)$, and $\pm a$ falls into either Case 1 or Case 2. 

Since every case results in a contradiction, no more than two multiplicity-3 differences occur.

\subsection*{Disproving the Remaining Cases}

 There are a total of 20 possible nonzero difference pairs in $A-A$, which are given by the set $D = \{(x,y): x, y \in A, x\ne y\}$. If $|A-A|=8$, then there are exactly 7 nonzero differences in $A-A$. Therefore, we have $|D| \le 5(2)+2(3)=16$, a contradiction. Similarly, if we assume $|A-A|=9$, then we have that $|D| \le 6(2)+2(3)=18$, a contradiction. 

Therefore, $|A-A|\ge 10$, and $|A+A| \ge 11$. If any multiplicity-3 difference occurs, then we get $|A+A|=12$ and $|A-A|\ge 13$. We claim that every possible difference collision now forces either a third multiplicity-3 difference, or $|A+A| \le 10$. This claim is verified in the appendix by checking through each potential difference collision after assuming a multiplicity-3 difference occurs. We fully check through each of the three types of multiplicity-3 differences. As an example, if a 3-chain $a-b=b-c=c$ occurs, then if $a=-d$, we get $|A+A|=11$ and $|A-A|=13$. 

Therefore, no multiplicity-3 difference can occur. If $|A-A|=10$, then there are nine nonzero differences, so $|D| \le 9(2)=18$, a contradiction. Hence, $|A-A| \ge 11$, so $|A+A| \ge 12$. If $|A-A|=11$, then there are ten multiplicity-2 differences. If WLOG, the $C_4$ sum collision $a+b=c+d$ exists, then at present, $A-A$ contains the following 16 nonzero differences: \begin{itemize}
    \begin{multicols}{2}
    \item $a-c=d-b$
    \item $c-a = b-d$
    \item $a-d = c-b$
    \item $d-a = b-c$
    \item $\pm (a-b), \pm (c-d)$
    \item $\pm a, \pm b, \pm c, \pm d$
     \end{multicols}
\end{itemize}

Additionally, $A+A$ contains the following 14 sums:
\begin{itemize}
    \begin{multicols}{2}
        \item $a+b=c+d$
        \item $a+c,a+d,b+c,b+d$
        \item $2a,2b,2c,2d$
        \item $a,b,c,d,0$
    \end{multicols}
\end{itemize}
We see that $a,b$ are interchangeable and $c,d$ are interchangeable. For $a-b$, we get a multiplicity-3 difference unless $a-b=b$, WLOG. Similarly, we have that $c-d=d$, WLOG. Then $|A+A|=12$ and $|A-A|=13$, but another difference collision forces $|A+A| \le 11$. 

Therefore, each multiplicity-2 difference corresponds to either a $C_1$ or a $C_2$ sum collision. As such, there are at least five sum collisions, but this can only happen if a multiplicity-4 sum occurs since $|A+A|\ge 12$. This means we must have a sum that is one of two forms: \begin{enumerate}
    \item $2a=2b=2c=2d$. Then we have six $C_1$ collisions, which gives a total of six difference collisions.
    \item $2a = 2b = 2c= d+0$. Then we have three $C_1$ collisions and three $C_2$ collisions, which gives a total of nine difference collisions.
\end{enumerate} 
Both forms lead to contradictions since there are 10 difference collisions. Hence, $|A-A|=11$ is impossible.

If $|A-A|=12$, then there are nine multiplicity-2 differences and two multiplicity-1 differences, so there are nine difference collisions. Therefore, at least three sum collisions occur, and since $|A+A| \ge 13$, a multiplicity-3 sum must occur. However, no multiplicity-3 sum produces two $C_4$ sum collisions and one $C_1$ sum collision, which is the only theoretical way a multiplicity-3 sum can yield nine difference collisions. Hence, $|A-A|=12$ is impossible. Therefore, $|A-A|\ge 13$, so we have no more than one sum collision, which forces $|A-A| \ge 16$, a contradiction. Since every value of $|A-A|$ is impossible, $A$ cannot be MSTD.
\end{proof}

\section{Small MSTD Sets in Fields}

We now consider MSTD sets in fields. We first establish two relevant definitions. \begin{definition}
    Let $A$ be a set in a field $F$. If $x \in F$ and $x\ne 0$, then $B=x \cdot A = \{x\cdot a: a\in A\}$ is a \textit{dilation} of $A$.
\end{definition}  Since every nonzero element of a field has a unique multiplicative inverse, we know that the operation $T_x: F \ra F$ defined as $T_x(a)=a\cdot x$ is a bijection when  $x\ne 0$. As such, $|B+B|=|A+A|$ and $|B-B|=|A-A|$ \cite{Nathanson2007}. Additionally, we have the following definition:
\begin{definition}
    A set $C$ is \textit{affinely equivalent} to a set $A$ in a field $F$ if there exists $x,y \in F$, $x \ne 0$ such that $x\cdot A+y=C$. Equivalently, $C$ is also referred to as an affine transformation of $A$.
\end{definition}

We introduce a computer program that produces the MSTD classifications given in this section, which was inspired by Hegarty's search for small MSTD sets in the integers. 

\subsection*{Hegarty's Idea}
In \cite{Hegarty2007}, Hegarty describes the steps his program takes to
classify MSTD sets in the integers, and our program adapts this strategy. The starting observation is simple: An MSTD set $A=\{0,a_1,\dots,a_{n-1}\}$ must contain difference collisions. Each new, independent difference collision in the integers removes a free variable, since it expresses one variable as a linear combination of the
others. For instance, if $a_1-a_2=a_3-a_4$, then $a_1=a_2+a_3-a_4$, and $A=\{0,a_2+a_3-a_4,a_2,a_3, \dots, a_{n-1}\}$. Hegarty's idea is to use these collisions to progressively restrict the structure of $A$: one forces a difference collision, eliminates a
variable, and checks whether the resulting configuration is MSTD; if it is
not, one forces a further collision and repeats. This process terminates since $A$ is finite, and when $A$ depends on only one free variable $a_i$, dividing each element
by $a_i$ reduces the problem to checking whether the resulting set is MSTD in $\Z$. Hegarty's program exhaustively checks through every possible combination of difference collisions. In this way every MSTD set of size $n$ (up to affine transformation) is eventually located.

Running this program directly is infeasible even for $n=8$. To account for this, Hegarty proves that any size-$8$ MSTD set must be in the form of at least one of $18$ configurations, each of which contains no more than five free variables. His program then examines each of these configurations in turn rather than searching from scratch. Before delving into the computational limits for our program, we first explain how our program differs from Hegarty's.

\subsection*{Program Modifications for Arbitrary Fields}
Working over arbitrary fields rather than $\Z$ introduces a few new considerations that our program must account for. First, an arbitrary field may not be an ordered set, so any two nonzero difference pairs can be in a collision as long as the collision does not force two distinct values of $A$ to be equal. This is opposed to $\Z$, where we can rule out any positive difference equaling a negative difference. Second, instead of just checking if a set $A$ is MSTD over $\Z$, we have to check if $A$ is MSTD in fields with prime characteristic $p$, for various $p$. Third, we need to account for the fact that a nonzero prime characteristic can lead to errors when progressively removing free variables.

Our program follows these steps: \begin{enumerate}
    \item We input the initial configuration of free variables. For instance, $A= \{0, a,b, a-b\}$.
    \item The program calculates all possible difference collisions that are not redundant (does not force $0=0$). For instance, $a-0=b-0$ is impossible because we assume $a$ and $b$ are distinct, and $(a-b)-a = 0-b$ is redundant because this gives $0=0$. However, $a-0=0-b$ is possible.
    \item The program picks a difference collision, and then eliminates a free variable in the configuration. For instance, if $a-0=0-b$, then $a=-b$, so $A=\{0, -b,b,-2b\}$. 
    \item If $A$ still has more than 1 free variable, we check if the configuration is MSTD, and then repeat Steps 2 and 3 for the new configuration. If $A$ has 1 free variable, we dilate $A$ by that variable's multiplicative inverse. For instance, $b^{-1}\cdot A=\{0, -1,1,-2\}$.
    \item Let $M=\max(A)$ and $m= \min(A)$. We check if $A$ is MSTD in $\Z$ and in $\Z/p\Z$ for all primes $p$ such that $|A|<p \le 2M-2m$. If $p >2M-2m$, then $A+A$ and $A-A$ have the same cardinality as subsets of $\Z$ that they do as subsets of $\Z/p\Z$. Also, note that $\Z$ is a proxy for an arbitrary characteristic 0 field (even though $\Z$ itself is not a field), since $|A+A|$ and $|A-A|$ are equivalent in $\Z$ and any characteristic 0 field. As an example for this step, if $A= \{0,-1,1,-2\}$, $2M-2m=6$, so we check if $A$ is MSTD in $\Z$ and $\Z/5\Z$.
    \item We repeat Steps 2-5 until all possible difference collision branches are exhausted.
\end{enumerate}


Of important note is how we deal with collisions that give linear relations such as $k_1a=k_2b$, where $k_1, k_2>1$ are integer scalars. If GCD$(k_1,k_2)=1$, and if $k_1 <k_2$, then we dilate $A$ by a factor of $k_1$ and substitute in $k_2b$ wherever $k_1a$ appears. Furthermore, our field $F$ cannot have characteristic dividing $k_1$, since then $k_2b=0$, but $k_2b \ne 0$ since $b\ne 0$ and GCD$(k_1,k_2)=1$. We get the same contradiction if our field $F$ has characteristic dividing $k_2$. 

If GCD$(k_1, k_2)>1$, then we multiply $k_1$ and $k_2$ by the multiplicative inverse of their GCD (if it exists), and then apply the relation we get. For instance, if our collision is $2a-4b=2b-a$, then $3a=6b$ is our linear relation, and we apply the relation $a=2b$. However, if the characteristic of $F$ divides the GCD, then the relation $k_1a =k_2b$ is redundant, and reducing the terms by their GCD could lead to an inaccurate collision. To account for this, the program tracks the maximal prime factor dividing the GCD of such collisions. Let $\phi$ denote this maximal prime factor. Note that $\phi$ also tracks $k_1$ in difference collisions of the form $k_1a=0$ (before marking it as invalid, since reducing by the `GCD' gives $a=0$, a contradiction). The program's output will return all MSTD sets (up to affine transformation) in the given configuration in fields with characteristic 0 or characteristic greater than $\phi$, but if $0<\operatorname{char}(F)\le \phi$, it is unclear how many MSTD sets of the given configuration exist or not. It is simple enough to check $\Z/p\Z$ for all primes $p $ such that $0 < p \le \phi$, but no such finite computation can be performed for arbitrarily large fields with the same characteristic $p$. 

To account for the region where $0< \operatorname{char}(F)\le \phi$, the program admits an optional parameter $\beta$. When $\beta=p$, the program treats $A$ as if it lives in a field of characteristic $p$. The primary change occurs in Step~2: after computing $A+A$ and $A-A$, the program reduces all coefficients modulo $p$, which prevents redundant collisions. For example, if $\operatorname{char}(F)=3$ and $A-A=\{0, a-2b, b-2a\}$, then without the parameter the collision $a-2b=b-2a$ would yield $3a=3b$ and hence $a=b$, which may be inaccurate. With $\beta=3$, however, $A-A=\{0,a+b,b+a\}=\{0,a+b\}$, so the redundant collision never arises. The parameter also affects Step~4, as it alters the computed sizes of $A+A$ and $A-A$, thereby ensuring no MSTD configuration is overlooked. For instance, if $A+A=\{0,a+b,2a+2b\}$ and $A-A=\{0,a-2b,b-2a\}$, the original program would fail to classify this configuration as MSTD. However, with $\beta=3$, $A+A$ remains unchanged while $|A-A|=2$, and the configuration is correctly identified as a multivariable MSTD configuration. If the program returns no MSTD sets under the parameter $\beta =p$, then we can conclude that no field of characteristic $p$ admits an MSTD set in the form of the inputted configuration. This becomes especially important for fields of $p^r$ elements with $r>1$. In a multivariable configuration, continuously forcing additional difference collisions eventually restricts $A$ to a subset of a field additively isomorphic to $\mathbb{Z}/p\mathbb{Z}$. However, if the variables are linearly independent over $\Z/p\Z$, then no further difference collisions can be forced, and any MSTD configuration remains MSTD.

\subsection*{Program Limitations}

It is important to note that the program cannot accurately track MSTD sets in abelian groups that are not fields. The program repeatedly dilates a configuration by a scalar coefficient to eliminate a variable, which preserves $|A+A|$ and $|A-A|$ only when multiplication by that coefficient is bijective. This fails in groups with nontrivial $k-$torsion, as the endomorphism $a \ra k \cdot a$ need not be injective. Additionally, reducing a linear relation by the GCD (call it $g$) of its coefficients is also invalid in groups where the group has non-trivial $g-$torsion. However, in fields, this issue only occurs for finitely many prime characteristics, whereas in arbitrary groups, $g$ may fail to be invertible for infinitely many moduli.

We now describe the computational limits of our program. Like Hegarty's program, the primary determinant in our program's total runtime is the number of free variables in the input configuration, as each additional variable creates an additional level, leading to exponentially more collision branches the program has to check through. The size of the input configuration is also important, as a configuration with a higher cardinality contains many more possible collisions at a given level. 

In practice, if no $\beta$ is input, configurations of size seven or less with five or fewer free variables tend to complete in a few seconds. Configurations of size eight with five free variables complete in a few minutes, and configurations of size nine with five free variables tend to take under an hour to complete. If a configuration has six free variables, then the program takes around two minutes to complete if it has size seven, around an hour to complete if it has size eight, and nearly a day to complete if it has size nine. The program does not complete in a feasible amount of time if we input a configuration with seven or more variables.

Although the program run-times seem feasible when six free variables are inputted, it is important to note that the program outputs a $\phi$, and that we need to individually check through each prime characteristic $p$ satisfying $3 \le p \le \phi$ by using the $\beta$ parameter. Larger configurations with more free variables tend to produce larger $\phi$, which greatly increases the total time needed to fully check through a given configuration. For this reason, we attempt to avoid inputting configurations with six free variables wherever possible, as checking through each $\beta$ case can lead to a configuration taking week(s) to fully check through. An important note is that inputting a small $\beta$ into the program significantly reduces the program's run-time, since the calculated sum and difference sets are much smaller in size at each step. For instance, inputting a size nine configuration with eight free variables and $\beta =3$ takes only a few seconds to complete. However, as $\beta$ increases, the number of potential difference collisions also increases, and the program's run-time begins to approach the run-time when no $\beta$ is input. 

\subsection*{Results for Configurations of Sizes Six and Seven}
Throughout the section, we refer the reader to the appendix to see detailed program outputs for each instance where the program was called.

First, we note that no subset $A$ of a field of characteristic 2 is MSTD, as then $A=-A$ and $A+A=A-A$ \cite{PenmanWells2014}. We now provide a lower bound on the cardinality of MSTD sets in arbitrary fields by proving Theorem \ref{Size6Field}.

\begin{proof}

Chu provided a computer-free proof that no MSTD set of size $6$ exists in the integers \cite{Chu2020}, but his argument relies on the ordering of $\Z$ and therefore does not extend to arbitrary fields.

Assume $A$ is a set of size 6 in an arbitrary field. Inputting the base configuration $A = \{0, a, b, c, d,e\}$, the program returns no MSTD configurations and $\phi = 13$. The program also returns no MSTD configurations when we input $\beta = 3,5,7,11,13$. Hence, $A$ is not MSTD in any field.
\end{proof}

We now examine MSTD sets of size 7 in arbitrary fields, and use the program to prove Theorem \ref{Size 7}.

\begin{proof}
    
Inputting the base configuration $A= \{0,a,b,c,d,e,f\}$, the program returns $\phi = 31$ and the two configurations given in the theorem statement when char$(F) \in \{17,19\}$. Let $X$ denote the set of primes up to 31, excluding 2. Inputting the base configuration for $\beta \in X\setminus \{7,17,19\}$, the program returns no MSTD configurations. For $\beta =7$, the program returned the unique (up to affine transformation) multivariable configuration given in the theorem. For $\beta =17,19$, the program returned the MSTD configurations given in the theorem.
\end{proof}

The multivariable configuration establishes the existence of an MSTD set of size $7$ in every field of characteristic 7, except for the field of 7 elements.

\subsection*{Results for Configurations of Size Eight}
We now examine MSTD sets of size 8, and prove Theorem \ref{Size8}.

\begin{proof}

To avoid inputting any configurations with six free variables, we show that one of three cases must occur, which allows us to input configurations of five variables. Our proof follows this outline: \begin{enumerate}
    \item We assume in Case 1 that a multiplicity-3 sum of the form $a+b=c+d=e$ occurs, and we input the resulting configuration into the program. We track each multivariable configuration found, as well as the largest $p$ for which $p$ can contain an MSTD set in the form of the input configuration in $\Z/p\Z$.
    \item For Case 2, we adopt one of Hegarty's arguments to show that if no multiplicity-3 sum of the form $a+b=c+d=e$ occurs and if no multiplicity-3 difference occurs, we reach a contradiction. 
    \item We assume in Case 3 that our field is not of characteristic 3 or 5, and that a multiplicity-3 difference occurs. We consider each of the three types of multiplicity-3 differences and input each resulting configuration into the program. We track each multivariable configuration found, as well as the largest possible $p$ for which $p$ can contain an MSTD set in the form of the input configuration in $\Z/p\Z$.
    \item We conduct a separate computer search to find all possible MSTD sets of size 8 in $\Z/p\Z$.
    \item We input the base configuration $A= \{0,a,b,c,d,e,f,g\}$ with $\beta =3,5$ to find all MSTD sets of size 8 in fields of characteristic 3 and 5.
\end{enumerate}

For MSTD sets of size 8, since Conway's set has a maximum of 14 and a minimum of 0, if char$(F)\ge 29$, then Conway's set is MSTD in $F$. Now, let $\gamma$ denote the maximal prime characteristic $p$ for which the original program finds an MSTD configuration that is not affinely equivalent to an MSTD set in $\Z$, and let $X_{p}$ denote the set of primes up to and including $p$, but excluding 2,3, and 5. 

We now consider the cases given in the outline. Note that we assume for now that char$(F)\ne 3,5$.

\textbf{Case 1:} A multiplicity-$3$ sum containing at least 6 distinct elements occurs. Then WLOG, we must have $a+b=c+d=e$. Inputting $A=\{0,a,e-a,c,e-c,e,f,g\}$, the program returns $\phi = 31$, $\gamma=31$, and no multivariable configurations. A check with $\beta \in X_{31}$ also returns no multivariable configurations. 

\textbf{Case 2:} No multiplicity-$3$ sum containing at least 6 distinct elements occurs, and no multiplicity-3 difference occurs. We now adopt one of Hegarty's ideas to further restrict the structure of $A$ if $A$ is MSTD \cite{Hegarty2007}. Let $\delta$ denote the number of nonzero differences in $A-A$, and let $\sigma = |A+A|$. First, $\delta$ must be even since $F$ cannot have characteristic 2. Second, we know $\delta+ 1<\sigma \le 36$, so $\delta \le 34$. We also know that there are a total of $8^2-8=56$ nonzero difference pairs. Therefore, we have at least 22 difference collisions. 

Any sum collision not belonging to a multiplicity-3 sum eliminates one sum and corresponds to at most four difference collisions. If a multiplicity-3 sum occurs, then WLOG, it must be of the form $a+b=c+d=2(0)$. This eliminates two sums and corresponds to eight difference collisions, since there are two $C_2$ sum collisions. Although one of its eliminated sums may appear to correspond to six difference collisions, the other then corresponds to only two. Consequently, at least one sum is eliminated for every four difference collisions, so $\sigma \le 36-0.25D$, where $D$ is the number of difference collisions. Since $D \ge 22$, we have $\sigma \le 36-0.25(22)=30.5$, so $\delta \le 28$. Hence, at least 28 difference collisions occur, so $\sigma \le 36-0.25(28)=29$. Therefore, $\delta \le 26$, so by pigeonhole principle, a multiplicity-$3$ difference must occur, and so this case is impossible.

\textbf{Case 3:} A multiplicity-3 difference occurs. Since we assume char$(F)\ne 3$, one of the following three equalities must occur, WLOG: \begin{enumerate}
    \item (3-chain) $a-b=b-c=c-0$. Inputting the configuration $A=\{0,3c,2c,c,d,e,f,g\}$, the program returns no multivariable configurations, Conway's set, $\phi = 47$, and $\gamma =31$. When we check $\beta \in X_{47}$, we get no multivariable configurations. 
    \item (2-chain, 1 loose difference) $a-b=b-c=d-0$. Inputting the configuration $A=\{0,c+2d,c+d,c,d,e,f,g\}$, the program returns no multivariable configurations, Conway's set, $\phi =41$, and $\gamma =31$. When we check $\beta \in X_{41}$, we get no multivariable configurations.
    \item (3 loose differences) $a-b=c-d=e$. Inputting $A=\{0,a,a-e,c,c-e,e,f,g\}$, the program returns no multivariable configurations, Conway's set, $\phi = 31$, and $\gamma =31$. When we check $\beta \in X_{31}$, we get no multivariable configurations.
\end{enumerate} 

Next, we observe that across all cases, $\gamma$ did not exceed 31. The lack of multivariable MSTD configurations allows us to computationally check $\Z/p\Z$ to find all MSTD sets up to affine transformation in fields with characteristic $p$ such that $7 \le p \le 31$. This computational check is done by directly testing if every eight-element subset of $\Z/p\Z$ is MSTD in $\Z/p\Z$ (note we can fix 0 in each subset). Detailed program outputs for these computational checks are available in the appendix. Every MSTD set found is given in the theorem. 

If we run the program with $\beta =3$ and the base configuration $A= \{0,a,b,c,d,e,f,g\}$, the program returns the  multivariable configuration given in the theorem. This configuration is MSTD in all fields of characteristic 3 as long as $a,b$ and $c$ are linearly independent elements over $\Z/3\Z$. Therefore every field of characteristic 3 contains an MSTD set of size 8 except the fields of 3 and 9 elements. Next, when we run the program with $\beta =5$ and the base configuration, the program returns the multivariable configuration given in the theorem.

\end{proof}

\subsection*{Results for Configurations of Size Nine}

We now examine MSTD sets of size $9$. Hegarty found 9 MSTD sets in $\Z$ when assuming a multiplicity-4 sum appears \cite{Hegarty2007}. Penman and Wells later verified that these are the only MSTD configurations of size 9 in $\Z$, up to affine transformation. \cite{PenmanWells2013}. These 9 sets are given here: \begin{enumerate}
\begin{multicols}{2}
    \item $\{0,1,2,4,5,9,12,13,14\}$
    \item $\{0, 1, 2, 4, 7, 8, 12, 14, 15\}$
    \item $\{0, 2, 3, 4, 7, 9, 13, 14, 16\}$
    \item $\{0, 2, 3, 4, 7, 11, 12, 14, 16\}$
    \item $\{0, 2, 3, 4, 7, 11, 15, 16, 18\}$
    \item $\{0, 2, 4, 8, 9, 10, 15, 17, 19\}$
    \item $\{0, 4, 6, 7, 8, 14, 15, 17, 21\}$
    \item $\{0, 4, 6, 8, 11, 14, 19, 21, 25\}$
    \item $\{0, 5, 6, 9, 10, 13, 16, 17, 22\}$
\end{multicols}
\end{enumerate}

We now prove Theorem \ref{Size9}.

\begin{proof} 

Let $A$ be a MSTD set of size 9. To reduce the computational complexity, we show that one of five cases must occur. Our proof follows this outline: \begin{enumerate}
    \item We assume in Case 1 that a multiplicity-3 sum of the form $a+b=c+d=e$ occurs, and we input the resulting configuration into the program. We track each multivariable configuration found, as well as the largest possible $p$ for which $p$ can contain an MSTD set in the form of the input configuration in $\Z/p\Z$.
    \item We assume in Case 2 that a multiplicity-3 sum of the form $a+b=c+d=2(0)$ occurs, and that one of its corresponding $C_2$ difference collisions is in a multiplicity-3 difference. We input the resulting configuration into the program, and we track each multivariable configuration found, as well as the largest possible $p$ for which $p$ can contain an MSTD set in the form of the input configuration in $\Z/p\Z$.
    \item We assume in Case 3 that neither Case 1 nor Case 2 occur, that no multiplicity-4 difference occurs, and that at least sixteen multiplicity-3 differences comprised only of $C_4$ collisions occur. We input the resulting configuration into the program, and we track each multivariable configuration found, as well as the largest possible $p$ for which $p$ can contain an MSTD set in the form of the input configuration in $\Z/p\Z$.
    \item For Case 4, we adopt Hegarty's argument to show that if neither Case 1, Case 2, Case 3 nor a multiplicity-4 difference occurs, then we reach a contradiction.
    \item We assume in Case 5 that a multiplicity-4 difference occurs and that the characteristic is not 3. For each type of multiplicity-4 difference, we input the resulting configuration into the program, and we track each multivariable configuration found, as well as the largest possible $p$ for which $p$ can contain an MSTD set in the form of the input configuration in $\Z/p\Z$.
    \item We conduct a separate computer search to find all possible MSTD sets of size 9 in $\Z/p\Z$.
    \item We input the base configuration $A=\{0,a,b,c,d,e,f,g,h\}$ for $\beta =3$ to find all MSTD sets of size 9 in fields of characteristic 3.
\end{enumerate}

Let $X_p$ denote the set of all primes up to and including $p$, but excluding 2 and 3. We now consider the cases given in the outline.

\textbf{Case 1:} A multiplicity-3 sum containing at least six distinct elements occurs. Then WLOG, we must have  $a+b=c+d=e$. Inputting $A=\{0,a,e-a,c, e-c,e,f,g,h\}$, the program returns $\phi =73$, $\gamma =47$, all nine integer sets, and no multivariable configurations. A check with $\beta \in X_{73} \setminus\{5,7,11\}$ returns no multivariable configurations. When $\beta =5,7,11$, we get exactly the multivariable configurations given in Theorem \ref{Size9}. It ended up being simpler to input a size 9 configuration with six free variables as opposed to inputting all possible subcases when a multiplicity-3 sum of this form occurs. Unfortunately, this resulted in a cumulative run-time of 175.3 hours for this case, which is more than a week.

\textbf{Case 2:} A multiplicity-3 sum containing exactly five distinct elements occurs, and one of its corresponding $C_2$ difference collisions is in a multiplicity-3 difference. WLOG, our multiplicity-3 sum is $a+b=c+d=0$, and WLOG, we say that $a=-b$ is in a multiplicity-3 difference. Then WLOG, there are four possibilities to consider, since $a,b$ are interchangeable variable labels, since $c,d$ are interchangeable, and since $e,f,g,h$ are interchangeable. \begin{enumerate}
    \item $a=-b=c-d$. Then $a=2c$, $b=-2c$, so we input $A=\{0,2c,-2c,c,-c,e,f,g,h\}$. The program returns $\phi =61$, $\gamma =41$, no multivariable configurations, and no integer MSTD sets. Checking through $\beta \in X_{61}$, we get no multivariable configurations. 
    \item $a=-b=c-e$. Inputting $A=\{0,c-e, e-c,c,-c,e,f,g,h\}$, the program returns $\phi = 61$, $\gamma = 43$, no multivariable configurations, and seven of the nine integer MSTD sets. Checking through $\beta \in X_{61} \setminus \{5\}$, we get no multivariable configurations. When $\beta =5$, we get four of the five multivariable configurations given in the theorem.
    \item $a=-b=c-a$. Then $c=2a$ and $d=-2a$, so we input the configuration $A=\{0,a,-a,2a,-2a,e,f,g,h\}$. However, this is essentially the same configuration as the case where $a=-b=c-d$, so we can skip this case. 
    \item $a=-b=e-f$. Then we input the configuration $A=\{0,e-f,f-e,c,-c,e,f,g,h\}$, and the program returns $\phi =67$, $\gamma =41$, no multivariable configurations, and three of the nine integer MSTD sets. Checking through $\beta \in X_{67} \setminus \{5\}$, we get no multivariable configurations. When $\beta =5$, we get three of the five multivariable configurations given in the theorem.
\end{enumerate}

The interchangeability of variables allows us to cut down the number of cases we have to check. For instance, we need not check $a=-b=e-c$ since $a$ and $b$ are interchangeable, so the case $a=-b=c-e$ covers the case $b=-a=c-e$, which gives $a=-b=e-c$. 

\textbf{Case 3:} Neither Case 1 nor Case 2 occurs, sixteen multiplicity-3 differences comprised of three loose differences occur, and no multiplicity-4 difference occurs. Each of these sixteen differences is comprised of six distinct elements of $A$. Hence, there are 96 `slots' for elements among the 16 differences, so by pigeonhole principle, some element $a \in A$ appears in at least twelve of these sixteen differences (since it is impossible for $a$ to appear in odd many multiplicity-3 differences). WLOG, let $a-b=c-d=e-0$ be one of these sixteen multiplicity-3 differences. Then $a+d=b+c$, $a=b+e$, and $c=d+e$. If $a-c=b-d$ is in another one of these sixteen differences, then we must have $a-c=b-d=f-g$, WLOG. This is because if $0$ or $e$ are in the third loose difference, we either get a multiplicity-4 difference or a multiplicity-3 sum of at least six distinct elements. Inputting the configuration $A=\{0,b+e,b,d+e,d,e,b-d+g,g,h\}$, the program returns three of the nine integer MSTD sets, no multivariable configurations, $\phi =47$ and $\gamma =37$. Checking $\beta \in X_{47} \setminus \{5\}$, we get no multivariable configurations. When $\beta =5$, we get the five multivariable configurations given in the theorem. 

We have now fully checked the subcase where $a-c=b-d$ is a multiplicity-$3$ difference of three loose differences. The cases where $a-e=b$ and $c-e=d$ are multiplicity-$3$ differences of three loose differences yield identical results. Therefore, the only remaining case to consider is when neither $\pm(a-c)$ nor $\pm (a-e)$ are multiplicity-3 differences of three loose differences. Since $a$ appears in twelve of the sixteen multiplicity-3 differences, $a-d$ must be a multiplicity-3 difference containing three loose differences. If none of $f,g,h$ occur in one of the other two loose differences, then we get a multiplicity-4 difference or a multiplicity-3 sum of at least six distinct elements. Therefore, since $f,g,h$ are interchangeable at this stage, we may assume $f$ appears in a loose difference. The following list gives every valid case that does not result in a multiplicity-4 difference, a multiplicity-3 sum of at least six distinct elements, or a previously explored case, WLOG: \begin{enumerate}
    \item $a-d=f-e$. Inputting the configuration $A=\{0,b+e,b,d+e,d,e,b+2e-d,g,h\}$, the program returns $\phi  =67$, $\gamma =47$, no multivariable configurations, and eight of the nine integer MSTD sets. When we check $\beta \in X \setminus \{5,7,11\}$, we get no multivariable configurations. For $\beta =5,7,11$, we get exactly the MSTD configurations given in the theorem. 
    \item $a-d=f-b$. Inputting the configuration $A=\{0,b+e,b,d+e,d,e,2b+e-d,g,h\}$, the program returns $\phi =61$, $\gamma =43$, no multivariable configurations, and eight of the nine integer sets. When we check $\beta \in X \setminus \{5,11\}$, we get no multivariable configurations. For $\beta =5,11$, we get exactly the MSTD configurations given in the theorem.
    \item $a-d=f-g$. Inputting the configuration $A=\{0,b+e,b,d+e,d,e,b+e-d+g,g,h\}$, the program returns $\phi =61$, $\gamma = 47$, no multivariable configurations, and eight of the nine integer sets. When we check $\beta \in X \setminus \{5,11\}$, we get no multivariable configurations. For $\beta =5,11$, we get exactly the MSTD configurations given in the theorem. 
    \item $a-d=c-f$. Then $b+e-d=d+e-f$, so $f=2d-b$. Inputting the configuration $\{0,b+e, b,d+e,d,e,2d-b,g,h\}$, we get $\phi = 61$, $\gamma = 43$, no multivariable configurations, and eight of the nine integer sets. When we check $\beta \in X \setminus \{5,11\}$, we get no multivariable configurations. For $\beta =5,11$, we get exactly the MSTD configurations given in the theorem.
    \item $a-d =0-f$. Then $f=d-b-e$. Inputting the configuration $\{0,b+e, b,d+e,d,e,d-b-e,g,h\}$, we get $\phi =67$, $\gamma = 47$, no multivariable configurations, and eight of the nine integer sets. When we check $\beta \in X \setminus \{5,7,11\}$, we get no multivariable configurations. For $\beta =5,7,11$, we get exactly the MSTD configurations given in the theorem.
\end{enumerate}

\textbf{Case 4:} Neither Case 1, Case 2, Case 3, nor a multiplicity-4 difference occurs. We have 72 nonzero difference pairs, and $\sigma =|A+A| \le 45$. Therefore, $\delta \le 42$, so $\sigma \le 45-0.25(30)=37.5$. Hence, $\delta \le 34$, so at least $4\binom{3}{2}+30\binom{2}{2}=42$ difference collisions occur. Therefore, $\sigma \le 45-0.25(42)=34.5$, so $\delta \le 32$. Hence, at least $8\binom{3}{2}+24\binom{2}{2}=48$ difference collisions occur. Therefore, $\sigma \le 45-48(0.25)=33$, so $\delta \le 30$. Hence, at least $12\binom{3}{2}+18\binom{2}{2}=54$ difference collisions occurs so $\sigma \le 45-54(0.25)=31.5$, so $\delta \le 28$. Then we have at least $16\binom{3}{2}+12\binom{2}{2}=60$ difference collisions, in the case where there are sixteen multiplicity-3 differences and twelve differences of exact multiplicity-2. However, since Case 3 does not occur, at least one of the sixteen multiplicity-three differences (and its negative) must contain a $C_2$ collision. Since Case 2 does not occur, the corresponding $C_2$ sum collision belongs to an exact multiplicity-2 sum. Therefore, $\sigma \le 44-0.25(58)=29.5$, so $\delta \le 26$. As such, at least $20\binom{3}{2}+6\binom{2}{2}=66$ difference collisions occur, and because Case 3 does not occur, at least six of these twenty multiplicity-3 differences must contain a $C_2$ collision, and because Case 2 does not occur, each of the three corresponding $C_2$ sum collisions have exact multiplicity-2. Therefore, $\sigma \le 42-60(0.25)=27$, so $\delta \le 24$. Hence, every difference is multiplicity-3, so 72 difference collisions occur, and at least 10 of them must contain a $C_2$ collision, so because Case 2 and Case 3 do not occur, $\sigma \le 40-0.25(62)=24.5$. Therefore, $\delta \le 22$, so by pigeonhole principle, there exists at least one multiplicity-4 difference, so this case is impossible.

\textbf{Case 5:} A multiplicity-4 difference occurs. This is our final case to consider, and if we assume char$(F)\ne3$, there are five types to consider: \begin{enumerate}
    \item (4-chain) $a-b=b-c=c-d=d$. Inputting the configuration $A=\{0, 4d,3d,2d,d,e,f,g,h\}$, the program returns $\phi = 61$, $\gamma =41$, no multivariable configurations, and no integer sets. A check with $\beta \in X_{61}$ confirms no multivariable configuration exists. 
    \item (3-chain, 1 loose difference) $a-b=b-c=c-d=e$. Inputting the configuration $A=\{0, d+3e,d+2e,d+e,d,e,f,g,h\}$, the program returns $\phi = 53$, $\gamma =43$, no multivariable configurations, and four of the nine integer MSTD sets. A check with $\beta \in X_{53}\setminus \{5, 11\}$ confirms no multivariable configuration exists. When $\beta = 11$ we get the multivariable configuration given in Theorem \ref{Size9}. When $\beta =5$, we get three of the multivariable configurations given in Theorem \ref{Size9}.
    \item (two 2-chains) $a-b=b-c=d-e=e-0$. Inputting the configuration $A= \{0, c+2e,c+e,c,2e,e,f,g,h\}$, we get $\phi =53$, $\gamma =37$, no multivariable configurations, and five of the nine integer MSTD sets. A check with $\beta \in X_{53}\setminus \{5, 11\}$ confirms no multivariable configuration exists. When $\beta = 11$ we get the multivariable configuration given in Theorem \ref{Size9}. When $\beta =5$, we get two of the multivariable configurations given in Theorem \ref{Size9}.
    \item (2-chain, 2 loose differences) $a-b=b-c=d-e=f$. Inputting the configuration $A=\{0, c+2f,c+f,c,e+f,e,f,g,h\}$, the program returns $\phi = 47$, $\gamma = 37$, no multivariable configurations, and five of the nine integer MSTD sets. A check with $\beta \in X_{47}\setminus \{5, 11\}$ confirms no multivariable configuration exists. When $\beta = 11$ we get the multivariable configuration given in Theorem \ref{Size9}. When $\beta =5$, we get the five multivariable configurations given in Theorem \ref{Size9}.
    \item (4 loose differences) $a-b=c-d=e-f=g$. Inputting the configuration $A= \{0, b+g, b, d+g,d, f+g,f,g,h\}$, the program returns $\phi = 41$, $\gamma = 37$, no multivariable configurations, and no integer MSTD sets. A check with $\beta \in X_{41}\setminus \{5, 11\}$ confirms no multivariable configuration exists. When $\beta = 11$ we get the multivariable configuration given in Theorem \ref{Size9}. When $\beta =5$, we get the five multivariable configurations given in Theorem \ref{Size9}.
\end{enumerate}

Across all five cases, $\gamma$ did not exceed 47. Given that our search gave us every possible multivariable MSTD configuration of size 9, we only need to check $\Z/p\Z$ for $5 \le p \le 47$ to find all remaining MSTD configurations that are not affine transformations of one of the integer sets in sufficiently large $p$. This computational check is done by directly testing if every nine-element subset of $\Z/p\Z$ is MSTD in $\Z/p\Z$ (note we can fix 0 in each subset). Program outputs for these computational checks are available in the appendix. In total, we found 19 unique MSTD configurations in $\Z/{23}\Z$, 58 in $\Z/{29}\Z$, 33 in $\Z/31\Z$, 8 in $\Z/37\Z$, 6 in $\Z/{41}\Z$, 2 in $\Z/{43}\Z$, and 1 in $\Z/{47}\Z$. Each of these sets is listed in an appendix. 

We have now found all MSTD sets of size 9 in arbitrary fields when char$(F)>3$. If char$(F)=3$, we input the base configuration with $\beta =3$, and get the multivariable configurations given in the theorem.

\end{proof}



\section{Multiplicative Subgroups as MSTD Sets}

When counting MSTD sets in $\mathbb{Z}/p\mathbb{Z}$ for small $p$, we observed that the total number was always divisible by $p(p-1)$. In the initial cases we examined, this occurred because every MSTD set $A \subset \mathbb{Z}/p\mathbb{Z}$ satisfied $m_1\cdot A+d_1 \neq m_2\cdot A+d_2$ for all $m_1, m_2 \in \Z/p\Z^*$ and $d_1,d_2 \in \Z/p\Z$ such that $(m_1,d_1) \ne (m_2,d_2)$. This naturally raises the question of whether the same holds for all MSTD sets $A \subset \mathbb{Z}/p\mathbb{Z}$ as $p$ grows arbitrarily large.

\begin{definition}
    An \textit{$m-$invariant} set $B \subset \Z/p\Z$ satisfies $m\cdot B=B$, where $m \in \Z/p\Z^*$ and $m \ne 1$.
\end{definition}

We recall a standard result in group theory that holds for any proper subset $A \subset \Z/p\Z$.

\begin{lemma}\label{SubgroupLmeema}
Let $A$ be a proper subset of $\Z/p\Z$. Then there exists $m_1, m_2 \in \Z/p\Z^*$ and $d_1,d_2 \in \Z/p\Z$ such that $m_1\cdot A+d_1=m_2\cdot A+d_2$ and $(m_1,d_1) \ne (m_2,d_2)$ if and only if there exists an affine transformation of $A$ which is $m$-invariant for some $m \in \Z/p\Z^*$.
\end{lemma}

\begin{proof}
    The backwards direction is trivial. For the forwards direction, let $m_1A+d_1=m_2A+d_2$. Then if $m= m_1^{-1}m_2$ and $d=m_1^{-1}(d_2-d_1)$, we get that $A= mA+d$. If $m=1$ and $d \ne 0$, then $d \in S(A)$, so $|S(A)|>1$. From Lemma \ref{StabSUBGROUP}, we have that $S(A) = \Z/p\Z$, so $A= \Z/p\Z$, which is a contradiction since $A$ is a proper subset of $\Z/p\Z$. If $m=1$ and $d=0$, then our initial collision is trivial. If $m \ne 1$, then the set $B= A-(1-m)^{-1}d$ satisfies $m\cdot B=B$.
\end{proof}

We now examine whether there exists an $m-$invariant MSTD set. 

\begin{corollary}
    The smallest MSTD set $A$ in $\Z/p\Z$ for which there exists nontrivial $m$ such that $A$ is $m-$invariant has cardinality 10.
\end{corollary}

\begin{proof}
    First, note that if $B$ is a multiplicative subgroup of even order, then $-1 \in B$, so $B=-B$. Hence, $B+B=B-B$, so $B$ is sum-difference balanced. Next, if $C$ is a union of cosets of a multiplicative subgroup, then $C+C$ and $C-C$ are also unions of cosets of the same multiplicative subgroup (including potentially 0). Therefore, if $A$ is a union of cosets of a subgroup of odd order $n$, then $|A-A| \equiv 1 \pmod {2n}$ since $0 \in A-A$ and since $A$, $-A$ are disjoint. Additionally, $A+A \equiv k \pmod {n}$ where $k \in \{0,1\}$, and $k$ depends on whether or not $0 \notin A+A$.
    
    Consider $A$ such that $|A|=7$. We can computationally check if $A=m\cdot A+d$ for any $m \in \Z/p\Z^*$, $d \in \Z/p\Z$, where $A$ is one of the MSTD configurations given that is not multivariable. This is then enough to verify whether an affine transformation of $A$ is $m-$invariant due to Lemma \ref{SubgroupLmeema}. For the two configurations in $\Z/17\Z$ and $\Z/19\Z$ in Theorem \ref{Size 7}, we computationally verify that neither is $m-$invariant. 
    

    Consider $A$ such that $|A|=8$. If $0 \notin A$, then $A$ is a union of cosets of even order and is sum-difference balanced. If $0 \in A$, then $A=C \cup \{0\}$, where $C$ is a coset of order 7. We know that $|C-C|\equiv 1 \pmod {14}$, and $A-A=C\cup -C\cup (C-C)$. Therefore, $|A-A| \equiv 1 \pmod {14}$ as well. As such, no affine transformation of Conway's set is ever $m-$invariant since if $A$ is Conway's set, then $|A-A|=25 \equiv 11 \mod 14$. We verify that none of the MSTD configurations in $\Z/p\Z$ in Theorem \ref{Size8} are $m-$invariant. 
    


   Now, consider an $m-$invariant set $A$ such that $|A|=9$. Recall that the smallest $p$ where $\Z/p\Z$ contains an MSTD set of size 9 is $p=23$. A computational check verifies that none of the sets MSTD in $\Z/p\Z$ for $23 \le p \le 47$ where $p$ is prime are $m-$invariant. It is still possible for one of the affine transformations of the nine integer sets of size 9 to be $m-$invariant in $\Z/p\Z$ for large enough $p$.

   If $|A|=9$ and $0 \in A$, then if we assume $A$ is $m-$invariant, we have that $A$ is a union of cosets of even order and $\{0\}$. Therefore, $A=-A$ and $A+A=A-A$. If $A$ is a single coset of order $9$, then $|A+A| \equiv 0 \pmod{9}$ and $|A-A| \equiv 1 \pmod{18}$. If $A$ is a union of three cosets of order 3, then $|A-A| \equiv1 \pmod{6}$, and $|A+A|\equiv 0 \pmod{3}$ or $|A+A|\equiv 1\pmod{3}$ if $0 \in A+A$. In either case, we must have $|A-A|\equiv 1\pmod{6}$. Of the nine MSTD sets of size 9 in $\Z$ only $\{0, 2, 3, 4, 7, 11, 15, 16, 18\}$, $\{0, 4, 6, 7, 8, 14, 15, 17, 21\}$, and $\{0, 5, 6, 9, 10, 13, 16, 17, 22\}$ satisfy this constraint, as they all have a difference set of cardinality 31. However, their sumsets all have cardinality 32, which is a contradiction since we have that $|A+A| \not \equiv 2 \pmod 3$. Thus, no MSTD set of size 9 satisfies $mA=A$ in $\Z/p\Z$.

    In $\Z/{43}\Z$, $C=\{1,6,36\}$ is a multiplicative subgroup of size 3, and $$A=\{0\} \cup C \cup 2\cdot C\cup 10\cdot C=\{0, 1, 2, 6, 10, 12, 16, 17, 29, 36\}$$ is MSTD in $\Z/{43}\Z$ and satisfies $A=6\cdot A$.  
\end{proof}

Although a union of cosets of a multiplicative subgroup can be MSTD at small cardinalities, one can computationally verify that for small primes $p<1000$ a multiplicative subgroup in $\Z/p\Z$ cannot be MSTD. This raises the question of whether a multiplicative subgroup can be MSTD at all. The sum-product phenomenon says that a set cannot be highly structured under both addition and multiplication at the same time ~\cite{Bourgain2004}. A multiplicative subgroup $A$ satisfies $A\cdot A = A$, so the sum-product phenomenon predicts that its sumset must be large. Indeed, using basic Fourier-analytic techniques, one can verify that a multiplicative subgroup $A$ of size at least $p^{3/4}$ in $\Z/p\Z$ satisfies $\Z/p\Z \subseteq A-A$, and therefore, $A$ is not MSTD. For smaller subgroups, however, these techniques are no longer strong enough to force $A-A$ to cover all of $\Z/p\Z$, which means a multiplicative subgroup could be MSTD. This motivates our computational search for such examples. 

We now prove Theorem \ref{SubgroupThm}.

\begin{proof}
    The first statement is easily verifiable with a direct computer search. The second statement is not as easily verifiable since we can have a multiplicative subgroup of odd order $n$ as long as $n |p-1$, and Dirichlet's Theorem on arithmetic progressions gives that there are infinitely many such $p$. As such, we need to narrow down which characteristics a sum collision can occur in. 
    
    Recall that for $f, g \in \Z/p\Z[x]$, the resultant $\operatorname{Res}(f, g)$ is congruent to zero modulo $p$ if and only if $f$ and $g$ share a common root in the algebraic closure of $\Z/p\Z$. In our setting, if $m \in \Z/p\Z$ is a generator of the subgroup $A$ of order $n$, taking the $n$th cyclotomic polynomial $\Phi_n$ and $g(x) = x^i + x^j - x^k - x^l$ gives $\operatorname{Res}(\Phi_n, g) \equiv 0 \pmod{p}$ exactly when a sum collision of the form $m^i + m^j = m^k + m^l$ occurs in $A$. We know that such a sum collision must occur because $A$ is MSTD, and every element in $A$ can be expressed as $m^i$ for some $0 \le i \le n-1$. 

    We built a program that uses these facts to reduce the problem down to a finite search. For a given odd order $n$, the program runs through every possible sum collision, each of which generates a form of the function $g$. To speed up the program, we can assume that $i=0$, since $m^i+m^j= m^k+m^l$ gives $1+m^{j-i} = m^{k-i}+m^{l-i}$, which is also a sum collision in $A$ because $A$ is $m-$invariant. We can also scale terms in the collision by $m$ to ensure we do not check the same sum collision twice. For instance, if $1+m= m^2+m^4$ is a sum collision in a subgroup of order 5, then scaling each term by $m$ gives $1+m^3=m+m^2$, which is the exact same collision up to affine transformation.
    
    The program uses Python's SymPy module to calculate the resultant of $\Phi_n(x)$ and $g$ over complex numbers (i.e., if they share a complex root, the program will return 0). Next, if the resultant is nonzero, it calculates all prime factors $p$ of the resultant that satisfy $n |(p-1)$, and then calculates whether the multiplicative subgroup of order $n$ is MSTD in each possible $\Z/p\Z$. After running the program on all odd orders $n$ up to 141, which took roughly eighty minutes, we get the second statement in the theorem. The program did not find any instances where $\Phi_n$ and $g$ shared a complex root for $n \le 141$, which allowed the program to run smoothly. Detailed program outputs are available in an appendix. 
\end{proof}

Table~\ref{Table 1} below gives more examples of MSTD multiplicative subgroups in $\Z/p\Z$. Note that $n$ denotes the order, $p$ the prime characteristic, and $\alpha$ is the value such that $p^\alpha = n$.

\begingroup
\renewcommand{\arraystretch}{1.15}
\setlength{\tabcolsep}{6pt}
\begin{longtable}{|c c c || c c c || c c c|}
\caption{Examples of MSTD multiplicative subgroups}
\label{Table 1}\\
\hline
$n$ & $p$ & $\alpha$ & $n$ & $p$ & $\alpha$ & $n$ & $p$ & $\alpha$ \\
\hline
\endfirsthead
\multicolumn{9}{c}{\tablename\ \thetable\ -- continued from previous page}\\
\hline
$n$ & $p$ & $\alpha$ & $n$ & $p$ & $\alpha$ & $n$ & $p$ & $\alpha$ \\
\hline
\endhead
\hline
\multicolumn{9}{r}{\emph{continued on next page}}\\
\endfoot
\hline
\endlastfoot
141  & 4231   & 0.593 & 1197 & 186733 & 0.584 & 2713 & 748789  & 0.584 \\ \hline
161  & 3221   & 0.629 & 1527 & 238213 & 0.592 & 2841 & 1011397 & 0.575 \\ \hline
213  & 7669   & 0.599 & 1585 & 209221 & 0.601 & 2853 & 810253  & 0.585 \\ \hline
265  & 10601  & 0.602 & 1607 & 231409 & 0.598 & 2903 & 766393  & 0.588 \\ \hline
333  & 15319  & 0.603 & 1703 & 228203 & 0.603 & 2909 & 994879  & 0.578 \\ \hline
339  & 11527  & 0.623 & 1741 & 278561 & 0.595 & 2931 & 679993  & 0.594 \\ \hline
339  & 16273  & 0.601 & 1743 & 355573 & 0.584 & 2965 & 764971  & 0.590 \\ \hline
375  & 22501  & 0.591 & 1795 & 323101 & 0.591 & 3027 & 992857  & 0.580 \\ \hline
543  & 38011  & 0.597 & 1847 & 343543 & 0.590 & 3093 & 1033063 & 0.580 \\ \hline
571  & 39971  & 0.599 & 1939 & 372289 & 0.590 & 3225 & 1161001 & 0.579 \\ \hline
579  & 35899  & 0.607 & 2007 & 437527 & 0.585 & 3323 & 1242803 & 0.578 \\ \hline
645  & 46441  & 0.602 & 2073 & 468499 & 0.585 & 3325 & 1270151 & 0.577 \\ \hline
715  & 72931  & 0.587 & 2089 & 309173 & 0.605 & 3339 & 814717  & 0.596 \\ \hline
817  & 93139  & 0.586 & 2259 & 573787 & 0.582 & 3379 & 1115071 & 0.584 \\ \hline
831  & 81439  & 0.595 & 2345 & 497141 & 0.592 & 3453 & 1111867 & 0.585 \\ \hline
899  & 70123  & 0.610 & 2361 & 679969 & 0.578 & 3603 & 1405171 & 0.579 \\ \hline
957  & 91873  & 0.601 & 2387 & 663587 & 0.580 & 3669 & 1372207 & 0.581 \\ \hline
1011 & 151651 & 0.580 & 2519 & 559219 & 0.592 & 3701 & 1043683 & 0.593 \\ \hline
1147 & 135347 & 0.596 & 2679 & 476863 & 0.604 & 3705 & 1185601 & 0.588 \\ \hline
\end{longtable}
\endgroup

Through gathering data for $\Z/p\Z$ for primes up to $8{,}000{,}000$, we conjecture that there exist infinitely many pairs $(n,p)$ such that $\mathbb{Z}/p\mathbb{Z}$ contains an MSTD multiplicative subgroup of order $n$. This is supported by the observation that MSTD multiplicative subgroups continue to appear as $p$ grows. However, the proportion of primes $p$ for which $\Z/p\Z$ contains an MSTD multiplicative subgroup decreases as $p$ gets larger, among the primes we observed. Table \ref{Table 2} gives the proportion of primes in a certain range which contain an MSTD subgroup: 
\begin{table}[ht]
\centering
\renewcommand{\arraystretch}{1.15}
\setlength{\tabcolsep}{8pt}
\caption{Density of primes admitting an MSTD multiplicative subgroup (cumulative)}
\label{Table 2}
\begin{tabular}{|c|c|c|c|}
\hline
$m$ & Cumulative primes $\le m$ & With MSTD subgroup & Proportion \\
\hline
$10^{6}$              & 78498 & 50 & $6.370\times10^{-4}$ \\ \hline
$2\cdot 10^{6}$       & 148933 & 78 & $5.237\times10^{-4}$ \\ \hline
$3\cdot 10^{6}$       & 216816 & 100 & $4.612\times10^{-4}$ \\ \hline
$4\cdot 10^{6}$       & 283146 & 112 & $3.955\times10^{-4}$ \\ \hline
$5\cdot 10^{6}$      & 348513 & 129 & $3.702\times10^{-4}$ \\ \hline
$6\cdot 10^{6}$      & 412094 & 139 & $3.373\times10^{-4}$ \\ \hline
$7\cdot 10^{6}$       & 476648 & 148 & $3.105\times10^{-4}$ \\ \hline
$8\cdot 10^{6}$      & 543153 & 158 & $2.909\times10^{-4}$ \\ \hline
\end{tabular}
\end{table}

Another question of interest is the behavior of $\alpha$ as $p$ increases. We know that $\alpha < 0.75$, and for all MSTD subgroups found in primes greater than $476{,}863$, $\alpha$ remained within $(0.56, 0.6)$. However, further analysis of this trend lies outside the scope of this paper.

\subsection*{Acknowledgments}

The author thanks Nathan Kaplan for the many invaluable comments and suggestions during
the preparation of this work. 

\subsection*{Statement on the Use of Artificial Intelligence}

Artificial intelligence tools were used in a supporting capacity during the preparation of this work. Specifically, AI assistance was used to test the accompanying code for correctness, to improve the code's computational efficiency and readability, and to help compile and format the tables presented in the appendix. 

\end{document}


\maketitle

\section*{Section 2: Size-5 Proof}

\subsection*{No More Than Two Differences of Multiplicity 3, Case 1 Case-Check}
Assume $a-b=b-c=c-0$ and its negative are nonzero differences of multiplicity 3, and assume there exists another distinct nonzero difference of multiplicity 3. At present, $A-A$ contains the following 15 differences: \begin{enumerate}\begin{multicols}{2}
\item 0
\item $a-b=b-c=c$
\item $b-a=c-b=-c$
\item $a-c=b$
\item $c-a=-b$
\item $a$
\item $-a$
\item $-d$
\item $a-d$
\item $b-d$
\item $c-d$
\item $d$
\item $d-a$
\item $d-b$
\item $d-c$
    \end{multicols}
\end{enumerate}
Additionally, $A+A$ contains the following 12 sums:
\begin{enumerate}\begin{multicols}{2}
\item $2b=a+c$
\item $a=b+c$
\item $b=2c$
\item $0$
\item $c$
\item $2a$
\item $a+b$
\item $d$
\item $a+d$
\item $b+d$
\item $c+d$
\item $2d$
    \end{multicols}
\end{enumerate}
We first verify that $D4$ is not a difference of multiplicity 3. We consider all possible collisions involving $D4$ that do not immediately lead to contradictions due to the distinctness of $0,a,b,c$, and $d$.\begin{itemize}
    \item $D4 = D7$. Then $a-c=b-0=0-a$, so $D5=D6$, $S5=S6, S4=S7$, and $|A+A|=10$ while $|A-A|=13$. Any difference collision must eliminate no more than 1 sum, and at least 5 differences, or else we will get $|A+A| \le 8$. However, the next difference collision occurs between $D8-D15$, each of which includes $d$. Therefore, any difference collision will form a $C_1$ or $C_2$ collision of the form $2d= x_1+x_2$ for $x_1,x_2 \in A$. Since $2d$ currently is not in a sum collision, no multiplicity-4 sum can be formed, and even if $2d$ is a multiplicity-3 sum, then at most 4 differences are removed since any sum collision involving $2d$ removes no more than 2 differences. Therefore, this case is impossible.
    \item $D4=D8$. Then $a-c=b=-d$, so $D5=D12$, $D6= D11$, $D7=D15, S5=S9$, and $S4=S10$. At present, $|A-A|=11$ and $|A+A|=10$. If we collide $D4,D5$ again we will create a multiplicity-4 difference. Next, $D6=D11$ can only be in a collision with $D13$, or $D14$, since $D9$ and $D10$ lead to distinctness contradictions. If $D13$, then $a=c-d=d-a$, so $S6=S8$, $S1=S12$, and then $|A+A|\le8$. If $D14$, then $a=c-d=d-b$, so $D5=D9=D12$ is a multiplicity-4 difference. This also means $D7=D15$ is not in another collision. Hence, a collision occurs between $D9, D10, D13,$ or $D14$. Any such collision will either be a $C_2$ or a $C_1$ collision, and would force $|A+A|=9$ while $|A-A| \ge 9$, a contradiction as previously established.
    \item $D4 = D11$. Then $a-c=b-0=c-d$, so $D2=D8$, a contradiction since no multiplicity-4 difference exists.
    \item $D4 =D13$. This is invalid due to the same logic as $D4=D8$.
    \item $D4 = D14$. Then $a-c=b=d-b$, so $D3=D9$, a contradiction since no difference can have multiplicity 4.
\end{itemize}
Hence, neither $D4$ nor $D5$ is of multiplicity 3, so the second multiplicity-3 difference must be formed from $D6-$15. Since $d$ is in $D8-$15, we know that $D6$ must then be of multiplicity 3. We now consider all possible collisions involving $D6$ that do not immediately lead to contradictions due to the distinctness of $0,a,b,c$, and $d$.
\begin{itemize}
    \item $D6=D8 = D13$, giving $a=-d=d-a$, so $D7=D9=D12$, $S4 = S9$, $S6=S8$, $S2=S12$. Since $S2$ is multiplicity-3 now, this also gives $D10=D15$ and $D11=D14$. As such, $|A-A|=9$ while $|A+A|=9$, a contradiction since this forces $|A+A|\le 8$.
    \item $D6 = D8 = D14$, giving $a=-d=d-b$, so $D4 = D13$, a contradiction.
    \item $D6= D8 = D15$, giving $a=-d=d-c$, so $D2 = D13$, a contradiction.
    \item $D6 = D10$ is impossible since then $D2 = D8$.
    \item $D6 = D11$ is impossible since it forces $D4 = D8$.
\end{itemize}
Therefore, $D6$ is not a multiplicity-3 difference, so no other multiplicity-3 difference can occur in this case.

\subsection*{No Multiplicity-3 Difference Occurs}
We now know that $|A+A|\ge 11$. First, note that a multiplicity-3 difference in the form of $a-b=b-c=d$ yields $|A+A|=12$ and $|A-A|=15$. The next collision will remove at least one sum, so it must also remove at least five differences, since $|A+A| \ge 11$. This implies a multiplicity-3 sum must be formed, since a sum collision between two sums of exact multiplicity-1 eliminates at most 4 differences. Recall that we have the following differences:
\begin{enumerate}\begin{multicols}{2}
\item 0
\item $a-b=b-c=c$
\item $b-a=c-b=-c$
\item $a-c=b$
\item $c-a=-b$
\item $a$
\item $-a$
\item $-d$
\item $a-d$
\item $b-d$
\item $c-d$
\item $d$
\item $d-a$
\item $d-b$
\item $d-c$
\end{multicols}
\end{enumerate}
and the following sums:
\begin{enumerate}\begin{multicols}{2}
\item $2b=a+c$
\item $a=b+c$
\item $b=2c$
\item $0$
\item $c$
\item $2a$
\item $a+b$
\item $d$
\item $a+d$
\item $b+d$
\item $c+d$
\item $2d$
    \end{multicols}
\end{enumerate}
We know then that either $S1,S2$, or $S3$ must be involved in the next collision. We consider all cases which do not lead to distinctness contradictions: \begin{itemize}
    \item $S1 = S4$ forces $D3=D4$.
    \item $S1 = S8$ forces $D4 = D14$.
    \item $S1 = S12$ eliminates no more than three differences.
    \item $S2 = S12$ eliminates no more than four differences.
    \item $S3=S6$ removes no more than three differences.
    \item $S3 = S9$ forces $D3=D12$.
    \item $S3 = S12$ removes no more than three differences.
\end{itemize}
Therefore, no multiplicity-3 difference in the form of a 3-chain can exist. If a multiplicity-3 difference in the form of a 3-cycle exists, then recall that we have the following 12 sums:
\begin{enumerate}\begin{multicols}{2}
\item $2b=a$
\item $a+b=0$
\item $b=2a$
\item $c$
\item $d$
\item $2c$
\item $2d$
\item $a+c$
\item $b+c$
\item $a+d$
\item $b+d$
\item $c+d$
    \end{multicols}
\end{enumerate}
We know any collision must eliminate at least seven differences (since there are currently 17 differences), which forces a multiplicity-3 sum to occur. However, every multiplicity-3 sum that can be formed will have at least one $C_2$ collision, so it cannot eliminate at least seven differences.
Next, assume a multiplicity-3 difference with a 2-chain and a loose difference exists, in the form of $a-b=b-c=d-0$. At present, $A-A$ contains the following 13 differences:
\begin{enumerate}
    \begin{multicols}{2}
    \item $a-b = b-c = d$
    \item $b-a = c-b = -d$
    \item $a-d = b$
    \item $d-a = -b$
    \item $b-d=c$
    \item $d-b=-c$
    \item $a-c$
    \item $c-a$
    \item $a$
    \item $-a$
    \item $c-d$
    \item $d-c$
    \item 0
     \end{multicols}
\end{enumerate}
Additionally, $A+A$ contains the following 12 sums:
\begin{enumerate}
    \begin{multicols}{2}
        \item $2b = a+c$
        \item $a=b+d$
        \item $b=c+d$
        \item $2a$
        \item $2c$
        \item $2d$
        \item $0$
        \item $a+b$
        \item $a+d$
        \item $b+c$
        \item $c$
        \item $d$
    \end{multicols}
\end{enumerate}
We must have a difference collision between $D7-D12$, since no more than two multiplicity-3 differences occur. We consider all possible collisions that do not immediately lead to contradictions due to the distinctness of $0,a,b,c$, and $d$.
\begin{itemize}
\item $D7=D8$ gives $S4=S5$, so $|A+A|\le 11$ while $|A-A| \ge 11$, which forces $|A+A| \le 10$.
\item $D7 = D10$ forces $S4=S11$ and $D8=D9$, so $|A+A|\le 11$ while $|A-A| \ge 11$, which forces $|A+A| \le 10$.
\item $D7 = D11$ forces $S5=S9$ and $D8=D12$, so $|A+A|\le 11$ while $|A-A| \ge 11$, which forces $|A+A| \le 10$.
\item $D9 =D10$ also gives $|A+A|\le 11$ while $|A-A| \ge 11$.
\item $D9 = D11$ forces $D2=D7$.
\item $D9 = D12$ forces $D3=D6$.
\item $D11=D12$ gives $|A+A|\le 11$ while $|A-A| \ge 11$.
\end{itemize}
Hence, no multiplicity-3 difference can occur.

\section*{Section 3: Program Output Logs}
We provide detailed output logs for every instance where the program was used. Duplicate Terminations/Duplicates counts the number of branches that terminated early because a difference collision forced two elements in the configuration to be equal to each other. Dedup-skipped/Skipped counts the number of branches that terminated early because they resulted in configurations that the program had already fully checked through. The max reduction factor tracks the largest $k$ such that the program reduces a difference collision by a factor of $k$. 

\subsection*{Theorem 1.2}
\subsubsection*{No MSTD Sets of Size 6}
In the case where no specific characteristic constraint is imposed, with input configuration of size 6, we get the following results:
\begin{runsummary}
    \item[Inputted Configuration:] $A = \{0, a, b, c, d, e\}$
    \item[Runtime:] 1.2 s
    \item[Branches Processed:] 25{,}515
    \item[Duplicate Terminations:] 4{,}827
    \item[Dedup-skipped:] 19{,}282
    \item[Max Reduction Factor:] 15
    \item[Max Prime Dividing Any Reduction:] 13
    \item[Total MSTD Leaves Found:] 0
    \item[Integer-MSTD Leaves Found:] 0
    \item[Multi-var MSTD Configurations Found:] 0
\end{runsummary}
We now provide the following search statistics for each characteristic checked. Note that the program terminates a branch early when a difference collision forces two elements of a configuration to be duplicates. The number of times this occurs during a run is tracked in the Branches column. Furthermore, the program uses a cache to ensure that no configuration is checked through twice by the program. If a configuration arises that is already in the cache, the program will skip this configuration. The number of times this occurs is given in the Skipped column.
\begin{table}[H]
\centering
\setlength{\tabcolsep}{6pt}
\caption{Search statistics for size 6 MSTD sets}
\label{tab:size6_stats}
\begin{tabular}{|c|c|c|c|c|c|}
\hline
$\beta$ & Runtime (s) & Branches & Duplicates & Skipped & MSTD Found  \\
\hline
3  & 0.1 & 464     & 213   & 191   & 0 \\
\hline
5  & 1.7 & 6,180   & 2,734 & 2,614 & 0 \\
\hline
7  & 1.5 & 13,051  & 6,423 & 5,088 & 0 \\
\hline
11 & 2.0 & 23,857  & 8,492 & 13,262 & 0 \\
\hline
13 & 1.8 & 28,589  & 8,755 & 17,445 & 0 \\
\hline
\textbf{Total} & \textbf{7.1} & \textbf{72,141} & \textbf{26,617} & \textbf{38,600} & \textbf{0} \\
\hline
\end{tabular}
\end{table}
In all cases, the inputted configuration was $A = \{0, a, b, c, d, e\}$. No MSTD sets of size 6 were found in any characteristic checked, confirming Theorem 1.2.

\subsection*{Theorem 1.3}
\subsubsection*{MSTD Sets of Size 7}
In the case where no specific characteristic constraint is imposed, with input configuration of size 7, we get the following results:
\begin{runsummary}
    \item[Inputted Configuration:] $A = \{0, a, b, c, d, e, f\}$
    \item[Runtime:] 112.9 s
    \item[Branches Processed:] 1{,}441{,}790
    \item[Duplicate Terminations:] 271{,}246
    \item[Dedup-skipped:] 1{,}123{,}706
    \item[Max Reduction Factor:] 31
    \item[Max Prime Dividing Any Reduction:] 31
    \item[Total MSTD Leaves Found:] 140
    \item[Integer-MSTD Leaves Found:] 0
    \item[Multi-var MSTD Configurations Found:] 0
    \item[MSTD by Prime:] $\Zp{17}$: 90 configs; \quad $\Zp{19}$: 50 configs.
\end{runsummary}
After applying affine equivalence, the configurations reduce to the following unique MSTD sets:
\begin{itemize}[nosep,leftmargin=1.5em]
    \item $\Zp{17}$: $\{0, 1, 2, 3, 5, 6, 14\}$
    \item $\Zp{19}$: $\{0, 1, 2, 3, 8, 10, 13\}$
\end{itemize}
These correspond to the configurations in $\Z/p\Z$ given in Theorem 1.3. We now provide the following search statistics for each characteristic checked. The inputted configuration was $A = \{0, a, b, c, d, e, f\}$.
\begin{table}[H]
\centering
\setlength{\tabcolsep}{4pt}
\caption{Search statistics for size 7 MSTD sets}
\label{tab:size7_stats}
\begin{tabular}{|c|c|c|c|c|c|c|}
\hline
$\beta$ & Runtime (s) & Branches & Duplicates & Skipped & MSTD Found & Multi-var \\
\hline
3  & 0.5   & 4,669     & 2,331   & 1,940   & 0 & 0 \\
\hline
5  & 11.8  & 132,198   & 44,490  & 79,660  & 0 & 0 \\
\hline
7  & 48.1  & 408,965   & 189,145 & 184,298 & 0 & 30 \\
\hline
11 & 114.1 & 920,852   & 439,949 & 414,829 & 0 & 0 \\
\hline
13 & 114.3 & 1,097,710 & 473,128 & 554,729 & 0 & 0 \\
\hline
17 & 107.6 & 1,479,840 & 520,434 & 879,992 & 42 & 0 \\
\hline
19 & 108.6 & 1,625,469 & 517,812 & 1,025,095 & 42 & 0 \\
\hline
23 & 117.5 & 1,914,319 & 516,178 & 1,302,445 & 0 & 0 \\
\hline
29 & 142.9 & 2,368,577 & 530,761 & 1,702,670 & 0 & 0 \\
\hline
31 & 162.3 & 2,487,117 & 533,855 & 1,803,482 & 0 & 0 \\
\hline
\textbf{Total} & \textbf{927.7} & \textbf{12,439,527} & \textbf{4,168,083} & \textbf{7,949,140} & \textbf{84} & \textbf{30} \\
\hline
\end{tabular}
\end{table}
For $\beta = 7$, the multi-variable configuration found is affinely equivalent to $\{0, d, d+e, 3d, 4d+e, 6d+e, e\}$, which corresponds to the configuration given in Theorem 1.3.
For $\beta = 17$, the unique MSTD set (up to affine equivalence) is $\{0, 1, 2, 3, 5, 6, 14\}$.
For $\beta = 19$, the unique MSTD set (up to affine equivalence) is $\{0, 1, 2, 3, 8, 10, 13\}$.
These correspond to the configurations given in Theorem 1.3.

\subsection*{Theorem 1.4}
\subsubsection*{Mult-3 Sum: $a+b=c+d=e+0$}
In the case where a multiplicity-3 sum of the form $a+b=c+d=e+0$ occurs, with no specific characteristic constraint, we get the following results:
\begin{runsummary}
    \item[Inputted Configuration:] $A = \{0, e-b, b, e-d, d, e, f, g\}$
    \item[Runtime:] 167.7 s
    \item[Branches Processed:] 1{,}824{,}322
    \item[Duplicate Terminations:] 381{,}379
    \item[Dedup-skipped:] 1{,}379{,}454
    \item[Max Reduction Factor:] 36
    \item[Max Prime Dividing Any Reduction:] 31
    \item[Total MSTD Leaves Found:] 494
    \item[Integer-MSTD Leaves Found:] 2 (includes negatives)
    \item[Multi-var MSTD Configurations Found:] 0
    \item[MSTD by Prime:] $\Zp{19}$: 220; \quad $\Zp{23}$: 236; \quad $\Zp{29}$: 24; \quad $\Zp{31}$: 14.
\end{runsummary}
The two integer-MSTD leaves correspond to Conway's set and its negative. We now provide the following search statistics for each characteristic checked:
\begin{table}[H]
\centering
\setlength{\tabcolsep}{4pt}
\caption{Search statistics for size 8, Form (ii): $a+b=c+d=e+0$}
\label{tab:size8_formii_stats}
\begin{tabular}{|c|c|c|c|c|c|c|}
\hline
$\beta$ & Runtime (s) & Branches & Duplicates & Skipped & MSTD Found & Multi-var \\
\hline
3  & 0.5   & 2,951     & 1,969   & 688     & 0 & 0 \\
\hline
5  & 7.1   & 65,859    & 29,888  & 30,483  & 0 & 25 \\
\hline
7  & 29.4  & 214,550   & 115,928 & 77,853  & 0 & 0 \\
\hline
11 & 58.0  & 475,948   & 271,919 & 168,366 & 0 & 0 \\
\hline
13 & 63.1  & 576,681   & 300,546 & 237,867 & 0 & 0 \\
\hline
17 & 76.3  & 750,793   & 317,317 & 390,734 & 0 & 0 \\
\hline
19 & 84.7  & 837,498   & 321,892 & 469,208 & 84 & 0 \\
\hline
23 & 84.5  & 1,003,843 & 325,897 & 619,714 & 192 & 0 \\
\hline
29 & 127.1 & 1,229,772 & 327,851 & 811,718 & 56 & 0 \\
\hline
31 & 112.9 & 1,291,827 & 327,556 & 861,337 & 36 & 0 \\
\hline
\textbf{Total} & \textbf{643.6} & \textbf{6,449,722} & \textbf{2,340,763} & \textbf{3,667,968} & \textbf{368} & \textbf{25} \\
\hline
\end{tabular}
\end{table}
For $\beta = 5$, the multi-variable configuration found is $\{0, d, d+g, 3d+2g, 4d+2g, 4d+3g, g, 3g\}$, which corresponds to the configuration given in Theorem 1.4.

\subsubsection*{Mult-3 Difference, 3-Chain}
In the case where a multiplicity-3 difference of the form $a-b=b-c=c-0$ occurs, with no specific characteristic constraint, we get the following results:
\begin{runsummary}
    \item[Inputted Configuration:] $A = \{0, 3c, 2c, c, d, e, f, g\}$
    \item[Runtime:] 178.4 s
    \item[Branches Processed:] 1{,}694{,}743
    \item[Duplicate Terminations:] 306{,}842
    \item[Dedup-skipped:] 1{,}272{,}863
    \item[Max Reduction Factor:] 47
    \item[Max Prime Dividing Any Reduction:] 47
    \item[Total MSTD Leaves Found:] 954
    \item[Integer-MSTD Leaves Found:] 0
    \item[Multi-var MSTD Configurations Found:] 0
    \item[MSTD by Prime:] $\Zp{19}$: 534; \quad $\Zp{23}$: 343; \quad $\Zp{29}$: 30; \quad $\Zp{31}$: 47.
\end{runsummary}
We now give a table of the $\beta$ cases for this configuration:
\begin{table}[H]
\centering
\setlength{\tabcolsep}{4pt}
\caption{Search statistics for size 8, Mult-3 Difference, 3-Chain}
\label{tab:size8_3chain_stats}
\begin{tabular}{|c|c|c|c|c|c|c|}
\hline
$\beta$ & Runtime (s) & Branches & Duplicates & Skipped & MSTD Found & Multi-var \\
\hline
7  & 21.8  & 135,992   & 82,946  & 37,925  & 0 & 0 \\
\hline
11 & 64.4  & 466,978   & 310,445 & 117,309 & 0 & 0 \\
\hline
13 & 80.7  & 656,257   & 405,428 & 203,016 & 0 & 0 \\
\hline
17 & 96.2  & 942,872   & 477,561 & 409,336 & 0 & 0 \\
\hline
19 & 105.4 & 1,100,076 & 502,889 & 536,109 & 57 & 0 \\
\hline
23 & 124.9 & 1,406,522 & 533,922 & 797,924 & 79 & 0 \\
\hline
29 & 134.1 & 1,818,435 & 544,547 & 1,164,426 & 16 & 0 \\
\hline
31 & 140.1 & 1,958,767 & 548,079 & 1,283,634 & 16 & 0 \\
\hline
37 & 217.8 & 2,350,608 & 556,668 & 1,598,624 & 0 & 0 \\
\hline
41 & 216.1 & 2,564,172 & 559,573 & 1,758,374 & 0 & 0 \\
\hline
43 & 159.2 & 2,645,575 & 556,623 & 1,815,961 & 0 & 0 \\
\hline
47 & 175.6 & 2,810,104 & 560,504 & 1,932,171 & 0 & 0 \\
\hline
\textbf{Total} & \textbf{1,536.3} & \textbf{17,856,358} & \textbf{5,639,185} & \textbf{10,654,809} & \textbf{168} & \textbf{0} \\
\hline
\end{tabular}
\end{table}

\subsubsection*{Mult-3 Difference, 2-Chain, One Loose Difference}
In the case where a multiplicity-3 difference of the form $a-b=b-c=d-0$ occurs, with no specific characteristic constraint, we get the following results:
\begin{runsummary}
    \item[Inputted Configuration:] $A = \{0,c+2d,c+d,c,d,e,f,g\}$
    \item[Runtime:] 312.0 s
    \item[Branches Processed:] 3{,}680{,}017
    \item[Duplicate Terminations:] 709{,}592
    \item[Dedup-skipped:] 2{,}754{,}685
    \item[Max Reduction Factor:] 42
    \item[Max Prime Dividing Any Reduction:] 41
    \item[Total MSTD Leaves Found:] 2{,}102
    \item[Integer-MSTD Leaves Found:] 2 (includes negatives)
    \item[Multi-var MSTD Configurations Found:] 0
    \item[MSTD by Prime:] $\Zp{19}$: 909; \quad $\Zp{23}$: 1{,}079; \quad $\Zp{29}$: 74; \quad $\Zp{31}$: 40.
\end{runsummary}
The two integer-MSTD leaves correspond to Conway's set and its negative. We provide the following search statistics for each characteristic checked:
\begin{table}[H]
\centering
\setlength{\tabcolsep}{4pt}
\caption{Search statistics for size 8, Mult-3 Difference, Two-Chain + One Loose Difference}
\label{tab:size8_2chain_stats}
\begin{tabular}{|c|c|c|c|c|c|c|}
\hline
$\beta$ & Runtime (s) & Branches & Duplicates & Skipped & MSTD Found & Multi-var \\
\hline
7  & 59.4  & 344,228   & 215,399 & 89,304  & 0 & 0 \\
\hline
11 & 161.6 & 1,119,458 & 756,498 & 267,487 & 0 & 0 \\
\hline
13 & 180.0 & 1,427,681 & 889,722 & 432,819 & 0 & 0 \\
\hline
17 & 200.1 & 2,044,003 & 1,036,989 & 885,853 & 0 & 0 \\
\hline
19 & 224.0 & 2,324,808 & 1,063,011 & 1,132,798 & 81 & 0 \\
\hline
23 & 236.5 & 2,875,322 & 1,089,139 & 1,631,437 & 252 & 0 \\
\hline
29 & 271.6 & 3,728,508 & 1,122,129 & 2,369,808 & 48 & 0 \\
\hline
31 & 278.2 & 3,968,872 & 1,119,216 & 2,573,965 & 34 & 0 \\
\hline
37 & 303.4 & 4,665,352 & 1,126,455 & 3,116,936 & 21 & 0 \\
\hline
41 & 308.0 & 5,012,712 & 1,125,889 & 3,367,315 & 21 & 0 \\
\hline
\textbf{Total} & \textbf{2,222.8} & \textbf{27,510,944} & \textbf{9,544,447} & \textbf{15,867,722} & \textbf{457} & \textbf{0} \\
\hline
\end{tabular}
\end{table}
For $\beta = 37$ and $\beta = 41$, the MSTD sets found are affinely equivalent to Conway's set.

\subsubsection*{Mult-3 Difference, Three Loose Differences}
In the case where a multiplicity-3 difference of the form $a-b=c-d=e-0$ occurs, with no specific characteristic constraint, we get the following results:
\begin{runsummary}
    \item[Inputted Configuration:] $A = \{0, a, a-e, c, c-e, e, f, g\}$ with $a-b=c-d=e-0$
    \item[Runtime:] 336.8 s
    \item[Branches Processed:] 4{,}190{,}438
    \item[Duplicate Terminations:] 892{,}443
    \item[Dedup-skipped:] 3{,}137{,}070
    \item[Max Reduction Factor:] 36
    \item[Max Prime Dividing Any Reduction:] 31
    \item[Total MSTD Leaves Found:] 1{,}212
    \item[Integer-MSTD Leaves Found:] 2 (includes negatives)
    \item[Multi-var MSTD Configurations Found:] 0
    \item[MSTD by Prime:] $\Zp{19}$: 548; \quad $\Zp{23}$: 568; \quad $\Zp{29}$: 50; \quad $\Zp{31}$: 46.
\end{runsummary}
The two integer-MSTD leaves correspond to Conway's set and its negative. We now provide the following search statistics for each characteristic checked:
\begin{table}[H]
\centering
\setlength{\tabcolsep}{4pt}
\caption{Search statistics for size 8, Mult-3 Difference, Three Loose Differences}
\label{tab:size8_3loose_stats}
\begin{tabular}{|c|c|c|c|c|c|c|}
\hline
$\beta$ & Runtime (s) & Branches & Duplicates & Skipped & MSTD Found & Multi-var \\
\hline
7  & 73.9  & 472,444   & 281,236 & 140,417 & 0 & 0 \\
\hline
11 & 155.2 & 1,290,462 & 847,556 & 336,840 & 0 & 0 \\
\hline
13 & 176.5 & 1,619,994 & 985,223 & 518,934 & 0 & 0 \\
\hline
17 & 209.3 & 2,197,299 & 1,088,553 & 980,863 & 0 & 0 \\
\hline
19 & 228.2 & 2,468,632 & 1,104,241 & 1,228,515 & 108 & 0 \\
\hline
23 & 247.4 & 3,012,678 & 1,122,518 & 1,724,106 & 318 & 0 \\
\hline
29 & 265.5 & 3,785,060 & 1,132,522 & 2,390,083 & 73 & 0 \\
\hline
31 & 302.4 & 4,019,864 & 1,133,966 & 2,578,364 & 54 & 0 \\
\hline
\textbf{Total} & \textbf{1,658.4} & \textbf{18,866,433} & \textbf{7,695,815} & \textbf{9,898,122} & \textbf{553} & \textbf{0} \\
\hline
\end{tabular}
\end{table}

\subsubsection*{The case where $\beta = 3$}
In the case where the ambient field has characteristic 3, we get the following results:
\begin{runsummary}
    \item[Inputted Configuration:] $A = \{0, a, b, c, d, e, f, g\}$
    \item[Runtime:] 5.9 s
    \item[Branches Processed:] 63{,}693
    \item[Duplicate Terminations:] 27{,}225
    \item[Dedup-skipped:] 33{,}303
    \item[Max Reduction Factor:] 1
    \item[Max Prime Dividing Any Reduction:] 1
    \item[Total MSTD Leaves Found:] 0
    \item[Multi-var MSTD Configurations Found:] 16
    \item[MSTD by Prime:] $\Zp{3}$: 0 configs.
\end{runsummary}
The multi-variable configuration found is $\{0, c, c+e, c+2e+g, c+2e+2g, e, e+g, g\}$, which is affinely equivalent to the configuration given in Theorem 1.4.

\subsubsection*{The case where $\beta = 5$}
In the case where the ambient field has characteristic 5, we get the following results:
\begin{runsummary}
    \item[Inputted Configuration:] $A = \{0, a, b, c, d, e, f, g\}$
    \item[Runtime:] 447.3 s
    \item[Branches Processed:] 5{,}204{,}498
    \item[Duplicate Terminations:] 2{,}083{,}033
    \item[Dedup-skipped:] 2{,}929{,}977
    \item[Max Reduction Factor:] 1
    \item[Max Prime Dividing Any Reduction:] 1
    \item[Total MSTD Leaves Found:] 0
    \item[Multi-var MSTD Configurations Found:] 39
    \item[MSTD by Prime:] $\Zp{5}$: 0 configs.
\end{runsummary}
The multi-variable configuration found is $\{0, f, f+g, 2f+3g, 2f+4g, 3f, 3f+4g, g\}$, which is affinely equivalent to the configuration given in Theorem 1.4.

\subsection*{Size-8 MSTD Sets in $\Z/p\Z$ and Affine Equivalence to Conway's Set}
For each prime $p$, we list every size-8 MSTD configuration in $\Z/p\Z$ up to affine transformation. A configuration is flagged when it is affinely equivalent to Conway's set $\{0, 2, 3, 4, 7, 11, 12, 14\}$, the unique size-8 MSTD set in $\Z$ up to affine transformation.

\subsubsection*{$p = 19$ \normalfont(2 configurations; 1 affinely equivalent to Conway's set)}
\begin{itemize}[nosep,leftmargin=1.5em]
\item $\{0, 1, 2, 3, 4, 6, 7, 16\}$
\item $\{0, 1, 2, 3, 4, 7, 8, 15\}$\;{\color{blue!55!black}$\longrightarrow$ affinely equivalent to Conway's set $\{0, 2, 3, 4, 7, 11, 12, 14\}$}
\end{itemize}
Neither set has the same size of the sumset and difference set as the Conway set does in the integers, since $2(14)>19$.

\subsubsection*{$p = 23$ \normalfont(8 configurations; 1 affinely equivalent to Conway's set)}
\begin{itemize}[nosep,leftmargin=1.5em]
\item $\{0, 1, 2, 3, 4, 7, 8, 19\}$
\item $\{0, 1, 2, 3, 4, 10, 12, 16\}$
\item $\{0, 1, 2, 3, 5, 6, 11, 15\}$
\item $\{0, 1, 2, 3, 5, 6, 12, 19\}$
\item $\{0, 1, 2, 3, 6, 8, 18, 19\}$
\item $\{0, 1, 2, 3, 9, 10, 12, 15\}$
\item $\{0, 1, 2, 4, 5, 7, 14, 18\}$
\item $\{0, 1, 2, 4, 10, 18, 19, 20\}$\;{\color{blue!55!black}$\longrightarrow$ affinely equivalent to Conway's set $\{0, 2, 3, 4, 7, 11, 12, 14\}$}
\end{itemize}
No set has the same size of the sumset and difference set as the Conway set does in the integers, since $2(14)>23$.

\subsubsection*{$p = 29$ \normalfont(3 configurations; 1 affinely equivalent to Conway's set)}
\begin{itemize}[nosep,leftmargin=1.5em]
\item $\{0, 1, 2, 3, 4, 12, 16, 20\}$
\item $\{0, 1, 2, 4, 19, 21, 22, 26\}$\;{\color{blue!55!black}$\longrightarrow$ affinely equivalent to Conway's set $\{0, 2, 3, 4, 7, 11, 12, 14\}$}
\item $\{0, 1, 2, 5, 18, 22, 23, 25\}$
\end{itemize}
Only the set $A=\{0, 1, 2, 4, 19, 21, 22, 26\}$ is affinely equivalent to an integer MSTD set since $2(14)<29$, so $|A+A|$ and $|A-A|$ are equivalent to the sizes of the sumset and difference set of the Conway set.

\subsubsection*{$p = 31$ \normalfont(2 configurations; 1 affinely equivalent to Conway's set)}
\begin{itemize}[nosep,leftmargin=1.5em]
\item $\{0, 1, 2, 3, 6, 7, 17, 27\}$
\item $\{0, 1, 2, 4, 21, 23, 24, 28\}$\;{\color{blue!55!black}$\longrightarrow$ affinely equivalent to Conway's set $\{0, 2, 3, 4, 7, 11, 12, 14\}$}
\end{itemize}
Only the set $A=\{0, 1, 2, 4, 21, 23, 24, 28\}$ is affinely equivalent to an integer MSTD set since $2(14)<31$, so $|A+A|$ and $|A-A|$ are equivalent to the sizes of the sumset and difference set of the Conway set.

\subsection*{Theorem 1.5}

\subsubsection*{Case 1: Mult-3 Sum $a+b=c+d=e$}
In the case where a multiplicity-3 sum of the form $a+b=c+d=e+0$ occurs, with no specific characteristic constraint, we get the following results:
\begin{runsummary}
    \item[Inputted Configuration:] $A = \{0, e-b, b, e-d, d, e, f, g, h\}$
    \item[Runtime:] 65{,}391.5 s (18h 09m 51.5s)
    \item[Branches Processed:] 334{,}684{,}055
    \item[Duplicate Terminations:] 65{,}073{,}769
    \item[Dedup-skipped:] 262{,}594{,}727
    \item[Max Reduction Factor:] 84
    \item[Max Prime Dividing Any Reduction:] 73
    \item[Total MSTD Leaves Found:] 27{,}316
    \item[Integer-MSTD Leaves Found:] 18
    \item[Multi-var MSTD Configurations Found:] 0
    \item[MSTD by Prime:] $\Zp{23}$: 15{,}750; \ $\Zp{29}$: 8{,}262; \ $\Zp{31}$: 2{,}804; \ $\Zp{37}$: 284; \ $\Zp{41}$: 210; \ $\Zp{43}$: 4; \ $\Zp{47}$: 2.
\end{runsummary}
We now provide the following search statistics for each characteristic checked:
\begin{table}[H]
\centering
\setlength{\tabcolsep}{4pt}
\caption{Search statistics for size 9, Form (ii): $a+b=c+d=e$}
\label{tab:size9_formii_stats}
\begin{tabular}{|c|c|c|c|c|c|c|}
\hline
$\beta$ & Runtime (s) & Branches & Duplicates & Skipped & MSTD Found & Multi-var \\
\hline
5 & 456.4 & 4,051,802 & 1,978,611 & 1,905,584 & 0 & 589 \\
\hline
7 & 2,081.8 & 14,676,387 & 6,142,117 & 7,702,485 & 0 & 70 \\
\hline
11 & 15,798.5 & 59,744,236 & 38,573,745 & 16,857,667 & 0 & 249 \\
\hline
13 & 17,401.3 & 79,106,668 & 51,254,806 & 22,735,739 & 0 & 0 \\
\hline
17 & 11,023.4 & 87,139,293 & 47,282,883 & 35,645,288 & 0 & 0 \\
\hline
19 & 23,604.1 & 125,197,757 & 67,570,740 & 51,623,320 & 0 & 0 \\
\hline
23 & 27,085.4 & 159,920,538 & 74,722,574 & 78,383,566 & 2,352 & 0 \\
\hline
29 & 23,604.1 & 146,870,740 & 53,466,723 & 88,341,058 & 2,978 & 0 \\
\hline
31 & 22,946.7 & 155,736,264 & 53,438,969 & 96,982,287 & 1,713 & 0 \\
\hline
37 & 24,343.5 & 183,436,460 & 53,824,313 & 122,936,075 & 725 & 0 \\
\hline
41 & 25,230.0 & 200,565,485 & 53,864,141 & 138,609,105 & 717 & 0 \\
\hline
43 & 16,943.7 & 210,140,777 & 54,300,270 & 146,867,913 & 561 & 0 \\
\hline
47 & 59,166.3 & 372,420,884 & 91,998,777 & 213,660,218 & 1,276 & 0 \\
\hline
53 & 58,571.9 & 378,225,257 & 86,086,893 & 212,361,986 & 1,250 & 0 \\
\hline
59 & 60,075.6 & 402,900,104 & 87,122,644 & 227,236,084 & 1,254 & 0 \\
\hline
61 & 61,877.2 & 422,677,173 & 90,362,610 & 227,770,301 & 1,265 & 0 \\
\hline
67 & 32,504.8 & 376,876,064 & 77,936,394 & 267,462,486 & 804 & 0 \\
\hline
71 & 39,364.5 & 350,449,455 & 71,287,701 & 248,602,582 & 800 & 0 \\
\hline
73 & 43,636.4 & 371,317,094 & 75,632,827 & 264,531,876 & 782 & 0 \\
\hline
\textbf{Total} & \textbf{565,715.6} & \textbf{4,101,452,438} & \textbf{1,186,848,378} & \textbf{2,470,215,620} & \textbf{16,477} & \textbf{908} \\
\hline
\end{tabular}
\end{table}
For $\beta = 5$, the multi-variable configurations found are:
\begin{itemize}[nosep,leftmargin=1.5em]
    \item $\{0, b, b+h, 3b+h, 3b+2h, 3b+4h, 4b+2h, h, 2h\}$
    \item $\{0, d, d+f, 2d, 2d+f, 3d+f, 4d+2f, 4d+3f, f\}$
    \item $\{0, d, 2d+2h, 2d+3h, 3d+2h, 4d+h, 4d+4h, h, 4h\}$
    \item $\{0, g, g+h, g+4h, 2g+h, 2g+3h, 4g+3h, h, 3h\}$
    \item $\{0, e, e+h, 2e+3h, 4e+2h, 4e+3h, h, 3h, 4h\}$
\end{itemize}
All five configurations are affinely equivalent to the multi-variable configurations given in Theorem 1.5.
For $\beta = 7$, the multi-variable configuration found is $\{0, f, 3f, 3f+6g, 4f+4g, 4f+5g, 6f+4g, g, 4g\}$, which is affinely equivalent to the configuration given in Theorem 1.5.
For $\beta = 11$, the one unique multi-variable configuration found is $\{0, e, 2e+9f, 4e+8f, 5e+6f, 6e+5f, 6e+6f, 7e+5f, f\}$, which is affinely equivalent to the configuration given in Theorem 1.5.
For $\beta = 53, 59, 61, 67, 71, 73$, all MSTD sets found are affinely equivalent to one of the 9 integer MSTD sets from Hegarty.

\subsubsection*{Case 2: $a=-b=c-d$}
In the case where the multiplicity-3 difference $a-b=c-d=e$ and $a=-b=c-d$ occurs, the program outputs the following when no $\beta$ is input.
\begin{runsummary}
    \item[Inputted Configuration:] $A = \{0, 2c, -2c, c, -c, e, f, g\}$
    \item[Runtime:] 690.8 s
    \item[Branches Processed:] 3{,}904{,}568
    \item[Duplicate Terminations:] 702{,}968
    \item[Dedup-skipped:] 2{,}905{,}009
    \item[Max Reduction Factor:] 63
    \item[Max Prime Dividing Any Reduction:] 61
    \item[Total MSTD Leaves Found:] 2{,}617
    \item[Integer-MSTD Leaves Found:] 0
    \item[Multi-var MSTD Configurations Found:] 0
    \item[MSTD by Prime:] $\Zp{23}$: 1{,}623; \ $\Zp{29}$: 733; \ $\Zp{31}$: 202; \ $\Zp{37}$: 24; \ $\Zp{41}$: 35.
\end{runsummary}
We give the search statistics now for this case.
\begin{table}[H]
\centering
\setlength{\tabcolsep}{4pt}
\caption{Search statistics for Case 2: $a=-b=c-d$}
\label{tab:size9_3chain_addl_stats}
\begin{tabular}{|c|c|c|c|c|c|c|}
\hline
$\beta$ & Runtime (s) & Branches & Duplicates & Skipped & MSTD Found & Multi-var \\
\hline
5 & 2.5 & 18,339 & 9,293 & 7,698 & 0 & 0 \\
\hline
7 & 22.1 & 105,129 & 60,109 & 34,772 & 0 & 0 \\
\hline
11 & 83.0 & 452,389 & 314,593 & 101,978 & 0 & 0 \\
\hline
13 & 110.3 & 622,185 & 420,865 & 158,473 & 0 & 0 \\
\hline
17 & 151.4 & 1,012,501 & 604,975 & 351,410 & 0 & 0 \\
\hline
19 & 163.0 & 1,184,384 & 653,578 & 470,387 & 0 & 0 \\
\hline
23 & 181.4 & 1,520,231 & 710,133 & 741,144 & 68 & 0 \\
\hline
29 & 214.4 & 2,069,445 & 772,656 & 1,206,757 & 60 & 0 \\
\hline
31 & 251.9 & 2,222,669 & 774,018 & 1,349,659 & 20 & 0 \\
\hline
37 & 282.2 & 2,726,197 & 794,984 & 1,790,679 & 4 & 0 \\
\hline
41 & 319.4 & 3,025,183 & 798,545 & 2,047,338 & 4 & 0 \\
\hline
43 & 321.7 & 3,153,895 & 794,412 & 2,158,081 & 0 & 0 \\
\hline
47 & 360.6 & 3,451,124 & 801,158 & 2,397,592 & 0 & 0 \\
\hline
53 & 376.8 & 3,838,202 & 812,204 & 2,693,770 & 0 & 0 \\
\hline
59 & 329.4 & 4,027,423 & 793,608 & 2,830,861 & 0 & 0 \\
\hline
61 & 347.2 & 4,199,271 & 813,990 & 2,963,626 & 0 & 0 \\
\hline
\textbf{Total} & \textbf{3,117.8} & \textbf{30,628,567} & \textbf{9,339,121} & \textbf{19,304,225} & \textbf{156} & \textbf{0} \\
\hline
\end{tabular}
\end{table}

\subsubsection*{Case 2: $a=-b=c-e$}
In the case where a multiplicity-3 difference of the form $a=-b=c-e$ occurs, we get the following results:
\begin{runsummary}
    \item[Inputted Configuration:] $A = \{0, c-e, e-c, c, -c, e, f, g, h\}$ with $a=-b=c-e$
    \item[Runtime:] 1{,}450.7 s
    \item[Branches Processed:] 10{,}979{,}679
    \item[Duplicate Terminations:] 1{,}993{,}838
    \item[Dedup-skipped:] 8{,}149{,}372
    \item[Max Reduction Factor:] 63
    \item[Max Prime Dividing Any Reduction:] 61
    \item[Total MSTD Leaves Found:] 12{,}233
    \item[Integer-MSTD Leaves Found:] 13
    \item[Multi-var MSTD Configurations Found:] 0
    \item[MSTD by Prime:] $\Zp{23}$: 6{,}698; \ $\Zp{29}$: 4{,}011; \ $\Zp{31}$: 1{,}401; \ $\Zp{37}$: 56; \ $\Zp{41}$: 66; \ $\Zp{43}$: 1.
\end{runsummary}
We provide the following search statistics for each characteristic checked:
\begin{table}[H]
\centering
\setlength{\tabcolsep}{4pt}
\caption{Search statistics for size 9, Case of $a=-b=c-e$}
\label{tab:size9_abce_stats}
\begin{tabular}{|c|c|c|c|c|c|c|}
\hline
$\beta$ & Runtime (s) & Branches & Duplicates & Skipped & MSTD Found & Multi-var \\
\hline
5 & 11.6 & 73,926 & 39,767 & 28,825 & 0 & 62 \\
\hline
7 & 67.4 & 353,587 & 198,790 & 121,662 & 0 & 0 \\
\hline
11 & 353.6 & 1,541,018 & 1,097,257 & 318,379 & 0 & 0 \\
\hline
13 & 417.6 & 2,064,783 & 1,436,959 & 481,386 & 0 & 0 \\
\hline
17 & 467.4 & 3,005,362 & 1,829,939 & 1,005,727 & 0 & 0 \\
\hline
19 & 486.0 & 3,442,274 & 1,929,729 & 1,334,307 & 0 & 0 \\
\hline
23 & 604.0 & 4,273,343 & 2,027,423 & 2,048,814 & 266 & 0 \\
\hline
29 & 676.1 & 5,573,297 & 2,112,745 & 3,207,749 & 364 & 0 \\
\hline
31 & 663.4 & 6,001,181 & 2,129,102 & 3,590,855 & 174 & 0 \\
\hline
37 & 708.8 & 7,234,473 & 2,151,386 & 4,679,104 & 36 & 0 \\
\hline
41 & 840.2 & 8,037,125 & 2,165,336 & 5,353,119 & 40 & 0 \\
\hline
43 & 850.2 & 8,403,512 & 2,165,891 & 5,654,607 & 28 & 0 \\
\hline
47 & 732.0 & 9,128,479 & 2,178,946 & 6,224,843 & 28 & 0 \\
\hline
53 & 757.0 & 9,998,071 & 2,181,440 & 6,879,754 & 28 & 0 \\
\hline
59 & 792.7 & 10,590,968 & 2,171,660 & 7,309,445 & 28 & 0 \\
\hline
61 & 823.1 & 10,780,808 & 2,177,253 & 7,436,823 & 28 & 0 \\
\hline
\textbf{Total} & \textbf{8,451.3} & \textbf{84,502,166} & \textbf{27,913,523} & \textbf{55,265,455} & \textbf{1,020} & \textbf{62} \\
\hline
\end{tabular}
\end{table}
For $\beta = 5$, the multi-variable configurations found are:
\begin{itemize}[nosep,leftmargin=1.5em]
    \item $\{0, c, c+2h, 3c+h, 4c, 4c+h, h, 2h, 3h\}$
    \item $\{0, g, g+2h, 2g, 4g, 4g+3h, h, 2h, 3h\}$
    \item $\{0, e, e+2h, e+3h, 3e, 4e+2h, h, 2h, 3h\}$
    \item $\{0, c, c+4e, 2c+e, 4c, 4c+e, 4c+4e, e, 3e\}$
\end{itemize}
All four configurations are affinely equivalent to the multi-variable configurations given in Theorem 1.5.
For $\beta \ge 53$, the MSTD sets found for each prime are affinely equivalent to the 9 integer MSTD sets from Hegarty.

\subsubsection*{Case 2: $a=-b=e-f$}
In the case where a multiplicity-3 difference of the form $a=-b=e-f$ occurs, we get the following results:
\begin{runsummary}
    \item[Inputted Configuration:] $A = \{0, a, -a, b, c, d, e, f\}$ with $a=-b=e-f$
    \item[Runtime:] 2{,}055.1 s
    \item[Branches Processed:] 14{,}338{,}416
    \item[Duplicate Terminations:] 2{,}606{,}407
    \item[Dedup-skipped:] 10{,}508{,}746
    \item[Max Reduction Factor:] 67
    \item[Max Prime Dividing Any Reduction:] 67
    \item[Total MSTD Leaves Found:] 18{,}718
    \item[Integer-MSTD Leaves Found:] 6
    \item[Multi-var MSTD Configurations Found:] 0
    \item[MSTD by Prime:] $\Zp{23}$: 10{,}603; \ $\Zp{29}$: 5{,}740; \ $\Zp{31}$: 2{,}119; \ $\Zp{37}$: 228; \ $\Zp{41}$: 28.
\end{runsummary}
We provide the following search statistics for each characteristic checked:
\begin{table}[H]
\centering
\setlength{\tabcolsep}{4pt}
\caption{Search statistics for size 8, Case of $a=-b=e-f$}
\label{tab:size8_abcef_stats}
\begin{tabular}{|c|c|c|c|c|c|c|}
\hline
$\beta$ & Runtime (s) & Branches & Duplicates & Skipped & MSTD Found & Multi-var \\
\hline
5 & 14.7 & 98,572 & 50,792 & 41,925 & 0 & 41 \\
\hline
7 & 87.9 & 476,230 & 251,084 & 184,834 & 0 & 0 \\
\hline
11 & 426.5 & 2,177,956 & 1,485,360 & 523,618 & 0 & 0 \\
\hline
13 & 523.5 & 2,896,781 & 1,943,375 & 755,855 & 0 & 0 \\
\hline
17 & 627.5 & 4,226,520 & 2,512,043 & 1,482,924 & 0 & 0 \\
\hline
19 & 681.1 & 4,826,542 & 2,652,159 & 1,931,865 & 0 & 0 \\
\hline
23 & 742.5 & 6,008,731 & 2,815,505 & 2,924,543 & 309 & 0 \\
\hline
29 & 860.2 & 7,781,089 & 2,935,120 & 4,505,727 & 462 & 0 \\
\hline
31 & 887.7 & 8,349,920 & 2,952,686 & 5,020,707 & 249 & 0 \\
\hline
37 & 1,010.8 & 10,020,980 & 2,984,930 & 6,502,081 & 62 & 0 \\
\hline
41 & 1,056.4 & 11,081,528 & 2,996,267 & 7,405,433 & 35 & 0 \\
\hline
43 & 1,054.3 & 11,590,926 & 3,001,026 & 7,828,731 & 27 & 0 \\
\hline
47 & 1,087.0 & 12,530,774 & 3,005,893 & 8,587,752 & 27 & 0 \\
\hline
53 & 1,213.6 & 13,684,604 & 2,998,705 & 9,490,481 & 26 & 0 \\
\hline
59 & 1,249.9 & 14,593,090 & 3,004,721 & 10,163,366 & 27 & 0 \\
\hline
61 & 1,248.5 & 14,819,860 & 3,004,111 & 10,332,062 & 26 & 0 \\
\hline
67 & 1,265.0 & 15,376,993 & 3,011,701 & 10,728,434 & 27 & 0 \\
\hline
\textbf{Total} & \textbf{13,996.4} & \textbf{140,570,720} & \textbf{42,605,858} & \textbf{88,410,358} & \textbf{1,277} & \textbf{41} \\
\hline
\end{tabular}
\end{table}
For $\beta = 5$, the multi-variable configurations found are:
\begin{itemize}[nosep,leftmargin=1.5em]
    \item $\{0, f, 2f, 2f+2h, 2f+3h, 3f, 3f+2h, 4f+2h, h\}$
    \item $\{0, f, 2f, 2f+3h, 3f, 3f+2h, 3f+3h, h, 2h\}$
    \item $\{0, f, f+2h, 2f, 2f+3h, 3f+2h, 4f+2h, 4f+3h, h\}$
\end{itemize}
All three configurations are affinely equivalent to the multi-variable configurations given in Theorem 1.5.
For $\beta \ge 53$, the MSTD sets found for each prime are affinely equivalent to Conway's set.

\subsubsection*{Case 3: $a-b=c-d=e$, and $a-c=b-d=f-g$}
In the case where a multiplicity-3 difference occurs with relations of the form $a-b=c-d=e$ and $a-c=b-d=f-g$, we get the following results:
\begin{runsummary}
    \item[Inputted Configuration:] $A = \{0, b+e, b, d+e, d, e, b-d+g, g,h\}$
    \item[Runtime:] 4{,}346.8 s (1h 12m 26.8s)
    \item[Branches Processed:] 34{,}349{,}487
    \item[Duplicate Terminations:] 7{,}430{,}145
    \item[Dedup-skipped:] 25{,}697{,}777
    \item[Max Reduction Factor:] 56
    \item[Max Prime Dividing Any Reduction:] 47
    \item[Total MSTD Leaves Found:] 9{,}862
    \item[Integer-MSTD Leaves Found:] 6
    \item[Multi-var MSTD Configurations Found:] 0
    \item[MSTD by Prime:] $\Zp{23}$: 6{,}020; \ $\Zp{29}$: 2{,}724; \ $\Zp{31}$: 1{,}060; \ $\Zp{37}$: 58.
\end{runsummary}
We provide the following search statistics for each characteristic checked:
\begin{table}[H]
\centering
\setlength{\tabcolsep}{4pt}
\caption{Search statistics for Case 3, $a-c=b-d=f-g$}
\label{tab:size8_2loose_shared_stats}
\begin{tabular}{|c|c|c|c|c|c|c|}
\hline
$\beta$ & Runtime (s) & Branches & Duplicates & Skipped & MSTD Found & Multi-var \\
\hline
5  & 62.2   & 435,654   & 242,338 & 164,530 & 0 & 351 \\
\hline
7  & 386.4  & 2,000,924 & 1,166,709 & 642,658 & 0 & 0 \\
\hline
11 & 1,411.4 & 7,824,000 & 5,717,579 & 1,461,155 & 0 & 0 \\
\hline
13 & 1,827.6 & 10,360,658 & 7,444,767 & 2,171,453 & 0 & 0 \\
\hline
17 & 2,194.5 & 14,874,129 & 9,388,070 & 4,641,271 & 0 & 0 \\
\hline
19 & 2,686.9 & 16,887,041 & 9,821,383 & 6,188,721 & 0 & 0 \\
\hline
23 & 3,307.4 & 20,720,301 & 10,200,792 & 9,562,714 & 1,043 & 0 \\
\hline
29 & 3,437.2 & 26,389,685 & 10,395,834 & 14,769,503 & 1,382 & 0 \\
\hline
31 & 3,707.6 & 28,245,918 & 10,426,042 & 16,448,836 & 792 & 0 \\
\hline
37 & 3,567.1 & 33,574,683 & 10,466,421 & 21,109,654 & 166 & 0 \\
\hline
41 & 4,135.7 & 36,783,021 & 10,465,572 & 23,780,992 & 125 & 0 \\
\hline
43 & 4,077.8 & 38,255,295 & 10,463,974 & 24,970,846 & 125 & 0 \\
\hline
47 & 3,428.5 & 40,891,103 & 10,467,015 & 27,045,478 & 125 & 0 \\
\hline
\textbf{Total} & \textbf{34,230.3} & \textbf{277,242,412} & \textbf{106,666,496} & \textbf{152,957,811} & \textbf{3,758} & \textbf{351} \\
\hline
\end{tabular}
\end{table}
For $\beta = 5$, the multi-variable configurations found (up to affine scaling) are:
\begin{itemize}[nosep,leftmargin=1.5em]
    \item $\{0, e, 2e, 2e+3g, 3e, 3e+3g, 4e+2g, g, 2g\}$
    \item $\{0, e, e+2h, e+3h, 2e+h, 3e+h, 4e+4h, h, 2h\}$
    \item $\{0, b, 2b+3g, 3b, 3b+3g, 4b+2g, 4b+3g, g, 2g\}$
    \item $\{0, e, e+3g, e+4g, 2e+g, 2e+4g, 4e, g, 3g\}$
    \item $\{0, d, d+3g, d+4g, 2d+g, 2d+3g, 3d, g, 3g\}$
\end{itemize}

\subsubsection*{Case 3: $a-b=c-d=e$ and $a-d=f-e$}
In the case where a multiplicity-3 difference occurs with relations $a-b=c-d=e$ and $a-d=f-e$, we get the following results:
\begin{runsummary}
    \item[Inputted Configuration:] $A = \{0, b+e, b, d+e, d, e, b+2e-d,g,h\}$ with $a-b=c-d=e$ and $a-d=f-e$
    \item[Runtime:] 4{,}392.6 s (1h 13m 12.6s)
    \item[Branches Processed:] 35{,}331{,}568
    \item[Duplicate Terminations:] 7{,}010{,}146
    \item[Dedup-skipped:] 26{,}170{,}285
    \item[Max Reduction Factor:] 69
    \item[Max Prime Dividing Any Reduction:] 67
    \item[Total MSTD Leaves Found:] 30{,}802
    \item[Integer-MSTD Leaves Found:] 16
    \item[Multi-var MSTD Configurations Found:] 0
    \item[MSTD by Prime:] $\Zp{23}$: 17{,}229; \ $\Zp{29}$: 9{,}626; \ $\Zp{31}$: 3{,}396; \ $\Zp{37}$: 434; \ $\Zp{41}$: 111; \ $\Zp{43}$: 4; \ $\Zp{47}$: 2.
\end{runsummary}
We provide the following search statistics for each characteristic checked:
\begin{table}[H]
\centering
\setlength{\tabcolsep}{4pt}
\caption{Search statistics for size 9, Case of $a-b=c-d=e$ and $a-d=f-e$}
\label{tab:size7_abcde_stats}
\begin{tabular}{|c|c|c|c|c|c|c|}
\hline
$\beta$ & Runtime (s) & Branches & Duplicates & Skipped & MSTD Found & Multi-var \\
\hline
5 & 34.4 & 247,074 & 143,663 & 86,251 & 0 & 288 \\
\hline
7 & 250.5 & 1,204,750 & 705,808 & 383,991 & 0 & 10 \\
\hline
11 & 1,090.5 & 5,633,473 & 4,189,757 & 974,588 & 0 & 11 \\
\hline
13 & 1,735.6 & 7,662,075 & 5,651,173 & 1,455,981 & 0 & 0 \\
\hline
17 & 2,015.2 & 11,218,181 & 7,414,648 & 3,161,293 & 0 & 0 \\
\hline
19 & 2,520.9 & 12,841,801 & 7,870,353 & 4,299,818 & 0 & 0 \\
\hline
23 & 2,770.9 & 15,960,025 & 8,352,479 & 6,863,909 & 993 & 0 \\
\hline
29 & 2,284.7 & 20,504,216 & 8,606,126 & 10,935,946 & 1,671 & 0 \\
\hline
31 & 2,356.1 & 22,010,282 & 8,644,385 & 12,284,608 & 869 & 0 \\
\hline
37 & 2,734.4 & 26,485,088 & 8,715,400 & 16,153,645 & 234 & 0 \\
\hline
41 & 2,861.2 & 29,365,968 & 8,747,264 & 18,500,381 & 170 & 0 \\
\hline
43 & 3,337.1 & 30,736,228 & 8,756,820 & 19,573,818 & 140 & 0 \\
\hline
47 & 3,457.8 & 33,292,300 & 8,771,695 & 21,496,426 & 139 & 0 \\
\hline
53 & 2,999.0 & 36,535,781 & 8,784,703 & 23,793,585 & 140 & 0 \\
\hline
59 & 3,123.9 & 38,938,930 & 8,786,179 & 25,398,984 & 140 & 0 \\
\hline
61 & 3,624.2 & 39,562,219 & 8,785,656 & 25,801,867 & 140 & 0 \\
\hline
67 & 3,716.0 & 41,005,912 & 8,789,370 & 26,718,364 & 140 & 0 \\
\hline
\textbf{Total} & \textbf{38,912.4} & \textbf{381,204,303} & \textbf{127,915,279} & \textbf{236,383,455} & \textbf{5,636} & \textbf{309} \\
\hline
\end{tabular}
\end{table}
For $\beta = 5$, the multi-variable configurations found are:
\begin{itemize}[nosep,leftmargin=1.5em]
    \item $\{0, e, e+2g, e+3g, 3e, 4e, 4e+2g, g, 2g\}$
    \item $\{0, d, d+3g, 2d+2g, 2d+3g, 3d, 3d+2g, g, 3g\}$
    \item $\{0, e, e+2g, 3e+4g, 4e+3g, 4e+4g, g, 2g, 3g\}$
    \item $\{0, g, g+h, g+3h, 3g+2h, 3g+4h, 4g+2h, 4g+4h, h\}$
    \item $\{0, g, g+3h, g+4h, 3g+3h, 4g, 4g+h, h, 3h\}$
\end{itemize}
For $\beta = 7$, the multi-variable configuration found is $\{0, d, d+e, 3d+4e, 4d+2e, 4d+3e, 4d+6e, 5d+4e, e\}$.
For $\beta = 11$, the multi-variable configuration found is $\{0, e, 2e+9f, 4e+8f, 5e+6f, 6e+5f, 6e+6f, 7e+5f, f\}$.
All multi-variable configurations for $\beta = 5, 7, 11$ are affinely equivalent to the multi-variable configurations given in Theorem 1.5.
For $\beta \ge 53$, the MSTD sets found for each prime are affinely equivalent to the 9 integer MSTD sets from Hegarty.

\subsubsection*{Case 3: $a-d=f-b$}
In the case where a multiplicity-3 difference of the form $a-d=f-b$ occurs, we input $A=\{0,b+e,b,d+e,d,e,2b+e-d,g,h\}$ and get the following results:
\begin{runsummary}
    \item[Inputted Configuration:] $A = \{0,b+e,b,d+e,d,e,2b+e-d,g,h\}$
    \item[Runtime:] 5{,}504.6 s (1h 31m 44.6s)
    \item[Branches Processed:] 32{,}869{,}201
    \item[Duplicate Terminations:] 6{,}354{,}296
    \item[Dedup-skipped:] 23{,}809{,}753
    \item[Max Reduction Factor:] 66
    \item[Max Prime Dividing Any Reduction:] 61
    \item[Total MSTD Leaves Found:] 32{,}705
    \item[Integer-MSTD Leaves Found:] 15
    \item[Multi-var MSTD Configurations Found:] 0
    \item[MSTD by Prime:] $\Zp{23}$: 20{,}264; \ $\Zp{29}$: 9{,}081; \ $\Zp{31}$: 2{,}937; \ $\Zp{37}$: 333; \ $\Zp{41}$: 88; \ $\Zp{43}$: 2.
\end{runsummary}

\begin{table}[H]
\centering
\setlength{\tabcolsep}{4pt}
\caption{Search statistics for size 9, Case 4: $a-d=f-b$}
\label{tab:size9_case4_stats}
\begin{tabular}{|c|c|c|c|c|c|c|}
\hline
$\beta$ & Runtime (s) & Branches & Duplicates & Skipped & MSTD Found & Multi-var \\
\hline
5 & 34.3 & 225,576 & 126,299 & 83,215 & 0 & 170 \\
\hline
7 & 225.3 & 1,099,369 & 642,017 & 350,763 & 0 & 0 \\
\hline
11 & 947.0 & 4,832,381 & 3,523,691 & 909,948 & 0 & 3 \\
\hline
13 & 1,311.5 & 6,460,753 & 4,622,752 & 1,374,854 & 0 & 0 \\
\hline
17 & 1,542.6 & 9,426,649 & 5,956,462 & 2,933,839 & 0 & 0 \\
\hline
19 & 1,798.3 & 10,807,152 & 6,304,308 & 3,938,930 & 0 & 0 \\
\hline
23 & 1,954.6 & 13,483,899 & 6,682,189 & 6,167,068 & 899 & 0 \\
\hline
29 & 1,996.0 & 17,374,078 & 6,888,648 & 9,639,672 & 1,125 & 0 \\
\hline
31 & 2,054.3 & 18,649,052 & 6,919,615 & 10,770,614 & 523 & 0 \\
\hline
37 & 2,052.4 & 22,435,942 & 6,986,027 & 13,991,771 & 115 & 0 \\
\hline
41 & 2,157.6 & 24,825,870 & 7,010,457 & 15,896,177 & 75 & 0 \\
\hline
43 & 2,323.9 & 25,952,028 & 7,018,331 & 16,755,915 & 62 & 0 \\
\hline
47 & 2,418.9 & 28,024,470 & 7,028,535 & 18,268,882 & 63 & 0 \\
\hline
53 & 3,254.6 & 30,584,905 & 7,035,524 & 20,006,965 & 63 & 0 \\
\hline
59 & 2,965.8 & 32,439,166 & 7,038,876 & 21,193,794 & 63 & 0 \\
\hline
61 & 3,008.9 & 32,903,494 & 7,038,361 & 21,473,848 & 63 & 0 \\
\hline
\textbf{Total} & \textbf{30,045.0} & \textbf{283,824,779} & \textbf{91,822,093} & \textbf{164,755,855} & \textbf{3,051} & \textbf{173} \\
\hline
\end{tabular}
\end{table}
For $\beta = 5$, the multi-variable configurations found are:
\begin{itemize}[nosep,leftmargin=1.5em]
    \item $\{0, g, g+4h, 2g+4h, 3g+h, 3g+3h, 3g+4h, 4g+3h, h\}$
    \item $\{0, d, 2d+h, 2d+4h, 3d+2h, 4d+h, 4d+2h, h, 2h\}$
    \item $\{0, g, g+h, g+4h, 3g+h, 4g, 4g+4h, h, 3h\}$
    \item $\{0, e, e+4h, 2e, 2e+3h, 3e+h, 4e+h, h, 4h\}$
    \item $\{0, b, b+e, 3b+4e, 4b+2e, 4b+3e, 4b+4e, e, 3e\}$
\end{itemize}
For $\beta = 11$, we found the multi-variable configuration $\{0, g, g+h, 3g+h, 5g+h, 7g, 9g, 10g+h, h\}$.
All multi-variable configurations for $\beta = 5$ and $\beta = 11$ are affinely equivalent to the multi-variable configurations given in Theorem 1.5.
For $\beta \ge 53$, the MSTD sets found for each prime are affinely equivalent to the 9 integer MSTD sets from Hegarty.

\subsubsection*{Case 3: $a-d=f-g$}
In the case where a multiplicity-3 difference of the form $a-d=f-g$ occurs, we get the following results:
\begin{runsummary}
    \item[Inputted Configuration:] $A = \{0, b+e, b, d+e, d, e, b+e-d+g, g, h\}$
    \item[Runtime:] 10{,}258.3 s (2h 50m 58.3s)
    \item[Branches Processed:] 63{,}471{,}724
    \item[Duplicate Terminations:] 13{,}281{,}067
    \item[Dedup-skipped:] 46{,}748{,}401
    \item[Max Reduction Factor:] 66
    \item[Max Prime Dividing Any Reduction:] 61
    \item[Total MSTD Leaves Found:] 36{,}105
    \item[Integer-MSTD Leaves Found:] 16 (eight of the nine integer MSTD sets, including negatives)
    \item[Multi-var MSTD Configurations Found:] 0
    \item[MSTD by Prime:] $\Zp{23}$: 20{,}348; \ $\Zp{29}$: 11{,}399; \ $\Zp{31}$: 4{,}095; \ $\Zp{37}$: 197; \ $\Zp{41}$: 62; \ $\Zp{43}$: 2; \ $\Zp{47}$: 2.
\end{runsummary}
We provide the following search statistics for each characteristic checked:
\begin{table}[H]
\centering
\setlength{\tabcolsep}{4pt}
\caption{Search statistics for size 9, Case 3: $a-d=f-g$}
\label{tab:size9_case3_stats}
\begin{tabular}{|c|c|c|c|c|c|c|}
\hline
$\beta$ & Runtime (s) & Branches & Duplicates & Skipped & MSTD Found & Multi-var \\
\hline
5 & 69.9 & 487,700 & 279,676 & 176,299 & 0 & 365 \\
\hline
7 & 452.6 & 2,357,471 & 1,380,330 & 754,222 & 0 & 0 \\
\hline
11 & 1,920.5 & 10,350,148 & 7,688,356 & 1,800,564 & 0 & 24 \\
\hline
13 & 2,767.6 & 13,960,363 & 10,264,082 & 2,686,073 & 0 & 0 \\
\hline
17 & 3,298.6 & 20,378,679 & 13,376,967 & 5,838,792 & 0 & 0 \\
\hline
19 & 4,311.8 & 23,317,154 & 14,175,959 & 7,930,371 & 0 & 0 \\
\hline
23 & 4,651.2 & 28,839,008 & 14,942,055 & 12,578,606 & 1,120 & 0 \\
\hline
29 & 4,319.5 & 36,836,003 & 15,311,596 & 19,851,240 & 2,099 & 0 \\
\hline
31 & 4,439.7 & 39,458,297 & 15,361,250 & 22,222,871 & 1,143 & 0 \\
\hline
37 & 5,175.3 & 47,188,075 & 15,451,659 & 28,954,954 & 221 & 0 \\
\hline
41 & 5,406.4 & 52,061,762 & 15,479,774 & 32,961,727 & 183 & 0 \\
\hline
43 & 5,386.3 & 54,361,414 & 15,488,203 & 34,782,925 & 171 & 0 \\
\hline
47 & 5,650.1 & 58,593,675 & 15,496,336 & 38,022,877 & 171 & 0 \\
\hline
53 & 6,799.3 & 63,804,783 & 15,500,070 & 41,819,359 & 171 & 0 \\
\hline
59 & 6,494.7 & 67,527,770 & 15,502,465 & 44,412,481 & 171 & 0 \\
\hline
61 & 6,802.4 & 68,457,422 & 15,502,850 & 45,044,602 & 171 & 0 \\
\hline
\textbf{Total} & \textbf{68,946.3} & \textbf{607,579,724} & \textbf{230,293,233} & \textbf{360,184,173} & \textbf{5,621} & \textbf{389} \\
\hline
\end{tabular}
\end{table}
For $\beta = 5$, the multi-variable configurations found are:
\begin{itemize}[nosep,leftmargin=1.5em]
    \item $\{0, d, d+e, 2d, 2d+4e, 3d+e, 3d+2e, 4d+2e, e\}$
    \item $\{0, d, d+2g, d+4g, 2d+g, 2d+2g, 4d+2g, g, 4g\}$
    \item $\{0, d, d+e, 2d+e, 3d+e, 3d+2e, 4d+2e, 4d+3e, e\}$
    \item $\{0, e, e+h, e+3h, e+4h, 2e+h, 3e, h, 4h\}$
    \item $\{0, g, g+h, g+4h, 3g+2h, 4g+2h, 4g+3h, h, 3h\}$
\end{itemize}
For $\beta = 11$, the multi-variable configuration found is $\{0, e, 2e+9h, 4e+8h, 5e+6h, 6e+5h, 6e+6h, 7e+5h, h\}$.
All multi-variable configurations for $\beta = 5$ and $\beta = 11$ are affinely equivalent to the multi-variable configurations given in Theorem 1.5.
For $\beta \ge 53$, the MSTD sets found for each prime are affinely equivalent to the 9 integer MSTD sets from Hegarty.

\subsubsection*{Case 3: $a-d=c-f$}
In the case where a multiplicity-3 difference of the form $a-d=c-f$ occurs, we get the following results:
\begin{runsummary}
    \item[Inputted Configuration:] $A = \{0, b+e, b, d+e, d, e, 2d-b, g, h\}$
    \item[Runtime:] 5{,}619.6 s (1h 33m 39.6s)
    \item[Branches Processed:] 33{,}592{,}344
    \item[Duplicate Terminations:] 6{,}516{,}456
    \item[Dedup-skipped:] 24{,}358{,}166
    \item[Max Reduction Factor:] 66
    \item[Max Prime Dividing Any Reduction:] 61
    \item[Total MSTD Leaves Found:] 33{,}012
    \item[Integer-MSTD Leaves Found:] 15
    \item[Multi-var MSTD Configurations Found:] 0
    \item[MSTD by Prime:] $\Zp{23}$: 20{,}476; \ $\Zp{29}$: 9{,}124; \ $\Zp{31}$: 3{,}003; \ $\Zp{37}$: 326; \ $\Zp{41}$: 81; \ $\Zp{43}$: 2.
\end{runsummary}
We provide the following search statistics for each characteristic checked:
\begin{table}[H]
\centering
\setlength{\tabcolsep}{4pt}
\caption{Search statistics for size 9, Case 3: $a-d=c-f$}
\label{tab:size9_case3_cf_stats}
\begin{tabular}{|c|c|c|c|c|c|c|}
\hline
$\beta$ & Runtime (s) & Branches & Duplicates & Skipped & MSTD Found & Multi-var \\
\hline
5 & 34.8 & 232,701 & 131,696 & 83,851 & 0 & 193 \\
\hline
7 & 251.8 & 1,167,540 & 692,976 & 358,738 & 0 & 0 \\
\hline
11 & 1,043.7 & 5,228,558 & 3,852,712 & 939,603 & 0 & 7 \\
\hline
13 & 1,420.9 & 6,993,165 & 5,050,679 & 1,436,355 & 0 & 0 \\
\hline
17 & 1,663.3 & 10,180,658 & 6,474,153 & 3,123,675 & 0 & 0 \\
\hline
19 & 1,792.9 & 11,673,659 & 6,847,287 & 4,213,864 & 0 & 0 \\
\hline
23 & 1,955.9 & 14,572,558 & 7,251,045 & 6,631,851 & 947 & 0 \\
\hline
29 & 2,189.7 & 18,768,232 & 7,464,137 & 10,386,795 & 1,146 & 0 \\
\hline
31 & 2,272.5 & 20,151,693 & 7,498,908 & 11,613,610 & 500 & 0 \\
\hline
37 & 2,504.1 & 24,233,858 & 7,566,084 & 15,092,931 & 112 & 0 \\
\hline
41 & 2,638.3 & 26,809,001 & 7,591,099 & 17,152,796 & 73 & 0 \\
\hline
43 & 2,669.2 & 28,024,463 & 7,600,028 & 18,081,659 & 60 & 0 \\
\hline
47 & 2,795.9 & 30,248,815 & 7,609,599 & 19,717,899 & 60 & 0 \\
\hline
53 & 2,906.4 & 32,993,239 & 7,617,125 & 21,596,380 & 60 & 0 \\
\hline
59 & 3,656.8 & 34,952,952 & 7,618,186 & 22,873,718 & 60 & 0 \\
\hline
61 & 3,094.4 & 35,440,170 & 7,617,043 & 23,170,124 & 60 & 0 \\
\hline
\textbf{Total} & \textbf{32,890.5} & \textbf{303,571,262} & \textbf{102,882,757} & \textbf{184,473,869} & \textbf{3,078} & \textbf{200} \\
\hline
\end{tabular}
\end{table}
For $\beta = 5$, the multi-variable configurations found are:
\begin{itemize}[nosep,leftmargin=1.5em]
    \item $\{0, g, g+3h, 2g+h, 2g+4h, 4g, 4g+4h, h, 3h\}$
    \item $\{0, e, e+g, e+2g, e+3g, 2e+g, 2e+2g, 3e+3g, g\}$
    \item $\{0, d, d+e, d+2e, 2d+3e, 2d+4e, 3d+4e, e, 2e\}$
    \item $\{0, e, 2e+2h, 3e+3h, 4e+h, 4e+3h, 4e+4h, h, 4h\}$
    \item $\{0, d, 2d+h, 2d+3h, 2d+4h, 3d+2h, 3d+4h, h, 2h\}$
\end{itemize}
For $\beta = 11$, the multi-variable configuration found is $\{0, g, g+h, g+5h, g+9h, h, 3h, 7h, 10h\}$.
All multi-variable configurations for $\beta = 5$ and $\beta = 11$ are affinely equivalent to the multi-variable configurations given in Theorem 1.5.
For $\beta \ge 53$, the MSTD sets found for each prime are affinely equivalent to the 9 integer MSTD sets from Hegarty.

\subsubsection*{Case 3: $a-d=0-f$}
In the case where $a-b=c-d=e$ and $a-d=0-f$ occurs, we get the following results:
\begin{runsummary}
    \item[Inputted Configuration:] $A = \{0, b+e, b, d+e, d, e, d-b-e, g, h\}$ with $a-d=0-f$
    \item[Runtime:] 6{,}240.9 s (1h 44m 00.9s)
    \item[Branches Processed:] 35{,}851{,}394
    \item[Duplicate Terminations:] 7{,}108{,}785
    \item[Dedup-skipped:] 26{,}546{,}325
    \item[Max Reduction Factor:] 69
    \item[Max Prime Dividing Any Reduction:] 67
    \item[Total MSTD Leaves Found:] 31{,}391
    \item[Integer-MSTD Leaves Found:] 16
    \item[Multi-var MSTD Configurations Found:] 0
    \item[MSTD by Prime:] $\Zp{23}$: 17{,}545; \ $\Zp{29}$: 9{,}836; \ $\Zp{31}$: 3{,}457; \ $\Zp{37}$: 431; \ $\Zp{41}$: 116; \ $\Zp{43}$: 4; \ $\Zp{47}$: 2.
\end{runsummary}
We provide the following search statistics for each characteristic checked:
\begin{table}[H]
\centering
\setlength{\tabcolsep}{4pt}
\caption{Search statistics for size 9, Case 3: $a-d=0-f$}
\label{tab:size9_case3_0f_stats}
\begin{tabular}{|c|c|c|c|c|c|c|}
\hline
$\beta$ & Runtime (s) & Branches & Duplicates & Skipped & MSTD Found & Multi-var \\
\hline
5 & 37.6 & 246,045 & 142,625 & 86,493 & 0 & 285 \\
\hline
7 & 231.1 & 1,191,929 & 696,642 & 382,038 & 0 & 8 \\
\hline
11 & 1,180.0 & 5,459,594 & 4,044,075 & 962,686 & 0 & 11 \\
\hline
13 & 1,416.9 & 7,386,082 & 5,422,964 & 1,430,339 & 0 & 0 \\
\hline
17 & 1,788.7 & 10,797,756 & 7,104,797 & 3,076,585 & 0 & 0 \\
\hline
19 & 1,870.0 & 12,360,134 & 7,540,772 & 4,174,752 & 0 & 0 \\
\hline
23 & 2,239.9 & 15,360,940 & 7,999,978 & 6,648,946 & 916 & 0 \\
\hline
29 & 2,439.9 & 19,722,615 & 8,240,866 & 10,568,428 & 1,481 & 0 \\
\hline
31 & 2,564.1 & 21,163,906 & 8,276,641 & 11,863,931 & 756 & 0 \\
\hline
37 & 2,776.4 & 25,445,288 & 8,344,736 & 15,581,399 & 194 & 0 \\
\hline
41 & 2,748.3 & 28,181,434 & 8,371,319 & 17,821,687 & 139 & 0 \\
\hline
43 & 2,810.2 & 29,493,464 & 8,383,059 & 18,853,228 & 119 & 0 \\
\hline
47 & 3,437.9 & 31,908,827 & 8,393,933 & 20,678,832 & 119 & 0 \\
\hline
53 & 3,567.4 & 34,966,767 & 8,404,316 & 22,854,622 & 119 & 0 \\
\hline
59 & 3,381.0 & 37,219,109 & 8,405,542 & 24,354,484 & 119 & 0 \\
\hline
61 & 4,001.9 & 37,805,608 & 8,406,302 & 24,731,691 & 119 & 0 \\
\hline
67 & 3,617.9 & 39,140,233 & 8,408,457 & 25,579,687 & 119 & 0 \\
\hline
\textbf{Total} & \textbf{40,109.2} & \textbf{340,099,526} & \textbf{110,197,089} & \textbf{224,231,003} & \textbf{4,176} & \textbf{304} \\
\hline
\end{tabular}
\end{table}
For $\beta = 5$, the multi-variable configurations found are:
\begin{itemize}[nosep,leftmargin=1.5em]
    \item $\{0, g, g+4h, 2g+4h, 3g+h, 3g+3h, 3g+4h, 4g+3h, h\}$
    \item $\{0, d, d+e, d+2e, 2d+2e, 3d+e, 3d+2e, 3d+3e, e\}$
    \item $\{0, g, g+2h, 3g, 3g+h, 4g+h, 4g+2h, h, 4h\}$
    \item $\{0, e, e+2g, e+4g, 2e, 2e+4g, 3e, g, 4g\}$
    \item $\{0, e, e+3g, e+4g, 2e+2g, 2e+4g, 3e+2g, g, 3g\}$
\end{itemize}
For $\beta = 7$, the multi-variable configuration found is $\{0, g, 3g+3h, 3g+5h, 4g, 4g+5h, 6g+3h, h, 2h\}$.
For $\beta = 11$, the multi-variable configuration found is $\{0, d, d+e, d+3e, d+6e, d+7e, e, 4e, 6e\}$.
All multi-variable configurations for $\beta = 5, 7, 11$ are affinely equivalent to the multi-variable configurations given in Theorem 1.5.
For $\beta \ge 53$, the MSTD sets found for each prime are affinely equivalent to the 9 integer MSTD sets from Hegarty.

\subsubsection*{Case 5: $a-b=b-c=c-d=d$}
In the case where a multiplicity-4 difference of the form $a-b=b-c=c-d=d$ occurs, we get the following results:
\begin{runsummary}
    \item[Inputted Configuration:] $A = \{0, 4d, 3d, 2d, d, e, f, g, h\}$ with $a-b=b-c=c-d=d$
    \item[Runtime:] 520.2 s (8m 40.2s)
    \item[Branches Processed:] 4{,}234{,}719
    \item[Duplicate Terminations:] 756{,}567
    \item[Dedup-skipped:] 3{,}171{,}605
    \item[Max Reduction Factor:] 63
    \item[Max Prime Dividing Any Reduction:] 61
    \item[Total MSTD Leaves Found:] 2{,}639
    \item[Integer-MSTD Leaves Found:] 0
    \item[Multi-var MSTD Configurations Found:] 0
    \item[MSTD by Prime:] $\Zp{23}$: 1{,}625; \ $\Zp{29}$: 744; \ $\Zp{31}$: 210; \ $\Zp{37}$: 24; \ $\Zp{41}$: 36.
\end{runsummary}
We provide the following search statistics for each characteristic checked:
\begin{table}[H]
\centering
\setlength{\tabcolsep}{4pt}
\caption{Search statistics for size 9, Case 5: $a-b=b-c=c-d=d$}
\label{tab:size9_case5_stats}
\begin{tabular}{|c|c|c|c|c|c|c|}
\hline
$\beta$ & Runtime (s) & Branches & Duplicates & Skipped & MSTD Found & Multi-var \\
\hline
5 & 2.7 & 18,495 & 9,374 & 7,765 & 0 & 0 \\
\hline
7 & 32.9 & 152,632 & 92,305 & 44,231 & 0 & 0 \\
\hline
11 & 145.9 & 714,515 & 518,338 & 136,848 & 0 & 0 \\
\hline
13 & 185.7 & 1,016,616 & 711,866 & 231,934 & 0 & 0 \\
\hline
17 & 254.1 & 1,695,589 & 1,039,837 & 558,559 & 0 & 0 \\
\hline
19 & 268.6 & 1,959,110 & 1,102,603 & 753,347 & 0 & 0 \\
\hline
23 & 326.2 & 2,583,582 & 1,231,713 & 1,230,101 & 136 & 0 \\
\hline
29 & 379.4 & 3,497,370 & 1,315,504 & 2,021,233 & 119 & 0 \\
\hline
31 & 361.9 & 3,774,287 & 1,324,231 & 2,271,143 & 40 & 0 \\
\hline
37 & 427.7 & 4,656,231 & 1,360,697 & 3,034,244 & 8 & 0 \\
\hline
41 & 448.2 & 5,209,323 & 1,374,834 & 3,495,164 & 8 & 0 \\
\hline
43 & 469.0 & 5,509,535 & 1,388,857 & 3,735,069 & 0 & 0 \\
\hline
47 & 486.3 & 5,973,769 & 1,387,777 & 4,102,748 & 0 & 0 \\
\hline
53 & 477.1 & 6,535,400 & 1,376,098 & 4,521,306 & 0 & 0 \\
\hline
59 & 505.9 & 7,138,684 & 1,407,032 & 4,957,148 & 0 & 0 \\
\hline
61 & 582.5 & 7,293,593 & 1,407,343 & 5,055,127 & 0 & 0 \\
\hline
\textbf{Total} & \textbf{5,354.1} & \textbf{51,749,131} & \textbf{16,548,415} & \textbf{33,156,727} & \textbf{311} & \textbf{0} \\
\hline
\end{tabular}
\end{table}

\subsubsection*{Case 5: $a-b=b-c=c-d=e$}
In the case where a multiplicity-4 difference of the form $a-b=b-c=c-d=e$ occurs, we get the following results:
\begin{runsummary}
    \item[Inputted Configuration:] $A = \{0, d+3e, d+2e, d+e, d, e, f, g, h\}$ with $a-b=b-c=c-d=e$
    \item[Runtime:] 1{,}503.5 s (25m 03.5s)
    \item[Branches Processed:] 10{,}267{,}601
    \item[Duplicate Terminations:] 1{,}902{,}142
    \item[Dedup-skipped:] 7{,}575{,}937
    \item[Max Reduction Factor:] 58
    \item[Max Prime Dividing Any Reduction:] 53
    \item[Total MSTD Leaves Found:] 9{,}825
    \item[Integer-MSTD Leaves Found:] 7
    \item[Multi-var MSTD Configurations Found:] 0
    \item[MSTD by Prime:] $\Zp{23}$: 6{,}375; \ $\Zp{29}$: 2{,}360; \ $\Zp{31}$: 965; \ $\Zp{37}$: 56; \ $\Zp{41}$: 68; \ $\Zp{43}$: 1.
\end{runsummary}
We provide the following search statistics for each characteristic checked:
\begin{table}[H]
\centering
\setlength{\tabcolsep}{4pt}
\caption{Search statistics for size 9, Case 5: $a-b=b-c=c-d=e$}
\label{tab:size9_case5_e_stats}
\begin{tabular}{|c|c|c|c|c|c|c|}
\hline
$\beta$ & Runtime (s) & Branches & Duplicates & Skipped & MSTD Found & Multi-var \\
\hline
5 & 12.5 & 80,285 & 43,181 & 30,516 & 0 & 64 \\
\hline
7 & 95.2 & 433,639 & 268,173 & 118,668 & 0 & 0 \\
\hline
11 & 1,093.7 & 1,929,646 & 1,428,856 & 336,792 & 0 & 8 \\
\hline
13 & 514.0 & 2,710,184 & 1,940,411 & 570,979 & 0 & 0 \\
\hline
17 & 718.8 & 4,050,197 & 2,506,662 & 1,308,416 & 0 & 0 \\
\hline
19 & 718.4 & 4,735,235 & 2,692,275 & 1,790,021 & 0 & 0 \\
\hline
23 & 801.1 & 6,069,742 & 2,900,517 & 2,875,869 & 475 & 0 \\
\hline
29 & 911.3 & 8,079,520 & 3,065,290 & 4,614,982 & 345 & 0 \\
\hline
31 & 930.5 & 8,679,278 & 3,073,636 & 5,155,154 & 183 & 0 \\
\hline
37 & 1,174.7 & 10,584,339 & 3,131,167 & 6,773,576 & 36 & 0 \\
\hline
41 & 1,234.2 & 11,786,709 & 3,156,326 & 7,739,493 & 42 & 0 \\
\hline
43 & 1,168.2 & 12,323,449 & 3,157,087 & 8,156,768 & 24 & 0 \\
\hline
47 & 1,199.4 & 13,241,414 & 3,141,095 & 8,846,665 & 24 & 0 \\
\hline
53 & 1,252.2 & 14,534,208 & 3,159,689 & 9,759,703 & 24 & 0 \\
\hline
\textbf{Total} & \textbf{11,824.2} & \textbf{99,161,645} & \textbf{31,564,765} & \textbf{58,332,342} & \textbf{1,153} & \textbf{72} \\
\hline
\end{tabular}
\end{table}
For $\beta = 5$, the multi-variable configurations found are:
\begin{itemize}[nosep,leftmargin=1.5em]
    \item $\{0, f, 3f, 3f+h, 3f+3h, 3f+4h, 4f+3h, h, 2h\}$
    \item $\{0, d, d+e, d+2e, d+3e, 2d+3e, 2d+4e, 4d+3e, e\}$
    \item $\{0, e, e+2h, e+3h, 2e+h, 2e+3h, 3e+3h, h, 3h\}$
\end{itemize}
For $\beta = 11$, the multi-variable configuration found is $\{0, f, f+g, 3f, 4f+g, 7f+g, 8f, 9f+g, g\}$.
All multi-variable configurations for $\beta = 5$ and $\beta = 11$ are affinely equivalent to the multi-variable configurations given in Theorem 1.5.
For $\beta = 53$, the MSTD sets found for each prime are affinely equivalent to the 9 integer MSTD sets from Hegarty.

\subsubsection*{Case 5: $a-b=b-c=d-e=e$}
In the case where a multiplicity-4 difference of the form $a-b=b-c=d-e=e$ occurs, we get the following results:
\begin{runsummary}
    \item[Inputted Configuration:] $A = \{0, a, b, c, d, e, f, g, h\}$ with $a-b=b-c=d-e=e$
    \item[Runtime:] 901.9 s (15m 01.9s)
    \item[Branches Processed:] 7{,}555{,}898
    \item[Duplicate Terminations:] 1{,}500{,}002
    \item[Dedup-skipped:] 5{,}686{,}085
    \item[Max Reduction Factor:] 56
    \item[Max Prime Dividing Any Reduction:] 53
    \item[Total MSTD Leaves Found:] 2{,}954
    \item[Integer-MSTD Leaves Found:] 10
    \item[Multi-var MSTD Configurations Found:] 0
    \item[MSTD by Prime:] $\Zp{23}$: 2{,}016; \ $\Zp{29}$: 716; \ $\Zp{31}$: 202; \ $\Zp{37}$: 20.
\end{runsummary}
We provide the following search statistics for each characteristic checked:
\begin{table}[H]
\centering
\setlength{\tabcolsep}{4pt}
\caption{Search statistics for size 9, Case 5: $a-b=b-c=d-e=e$}
\label{tab:size9_case5_dee_stats}
\begin{tabular}{|c|c|c|c|c|c|c|}
\hline
$\beta$ & Runtime (s) & Branches & Duplicates & Skipped & MSTD Found & Multi-var \\
\hline
5 & 11.9 & 84,553 & 46,751 & 30,507 & 0 & 56 \\
\hline
7 & 89.3 & 426,543 & 266,245 & 113,701 & 0 & 0 \\
\hline
11 & 315.0 & 1,750,515 & 1,302,213 & 298,721 & 0 & 28 \\
\hline
13 & 394.2 & 2,362,410 & 1,695,439 & 492,738 & 0 & 0 \\
\hline
17 & 466.5 & 3,407,154 & 2,092,122 & 1,115,571 & 0 & 0 \\
\hline
19 & 509.6 & 3,944,547 & 2,214,296 & 1,516,752 & 0 & 0 \\
\hline
23 & 559.4 & 4,941,658 & 2,320,145 & 2,374,550 & 457 & 0 \\
\hline
29 & 800.8 & 6,428,816 & 2,392,248 & 3,690,869 & 372 & 0 \\
\hline
31 & 813.4 & 6,915,532 & 2,404,918 & 4,114,142 & 180 & 0 \\
\hline
37 & 793.7 & 8,362,325 & 2,442,747 & 5,307,434 & 80 & 0 \\
\hline
41 & 827.5 & 9,208,619 & 2,448,559 & 5,963,626 & 64 & 0 \\
\hline
43 & 864.4 & 9,563,462 & 2,441,370 & 6,226,975 & 63 & 0 \\
\hline
47 & 887.7 & 10,228,797 & 2,439,959 & 6,702,008 & 64 & 0 \\
\hline
53 & 828.5 & 11,016,671 & 2,445,275 & 7,230,442 & 64 & 0 \\
\hline
\textbf{Total} & \textbf{8,162.0} & \textbf{78,641,602} & \textbf{26,952,487} & \textbf{42,058,130} & \textbf{1,344} & \textbf{84} \\
\hline
\end{tabular}
\end{table}
For $\beta = 5$, the multi-variable configurations found are:
\begin{itemize}[nosep,leftmargin=1.5em]
    \item $\{0, e, 2e, 2e+3h, 3e, 3e+2h, 4e+2h, h, 2h\}$
    \item $\{0, e, 2e, 2e+4g, 3e+3g, 4e+g, 4e+3g, g, 3g\}$
\end{itemize}
For $\beta = 11$, the multi-variable configuration found is $\{0, e, 2e, 3e+h, 4e+h, 5e, 5e+h, 10e, h\}$.
All multi-variable configurations for $\beta = 5$ and $\beta = 11$ are affinely equivalent to the multi-variable configurations given in Theorem 1.5.
For $\beta \ge 53$, the MSTD sets found for each prime are affinely equivalent to the 9 integer MSTD sets from Hegarty.

\subsubsection*{Case 5: $a-b=b-c=d-e=f$}
In the case where a multiplicity-4 difference of the form $a-b=b-c=d-e=f$ occurs, we get the following results:
\begin{runsummary}
    \item[Inputted Configuration:] $A = \{0, c+2f, c+f, c, e+f, e, f, g, h\}$ with $a-b=b-c=d-e=f$
    \item[Runtime:] 2{,}546.2 s (42m 26.2s)
    \item[Branches Processed:] 21{,}659{,}868
    \item[Duplicate Terminations:] 4{,}376{,}165
    \item[Dedup-skipped:] 16{,}270{,}500
    \item[Max Reduction Factor:] 50
    \item[Max Prime Dividing Any Reduction:] 47
    \item[Total MSTD Leaves Found:] 8{,}822
    \item[Integer-MSTD Leaves Found:] 10
    \item[Multi-var MSTD Configurations Found:] 0
    \item[MSTD by Prime:] $\Zp{23}$: 5{,}952; \ $\Zp{29}$: 2{,}190; \ $\Zp{31}$: 678; \ $\Zp{37}$: 2.
\end{runsummary}
We provide the following search statistics for each characteristic checked:
\begin{table}[H]
\centering
\setlength{\tabcolsep}{4pt}
\caption{Search statistics for size 9, Case 5: $a-b=b-c=d-e=f$}
\label{tab:size9_case5_def_stats}
\begin{tabular}{|c|c|c|c|c|c|c|}
\hline
$\beta$ & Runtime (s) & Branches & Duplicates & Skipped & MSTD Found & Multi-var \\
\hline
5 & 32.4 & 202,133 & 109,312 & 77,151 & 0 & 166 \\
\hline
7 & 209.9 & 1,000,193 & 606,779 & 288,728 & 0 & 0 \\
\hline
11 & 723.5 & 4,046,449 & 2,965,182 & 741,640 & 0 & 8 \\
\hline
13 & 891.6 & 5,320,001 & 3,777,509 & 1,156,041 & 0 & 0 \\
\hline
17 & 1,046.6 & 7,673,366 & 4,693,974 & 2,538,189 & 0 & 0 \\
\hline
19 & 1,108.4 & 8,822,379 & 4,939,588 & 3,416,551 & 0 & 0 \\
\hline
23 & 1,213.4 & 11,046,118 & 5,191,740 & 5,321,182 & 809 & 0 \\
\hline
29 & 1,377.4 & 14,311,981 & 5,344,388 & 8,224,695 & 906 & 0 \\
\hline
31 & 1,414.8 & 15,394,736 & 5,375,755 & 9,163,046 & 393 & 0 \\
\hline
37 & 1,584.0 & 18,502,087 & 5,434,317 & 11,728,015 & 80 & 0 \\
\hline
41 & 1,658.7 & 20,359,014 & 5,449,187 & 13,150,265 & 80 & 0 \\
\hline
43 & 1,725.9 & 21,197,196 & 5,451,983 & 13,761,963 & 80 & 0 \\
\hline
47 & 1,794.0 & 22,668,420 & 5,453,400 & 14,790,453 & 80 & 0 \\
\hline
\textbf{Total} & \textbf{14,780.6} & \textbf{150,544,073} & \textbf{53,793,114} & \textbf{84,677,407} & \textbf{2,428} & \textbf{174} \\
\hline
\end{tabular}
\end{table}
For $\beta = 5$, the multi-variable configurations found are:
\begin{itemize}[nosep,leftmargin=1.5em]
    \item $\{0, g, g+h, g+3h, 2g+2h, 2g+3h, 4g, 4g+3h, h\}$
    \item $\{0, c, c+f, c+2f, 3c+4f, 4c+3f, 4c+4f, f, 3f\}$
    \item $\{0, f, f+2h, 2f+h, 3f+h, 4f+2h, 4f+4h, h, 2h\}$
    \item $\{0, f, f+h, 2f+h, 3f+h, 3f+2h, 4f+2h, 4f+3h, h\}$
    \item $\{0, f, f+g, 2f+g, 2f+4g, 3f+2g, 4f+2g, g, 2g\}$
\end{itemize}
For $\beta = 11$, the multi-variable configuration found is $\{0, g, 2g, 4g, 6g, 6g+h, 8g+h, 10g+h, h\}$.
All multi-variable configurations for $\beta = 5$ and $\beta = 11$ are affinely equivalent to the multi-variable configurations given in Theorem 1.5.

\subsubsection*{Case 5: $a-b=c-d=e-f=g$}
In the case where a multiplicity-4 difference of the form $a-b=c-d=e-f=g$ occurs, we get the following results:
\begin{runsummary}
    \item[Inputted Configuration:] $A = \{0, b+g, b, d+g, d, f+g, f, g,h\}$ with $a-b=c-d=e-f=g$
    \item[Runtime:] 705.8 s (11m 45.8s)
    \item[Branches Processed:] 6{,}715{,}478
    \item[Duplicate Terminations:] 1{,}460{,}224
    \item[Dedup-skipped:] 5{,}034{,}193
    \item[Max Reduction Factor:] 42
    \item[Max Prime Dividing Any Reduction:] 41
    \item[Total MSTD Leaves Found:] 2{,}172
    \item[Integer-MSTD Leaves Found:] 0
    \item[Multi-var MSTD Configurations Found:] 0
    \item[MSTD by Prime:] $\Zp{23}$: 1{,}484; \ $\Zp{29}$: 558; \ $\Zp{31}$: 108; \ $\Zp{37}$: 22.
\end{runsummary}
We provide the following search statistics for each characteristic checked:
\begin{table}[H]
\centering
\setlength{\tabcolsep}{4pt}
\caption{Search statistics for size 8, Case 5: $a-b=c-d=e-f=g$}
\label{tab:size8_case5_g_stats}
\begin{tabular}{|c|c|c|c|c|c|c|}
\hline
$\beta$ & Runtime (s) & Branches & Duplicates & Skipped & MSTD Found & Multi-var \\
\hline
5 & 18.3 & 127,310 & 67,200 & 52,104 & 0 & 150 \\
\hline
7 & 89.7 & 511,623 & 277,404 & 189,058 & 0 & 0 \\
\hline
11 & 290.6 & 1,719,388 & 1,171,053 & 415,749 & 0 & 0 \\
\hline
13 & 334.9 & 2,202,574 & 1,470,201 & 582,905 & 0 & 0 \\
\hline
17 & 400.1 & 3,029,155 & 1,767,735 & 1,094,384 & 0 & 0 \\
\hline
19 & 411.2 & 3,393,984 & 1,820,015 & 1,400,336 & 0 & 0 \\
\hline
23 & 446.6 & 4,141,545 & 1,880,895 & 2,066,385 & 436 & 0 \\
\hline
29 & 482.7 & 5,209,325 & 1,904,257 & 3,052,930 & 322 & 0 \\
\hline
31 & 480.2 & 5,559,237 & 1,910,033 & 3,367,056 & 81 & 0 \\
\hline
37 & 508.9 & 6,552,044 & 1,921,331 & 4,225,294 & 17 & 0 \\
\hline
41 & 550.0 & 7,095,151 & 1,914,788 & 4,676,882 & 0 & 0 \\
\hline
\textbf{Total} & \textbf{4,013.2} & \textbf{39,541,336} & \textbf{16,105,912} & \textbf{21,193,083} & \textbf{856} & \textbf{150} \\
\hline
\end{tabular}
\end{table}
For $\beta = 5$, the multi-variable configurations found are:
\begin{itemize}[nosep,leftmargin=1.5em]
    \item $\{0, g, 3g+2h, 3g+4h, 4g+2h, 4g+3h, 4g+4h, h, 3h\}$
    \item $\{0, g, g+2h, 3g+4h, 4g+3h, 4g+4h, h, 2h, 3h\}$
    \item $\{0, d, d+g, 2d+g, 3d, 3d+4g, 4d+3g, 4d+4g, g\}$
    \item $\{0, d, d+g, 2d, 3d, 3d+g, 4d+3g, 4d+4g, g\}$
    \item $\{0, f, f+g, f+3g, f+4g, 2f+2g, 2f+3g, 4f+3g, g\}$
\end{itemize}
All multi-variable configurations for $\beta = 5$ are affinely equivalent to the multi-variable configurations given in Theorem 1.4.

\subsubsection*{The case where $\beta = 3$}
In the case where a multiplicity-4 difference of the form $a-b=c-d=e-f=g$ occurs, with the additional constraint that the ambient field has characteristic 3, we get the following results:
\begin{runsummary}
    \item[Inputted Configuration:] $A = \{0, a, b, c, d, e, f, g\}$
    \item[Runtime:] 144.5 s
    \item[Branches Processed:] 1{,}405{,}694
    \item[Duplicate Terminations:] 633{,}176
    \item[Dedup-skipped:] 727{,}471
    \item[Max Reduction Factor:] 1
    \item[Max Prime Dividing Any Reduction:] 1
    \item[Total MSTD Leaves Found:] 0
    \item[Multi-var MSTD Configurations Found:] 65
    \item[MSTD by Prime:] $\Zp{3}$: 0 configs.
\end{runsummary}
The multi-variable configurations found are:
\begin{itemize}[nosep,leftmargin=1.5em]
    \item $\{0, b, 2b, 2b+f+g, 2b+g, f, f+g, 2f, g\}$
    \item $\{0, f, f+g+h, f+g+2h, f+2g, 2f+h, g, g+h, h\}$
    \item $\{0, e, e+2f, 2e, 2e+2h, f, f+h, 2f+2h, h\}$
\end{itemize}
All multi-variable configurations for $\beta = 3$ are affinely equivalent to the multi-variable configurations given in Theorem 1.5.

\subsection*{Size-9 MSTD Sets in $\Z/p\Z$ and Affine Equivalence to the Nine Integer Sets}
For each prime $p$ with $5 \le p \le 47$, we list every size-9 MSTD configuration in $\Z/p\Z$ up to affine transformation. There are no such sets for $p \le 19$. A configuration is flagged when it is affinely equivalent to one of the nine integer MSTD sets, with the corresponding integer set written out. Recall the nine integer MSTD sets:
\begin{enumerate}[nosep,leftmargin=2em]
\item $\{0, 1, 2, 4, 5, 9, 12, 13, 14\}$
\item $\{0, 1, 2, 4, 7, 8, 12, 14, 15\}$
\item $\{0, 2, 3, 4, 7, 9, 13, 14, 16\}$
\item $\{0, 2, 3, 4, 7, 11, 12, 14, 16\}$
\item $\{0, 2, 3, 4, 7, 11, 15, 16, 18\}$
\item $\{0, 2, 4, 8, 9, 10, 15, 17, 19\}$
\item $\{0, 4, 6, 7, 8, 14, 15, 17, 21\}$
\item $\{0, 4, 6, 8, 11, 14, 19, 21, 25\}$
\item $\{0, 5, 6, 9, 10, 13, 16, 17, 22\}$
\end{enumerate}

\subsubsection*{$p = 23$ \normalfont(19 configurations; 4 affinely equivalent to an integer set)}
\begin{itemize}[nosep,leftmargin=1.5em]
\item $\{0, 1, 2, 3, 4, 5, 8, 9, 18\}$
\item $\{0, 1, 2, 3, 4, 5, 8, 10, 18\}$\;{\color{blue!55!black}$\longrightarrow$ affinely equivalent to $\{0, 1, 2, 4, 5, 9, 12, 13, 14\}$}
\item $\{0, 1, 2, 3, 4, 6, 7, 9, 18\}$
\item $\{0, 1, 2, 3, 4, 6, 7, 13, 20\}$
\item $\{0, 1, 2, 3, 4, 6, 8, 9, 20\}$
\item $\{0, 1, 2, 3, 4, 6, 9, 10, 19\}$
\item $\{0, 1, 2, 3, 4, 6, 10, 19, 20\}$
\item $\{0, 1, 2, 3, 4, 7, 8, 9, 17\}$
\item $\{0, 1, 2, 3, 4, 7, 8, 10, 17\}$
\item $\{0, 1, 2, 3, 4, 8, 9, 10, 18\}$
\item $\{0, 1, 2, 3, 4, 10, 11, 13, 17\}$
\item $\{0, 1, 2, 3, 4, 10, 12, 14, 16\}$\;{\color{blue!55!black}$\longrightarrow$ affinely equivalent to $\{0, 2, 4, 8, 9, 10, 15, 17, 19\}$}
\item $\{0, 1, 2, 3, 5, 6, 7, 12, 19\}$\;{\color{blue!55!black}$\longrightarrow$ affinely equivalent to $\{0, 4, 6, 7, 8, 14, 15, 17, 21\}$}
\item $\{0, 1, 2, 3, 5, 6, 8, 9, 17\}$
\item $\{0, 1, 2, 3, 5, 6, 9, 10, 19\}$
\item $\{0, 1, 2, 3, 5, 6, 11, 14, 15\}$\;{\color{blue!55!black}$\longrightarrow$ affinely equivalent to $\{0, 2, 3, 4, 7, 11, 12, 14, 16\}$}
\item $\{0, 1, 2, 3, 5, 6, 14, 15, 20\}$
\item $\{0, 1, 2, 3, 5, 7, 8, 13, 20\}$
\item $\{0, 1, 2, 3, 5, 8, 17, 19, 20\}$
\end{itemize}
Note that although four of the nineteen configurations are affinely equivalent to integer sets, we still have that each configuration is unique to $\Z/23\Z$, since the sizes of the sumset and difference sets for these sets are not equal to those of the four integer sets. This is because $2x >23$ when $x=14,16,19,21$.

\subsubsection*{$p = 29$ \normalfont(59 configurations; 9 affinely equivalent to an integer set)}
\begin{itemize}[nosep,leftmargin=1.5em]
\item $\{0, 1, 2, 3, 4, 5, 9, 10, 24\}$
\item $\{0, 1, 2, 3, 4, 7, 8, 16, 24\}$
\item $\{0, 1, 2, 3, 4, 7, 9, 16, 24\}$
\item $\{0, 1, 2, 3, 4, 8, 10, 23, 24\}$
\item $\{0, 1, 2, 3, 4, 8, 13, 20, 24\}$
\item $\{0, 1, 2, 3, 4, 10, 13, 17, 21\}$
\item $\{0, 1, 2, 3, 4, 10, 14, 17, 21\}$
\item $\{0, 1, 2, 3, 4, 10, 16, 18, 22\}$
\item $\{0, 1, 2, 3, 4, 11, 12, 15, 19\}$
\item $\{0, 1, 2, 3, 4, 11, 14, 18, 21\}$\;{\color{blue!55!black}$\longrightarrow$ affinely equivalent to $\{0, 5, 6, 9, 10, 13, 16, 17, 22\}$}
\item $\{0, 1, 2, 3, 4, 12, 13, 16, 20\}$
\item $\{0, 1, 2, 3, 4, 12, 14, 18, 20\}$\;{\color{blue!55!black}$\longrightarrow$ affinely equivalent to $\{0, 4, 6, 8, 11, 14, 19, 21, 25\}$}
\item $\{0, 1, 2, 3, 4, 12, 16, 19, 20\}$
\item $\{0, 1, 2, 3, 5, 6, 7, 16, 25\}$
\item $\{0, 1, 2, 3, 5, 6, 9, 10, 22\}$
\item $\{0, 1, 2, 3, 5, 6, 13, 19, 27\}$\;{\color{blue!55!black}$\longrightarrow$ affinely equivalent to $\{0, 4, 6, 7, 8, 14, 15, 17, 21\}$}
\item $\{0, 1, 2, 3, 5, 6, 14, 18, 26\}$\;{\color{blue!55!black}$\longrightarrow$ affinely equivalent to $\{0, 2, 3, 4, 7, 11, 15, 16, 18\}$}
\item $\{0, 1, 2, 3, 5, 6, 17, 18, 26\}$
\item $\{0, 1, 2, 3, 5, 7, 8, 16, 23\}$
\item $\{0, 1, 2, 3, 5, 7, 10, 11, 25\}$
\item $\{0, 1, 2, 3, 5, 8, 9, 17, 25\}$
\item $\{0, 1, 2, 3, 5, 8, 10, 22, 25\}$
\item $\{0, 1, 2, 3, 5, 9, 14, 17, 25\}$
\item $\{0, 1, 2, 3, 5, 9, 15, 17, 25\}$\;{\color{blue!55!black}$\longrightarrow$ affinely equivalent to $\{0, 1, 2, 4, 7, 8, 12, 14, 15\}$}
\item $\{0, 1, 2, 3, 5, 10, 22, 24, 25\}$
\item $\{0, 1, 2, 3, 5, 11, 13, 14, 19\}$
\item $\{0, 1, 2, 3, 5, 12, 15, 19, 25\}$
\item $\{0, 1, 2, 3, 5, 14, 19, 21, 22\}$
\item $\{0, 1, 2, 3, 6, 7, 10, 21, 25\}$
\item $\{0, 1, 2, 3, 6, 7, 16, 22, 26\}$
\item $\{0, 1, 2, 3, 6, 8, 14, 18, 24\}$\;{\color{blue!55!black}$\longrightarrow$ affinely equivalent to $\{0, 2, 4, 8, 9, 10, 15, 17, 19\}$}
\item $\{0, 1, 2, 3, 6, 9, 10, 16, 25\}$
\item $\{0, 1, 2, 3, 6, 10, 16, 22, 25\}$
\item $\{0, 1, 2, 3, 6, 10, 21, 24, 25\}$
\item $\{0, 1, 2, 3, 6, 10, 22, 23, 25\}$
\item $\{0, 1, 2, 3, 7, 8, 10, 22, 25\}$\;{\color{blue!55!black}$\longrightarrow$ affinely equivalent to $\{0, 1, 2, 4, 5, 9, 12, 13, 14\}$}
\item $\{0, 1, 2, 3, 7, 8, 11, 20, 24\}$
\item $\{0, 1, 2, 3, 7, 9, 10, 21, 24\}$
\item $\{0, 1, 2, 3, 7, 9, 12, 20, 25\}$\;{\color{blue!55!black}$\longrightarrow$ affinely equivalent to $\{0, 2, 3, 4, 7, 9, 13, 14, 16\}$}
\item $\{0, 1, 2, 3, 7, 10, 16, 21, 23\}$
\item $\{0, 1, 2, 3, 7, 10, 16, 21, 25\}$
\item $\{0, 1, 2, 3, 7, 10, 21, 22, 24\}$
\item $\{0, 1, 2, 3, 8, 9, 11, 22, 24\}$
\item $\{0, 1, 2, 3, 8, 11, 17, 19, 22\}$
\item $\{0, 1, 2, 3, 8, 13, 14, 19, 22\}$
\item $\{0, 1, 2, 3, 8, 18, 20, 21, 23\}$
\item $\{0, 1, 2, 4, 5, 7, 8, 18, 23\}$
\item $\{0, 1, 2, 4, 5, 9, 11, 19, 23\}$
\item $\{0, 1, 2, 4, 5, 9, 12, 19, 22\}$
\item $\{0, 1, 2, 4, 5, 10, 11, 13, 23\}$
\item $\{0, 1, 2, 4, 5, 10, 13, 14, 25\}$
\item $\{0, 1, 2, 4, 5, 12, 13, 15, 20\}$
\item $\{0, 1, 2, 4, 5, 12, 14, 19, 21\}$
\item $\{0, 1, 2, 4, 5, 12, 14, 21, 26\}$
\item $\{0, 1, 2, 4, 5, 14, 19, 21, 22\}$
\item $\{0, 1, 2, 4, 7, 8, 18, 19, 23\}$
\item $\{0, 1, 2, 4, 7, 11, 13, 14, 26\}$
\item $\{0, 1, 2, 4, 12, 14, 18, 25, 26\}$
\item $\{0, 1, 2, 4, 17, 19, 21, 22, 26\}$\;{\color{blue!55!black}$\longrightarrow$ affinely equivalent to $\{0, 2, 3, 4, 7, 11, 12, 14, 16\}$}
\end{itemize}
Only the set $A=\{0, 1, 2, 3, 7, 8, 10, 22, 25\}$ is affinely equivalent to an integer MSTD set since $2(14)<29$, so $|A+A|$ and $|A-A|$ are equivalent to the sizes of the sumset and difference set of the integer MSTD set $\{0,1,2,4,5,9,12,13,14\}$.

\subsubsection*{$p = 31$ \normalfont(35 configurations; 9 affinely equivalent to an integer set)}
\begin{itemize}[nosep,leftmargin=1.5em]
\item $\{0, 1, 2, 3, 4, 6, 10, 11, 26\}$
\item $\{0, 1, 2, 3, 4, 8, 9, 17, 26\}$
\item $\{0, 1, 2, 3, 4, 8, 11, 24, 26\}$
\item $\{0, 1, 2, 3, 4, 11, 14, 18, 22\}$
\item $\{0, 1, 2, 3, 4, 11, 15, 18, 22\}$
\item $\{0, 1, 2, 3, 5, 6, 13, 24, 26\}$
\item $\{0, 1, 2, 3, 5, 7, 8, 16, 25\}$
\item $\{0, 1, 2, 3, 5, 10, 11, 23, 29\}$\;{\color{blue!55!black}$\longrightarrow$ affinely equivalent to $\{0, 4, 6, 7, 8, 14, 15, 17, 21\}$}
\item $\{0, 1, 2, 3, 5, 13, 17, 21, 22\}$
\item $\{0, 1, 2, 3, 6, 7, 10, 24, 28\}$
\item $\{0, 1, 2, 3, 6, 7, 14, 18, 20\}$
\item $\{0, 1, 2, 3, 6, 8, 11, 12, 27\}$
\item $\{0, 1, 2, 3, 6, 9, 10, 24, 27\}$
\item $\{0, 1, 2, 3, 6, 9, 11, 24, 27\}$
\item $\{0, 1, 2, 3, 6, 10, 13, 21, 28\}$\;{\color{blue!55!black}$\longrightarrow$ affinely equivalent to $\{0, 5, 6, 9, 10, 13, 16, 17, 22\}$}
\item $\{0, 1, 2, 3, 7, 8, 11, 23, 27\}$\;{\color{blue!55!black}$\longrightarrow$ affinely equivalent to $\{0, 2, 3, 4, 7, 11, 15, 16, 18\}$}
\item $\{0, 1, 2, 3, 7, 9, 15, 18, 26\}$\;{\color{blue!55!black}$\longrightarrow$ affinely equivalent to $\{0, 1, 2, 4, 7, 8, 12, 14, 15\}$}
\item $\{0, 1, 2, 3, 7, 11, 17, 23, 26\}$
\item $\{0, 1, 2, 3, 7, 11, 22, 23, 26\}$
\item $\{0, 1, 2, 3, 7, 15, 18, 21, 26\}$
\item $\{0, 1, 2, 3, 8, 10, 11, 23, 26\}$\;{\color{blue!55!black}$\longrightarrow$ affinely equivalent to $\{0, 1, 2, 4, 5, 9, 12, 13, 14\}$}
\item $\{0, 1, 2, 3, 8, 13, 14, 16, 19\}$
\item $\{0, 1, 2, 3, 8, 14, 16, 19, 25\}$
\item $\{0, 1, 2, 3, 9, 10, 13, 21, 25\}$\;{\color{blue!55!black}$\longrightarrow$ affinely equivalent to $\{0, 4, 6, 8, 11, 14, 19, 21, 25\}$}
\item $\{0, 1, 2, 4, 5, 9, 11, 25, 26\}$
\item $\{0, 1, 2, 4, 5, 10, 11, 13, 23\}$
\item $\{0, 1, 2, 4, 5, 10, 11, 25, 27\}$
\item $\{0, 1, 2, 4, 5, 20, 23, 24, 25\}$\;{\color{blue!55!black}$\longrightarrow$ affinely equivalent to $\{0, 2, 4, 8, 9, 10, 15, 17, 19\}$}
\item $\{0, 1, 2, 4, 7, 11, 16, 26, 28\}$
\item $\{0, 1, 2, 4, 9, 10, 12, 13, 28\}$
\item $\{0, 1, 2, 4, 9, 19, 21, 23, 24\}$
\item $\{0, 1, 2, 4, 19, 21, 22, 26, 28\}$\;{\color{blue!55!black}$\longrightarrow$ affinely equivalent to $\{0, 2, 3, 4, 7, 9, 13, 14, 16\}$}
\item $\{0, 1, 2, 4, 19, 21, 23, 24, 28\}$\;{\color{blue!55!black}$\longrightarrow$ affinely equivalent to $\{0, 2, 3, 4, 7, 11, 12, 14, 16\}$}
\item $\{0, 1, 2, 5, 7, 10, 11, 25, 27\}$
\item $\{0, 1, 2, 5, 7, 11, 12, 14, 18\}$
\end{itemize}
Only the two sets affinely equivalent to integer MSTD sets with diameter less than half of 31 in $\Z/31\Z$ have the same sumset and difference set sizes as the integer MSTD sets do in the integers.

\subsubsection*{$p = 37$ \normalfont(13 configurations; 9 affinely equivalent to an integer set)}
\begin{itemize}[nosep,leftmargin=1.5em]
\item $\{0, 1, 2, 3, 4, 14, 18, 21, 25\}$
\item $\{0, 1, 2, 3, 6, 7, 17, 23, 33\}$
\item $\{0, 1, 2, 3, 7, 11, 12, 28, 32\}$
\item $\{0, 1, 2, 3, 8, 11, 12, 28, 32\}$\;{\color{blue!55!black}$\longrightarrow$ affinely equivalent to $\{0, 5, 6, 9, 10, 13, 16, 17, 22\}$}
\item $\{0, 1, 2, 3, 9, 10, 12, 28, 31\}$\;{\color{blue!55!black}$\longrightarrow$ affinely equivalent to $\{0, 1, 2, 4, 5, 9, 12, 13, 14\}$}
\item $\{0, 1, 2, 3, 9, 12, 13, 27, 31\}$\;{\color{blue!55!black}$\longrightarrow$ affinely equivalent to $\{0, 2, 3, 4, 7, 11, 15, 16, 18\}$}
\item $\{0, 1, 2, 3, 11, 13, 19, 22, 28\}$\;{\color{blue!55!black}$\longrightarrow$ affinely equivalent to $\{0, 1, 2, 4, 7, 8, 12, 14, 15\}$}
\item $\{0, 1, 2, 3, 13, 16, 20, 23, 27\}$
\item $\{0, 1, 2, 3, 13, 17, 18, 22, 27\}$\;{\color{blue!55!black}$\longrightarrow$ affinely equivalent to $\{0, 4, 6, 7, 8, 14, 15, 17, 21\}$}
\item $\{0, 1, 2, 4, 5, 23, 26, 27, 28\}$\;{\color{blue!55!black}$\longrightarrow$ affinely equivalent to $\{0, 2, 4, 8, 9, 10, 15, 17, 19\}$}
\item $\{0, 1, 2, 4, 10, 12, 13, 17, 34\}$\;{\color{blue!55!black}$\longrightarrow$ affinely equivalent to $\{0, 4, 6, 8, 11, 14, 19, 21, 25\}$}
\item $\{0, 1, 2, 4, 25, 27, 28, 32, 34\}$\;{\color{blue!55!black}$\longrightarrow$ affinely equivalent to $\{0, 2, 3, 4, 7, 9, 13, 14, 16\}$}
\item $\{0, 1, 2, 4, 25, 27, 29, 30, 34\}$\;{\color{blue!55!black}$\longrightarrow$ affinely equivalent to $\{0, 2, 3, 4, 7, 11, 12, 14, 16\}$}
\end{itemize}
Only the five sets affinely equivalent to integer MSTD sets with diameter less than half of 37 in $\Z/37\Z$ have the same sumset and difference set sizes as the integer MSTD sets do in the integers.

\subsubsection*{$p = 41$ \normalfont(12 configurations; 9 affinely equivalent to an integer set)}
\begin{itemize}[nosep,leftmargin=1.5em]
\item $\{0, 1, 2, 3, 4, 16, 20, 24, 28\}$
\item $\{0, 1, 2, 3, 5, 6, 14, 30, 38\}$
\item $\{0, 1, 2, 3, 7, 8, 20, 24, 36\}$\;{\color{blue!55!black}$\longrightarrow$ affinely equivalent to $\{0, 4, 6, 7, 8, 14, 15, 17, 21\}$}
\item $\{0, 1, 2, 3, 8, 13, 30, 32, 35\}$
\item $\{0, 1, 2, 3, 9, 12, 13, 31, 35\}$\;{\color{blue!55!black}$\longrightarrow$ affinely equivalent to $\{0, 5, 6, 9, 10, 13, 16, 17, 22\}$}
\item $\{0, 1, 2, 3, 10, 11, 13, 31, 34\}$\;{\color{blue!55!black}$\longrightarrow$ affinely equivalent to $\{0, 1, 2, 4, 5, 9, 12, 13, 14\}$}
\item $\{0, 1, 2, 3, 10, 13, 14, 30, 34\}$\;{\color{blue!55!black}$\longrightarrow$ affinely equivalent to $\{0, 2, 3, 4, 7, 11, 15, 16, 18\}$}
\item $\{0, 1, 2, 3, 12, 14, 21, 24, 31\}$\;{\color{blue!55!black}$\longrightarrow$ affinely equivalent to $\{0, 1, 2, 4, 7, 8, 12, 14, 15\}$}
\item $\{0, 1, 2, 4, 5, 25, 28, 29, 30\}$\;{\color{blue!55!black}$\longrightarrow$ affinely equivalent to $\{0, 2, 4, 8, 9, 10, 15, 17, 19\}$}
\item $\{0, 1, 2, 4, 12, 14, 15, 19, 38\}$\;{\color{blue!55!black}$\longrightarrow$ affinely equivalent to $\{0, 4, 6, 8, 11, 14, 19, 21, 25\}$}
\item $\{0, 1, 2, 4, 29, 31, 32, 36, 38\}$\;{\color{blue!55!black}$\longrightarrow$ affinely equivalent to $\{0, 2, 3, 4, 7, 9, 13, 14, 16\}$}
\item $\{0, 1, 2, 4, 29, 31, 33, 34, 38\}$\;{\color{blue!55!black}$\longrightarrow$ affinely equivalent to $\{0, 2, 3, 4, 7, 11, 12, 14, 16\}$}
\end{itemize}
Only the six sets affinely equivalent to integer MSTD sets with diameter less than half of 41 in $\Z/41\Z$ have the same sumset and difference set sizes as the integer MSTD sets do in the integers.

\subsubsection*{$p = 43$ \normalfont(9 configurations; 9 affinely equivalent to an integer set)}
\begin{itemize}[nosep,leftmargin=1.5em]
\item $\{0, 1, 2, 3, 7, 8, 19, 27, 39\}$\;{\color{blue!55!black}$\longrightarrow$ affinely equivalent to $\{0, 4, 6, 7, 8, 14, 15, 17, 21\}$}
\item $\{0, 1, 2, 3, 10, 11, 14, 32, 36\}$\;{\color{blue!55!black}$\longrightarrow$ affinely equivalent to $\{0, 2, 3, 4, 7, 11, 15, 16, 18\}$}
\item $\{0, 1, 2, 3, 10, 12, 21, 24, 35\}$\;{\color{blue!55!black}$\longrightarrow$ affinely equivalent to $\{0, 1, 2, 4, 7, 8, 12, 14, 15\}$}
\item $\{0, 1, 2, 3, 11, 12, 15, 31, 35\}$\;{\color{blue!55!black}$\longrightarrow$ affinely equivalent to $\{0, 5, 6, 9, 10, 13, 16, 17, 22\}$}
\item $\{0, 1, 2, 3, 11, 13, 14, 32, 35\}$\;{\color{blue!55!black}$\longrightarrow$ affinely equivalent to $\{0, 1, 2, 4, 5, 9, 12, 13, 14\}$}
\item $\{0, 1, 2, 4, 5, 26, 29, 30, 31\}$\;{\color{blue!55!black}$\longrightarrow$ affinely equivalent to $\{0, 2, 4, 8, 9, 10, 15, 17, 19\}$}
\item $\{0, 1, 2, 4, 13, 15, 16, 20, 40\}$\;{\color{blue!55!black}$\longrightarrow$ affinely equivalent to $\{0, 4, 6, 8, 11, 14, 19, 21, 25\}$}
\item $\{0, 1, 2, 4, 31, 33, 34, 38, 40\}$\;{\color{blue!55!black}$\longrightarrow$ affinely equivalent to $\{0, 2, 3, 4, 7, 9, 13, 14, 16\}$}
\item $\{0, 1, 2, 4, 31, 33, 35, 36, 40\}$\;{\color{blue!55!black}$\longrightarrow$ affinely equivalent to $\{0, 2, 3, 4, 7, 11, 12, 14, 16\}$}
\end{itemize}
Only the seven sets affinely equivalent to integer MSTD sets with diameter less than half of 43 in $\Z/43\Z$ have the same sumset and difference set sizes as the integer MSTD sets do in the integers.

\subsubsection*{$p = 47$ \normalfont(9 configurations; 9 affinely equivalent to an integer set)}
\begin{itemize}[nosep,leftmargin=1.5em]
\item $\{0, 1, 2, 3, 11, 12, 15, 35, 39\}$\;{\color{blue!55!black}$\longrightarrow$ affinely equivalent to $\{0, 2, 3, 4, 7, 11, 15, 16, 18\}$}
\item $\{0, 1, 2, 3, 11, 13, 23, 26, 38\}$\;{\color{blue!55!black}$\longrightarrow$ affinely equivalent to $\{0, 1, 2, 4, 7, 8, 12, 14, 15\}$}
\item $\{0, 1, 2, 3, 12, 13, 16, 34, 38\}$\;{\color{blue!55!black}$\longrightarrow$ affinely equivalent to $\{0, 5, 6, 9, 10, 13, 16, 17, 22\}$}
\item $\{0, 1, 2, 3, 12, 14, 15, 35, 38\}$\;{\color{blue!55!black}$\longrightarrow$ affinely equivalent to $\{0, 1, 2, 4, 5, 9, 12, 13, 14\}$}
\item $\{0, 1, 2, 3, 14, 21, 22, 29, 36\}$\;{\color{blue!55!black}$\longrightarrow$ affinely equivalent to $\{0, 4, 6, 7, 8, 14, 15, 17, 21\}$}
\item $\{0, 1, 2, 4, 5, 28, 31, 32, 33\}$\;{\color{blue!55!black}$\longrightarrow$ affinely equivalent to $\{0, 2, 4, 8, 9, 10, 15, 17, 19\}$}
\item $\{0, 1, 2, 4, 15, 17, 18, 22, 44\}$\;{\color{blue!55!black}$\longrightarrow$ affinely equivalent to $\{0, 4, 6, 8, 11, 14, 19, 21, 25\}$}
\item $\{0, 1, 2, 4, 35, 37, 38, 42, 44\}$\;{\color{blue!55!black}$\longrightarrow$ affinely equivalent to $\{0, 2, 3, 4, 7, 9, 13, 14, 16\}$}
\item $\{0, 1, 2, 4, 35, 37, 39, 40, 44\}$\;{\color{blue!55!black}$\longrightarrow$ affinely equivalent to $\{0, 2, 3, 4, 7, 11, 12, 14, 16\}$}
\end{itemize}
Only the eight sets affinely equivalent to integer MSTD sets with diameter less than half of 47 in $\Z/47\Z$ have the same sumset and difference set sizes as the integer MSTD sets do in the integers.
\newpage 
\section*{Section 4: Multiplicative Subgroups as MSTD Sets}
\begin{longtable}{|c|c|c|c|c|c|c|}
\caption{Search statistics for the MSTD multiplicative-subgroup resultant scan by odd order $n$. Resultant values below $10^{10}$ are written in full; larger values are rounded to the thousandth place and given in scientific notation.}
\label{tab:subgroup_resultant_stats}\\
\hline
$n$ & Runtime (s) & Distinct Res. & Max Res. & $\mathrm{Res}=0$ & Primes checked & MSTD primes \\
\hline
\endfirsthead
\multicolumn{7}{c}{\tablename\ \thetable\ -- continued from previous page}\\
\hline
$n$ & Runtime (s) & Distinct Res. & Max Res. & $\mathrm{Res}=0$ & Primes checked & MSTD primes \\
\hline
\endhead
\hline \multicolumn{7}{r}{\emph{continued on next page}}\\
\endfoot
\hline
\endlastfoot
7 & 0.0 & 5 & 301 & 0 & 2 & 0 \\
\hline
9 & 0.0 & 10 & 729 & 0 & 4 & 0 \\
\hline
11 & 0.0 & 9 & 7,513 & 0 & 6 & 0 \\
\hline
13 & 0.0 & 13 & 35,503 & 0 & 9 & 0 \\
\hline
15 & 0.0 & 28 & 6,561 & 0 & 13 & 0 \\
\hline
17 & 0.1 & 22 & 742,747 & 0 & 19 & 0 \\
\hline
19 & 0.1 & 29 & 3,320,497 & 0 & 25 & 0 \\
\hline
21 & 0.1 & 53 & 531,441 & 0 & 32 & 0 \\
\hline
23 & 0.1 & 35 & 64,312,669 & 0 & 30 & 0 \\
\hline
25 & 0.2 & 67 & 503,284,375 & 0 & 47 & 0 \\
\hline
27 & 0.3 & 90 & 387,420,489 & 0 & 60 & 0 \\
\hline
29 & 0.4 & 66 & 5,189,752,159 & 0 & 57 & 0 \\
\hline
31 & 0.5 & 87 & $2.219\times10^{10}$ & 0 & 77 & 0 \\
\hline
33 & 0.6 & 131 & 3,486,784,401 & 0 & 94 & 0 \\
\hline
35 & 0.7 & 159 & $2.768\times10^{10}$ & 0 & 116 & 0 \\
\hline
37 & 2.0 & 131 & $1.695\times10^{12}$ & 0 & 121 & 0 \\
\hline
39 & 2.0 & 191 & $2.824\times10^{11}$ & 0 & 136 & 0 \\
\hline
41 & 2.4 & 160 & $3.005\times10^{13}$ & 0 & 159 & 0 \\
\hline
43 & 2.7 & 183 & $1.261\times10^{14}$ & 0 & 188 & 0 \\
\hline
45 & 2.5 & 336 & $2.824\times10^{11}$ & 0 & 234 & 0 \\
\hline
47 & 3.4 & 221 & $2.205\times10^{15}$ & 0 & 233 & 0 \\
\hline
49 & 4.9 & 286 & $2.239\times10^{17}$ & 0 & 288 & 0 \\
\hline
51 & 4.6 & 370 & $1.853\times10^{15}$ & 0 & 283 & 0 \\
\hline
53 & 5.8 & 287 & $1.591\times10^{17}$ & 0 & 301 & 0 \\
\hline
55 & 6.8 & 442 & $2.533\times10^{17}$ & 0 & 398 & 0 \\
\hline
57 & 6.8 & 506 & $1.501\times10^{17}$ & 0 & 427 & 0 \\
\hline
59 & 8.7 & 364 & $1.134\times10^{19}$ & 0 & 426 & 0 \\
\hline
61 & 10.7 & 393 & $4.689\times10^{19}$ & 0 & 464 & 0 \\
\hline
63 & 9.3 & 706 & $1.501\times10^{17}$ & 0 & 528 & 0 \\
\hline
65 & 12.0 & 620 & $7.662\times10^{20}$ & 0 & 603 & 0 \\
\hline
67 & 15.3 & 481 & $3.296\times10^{21}$ & 0 & 595 & 0 \\
\hline
69 & 14.9 & 757 & $9.848\times10^{20}$ & 0 & 645 & 0 \\
\hline
71 & 19.1 & 544 & $5.588\times10^{22}$ & 0 & 671 & 0 \\
\hline
73 & 21.4 & 578 & $2.298\times10^{23}$ & 0 & 707 & 0 \\
\hline
75 & 18.3 & 1,111 & $1.216\times10^{19}$ & 0 & 884 & 0 \\
\hline
77 & 25.0 & 850 & $6.105\times10^{24}$ & 0 & 982 & 0 \\
\hline
79 & 30.3 & 684 & $1.592\times10^{25}$ & 0 & 901 & 0 \\
\hline
81 & 24.5 & 1,103 & $5.815\times10^{25}$ & 0 & 972 & 0 \\
\hline
83 & 37.3 & 759 & $2.676\times10^{26}$ & 0 & 957 & 0 \\
\hline
85 & 37.8 & 1,086 & $7.011\times10^{27}$ & 0 & 1,247 & 0 \\
\hline
87 & 36.9 & 1,279 & $5.233\times10^{26}$ & 0 & 1,295 & 0 \\
\hline
89 & 56.0 & 880 & $1.836\times10^{28}$ & 0 & 1,175 & 0 \\
\hline
91 & 65.0 & 1,191 & $5.531\times10^{29}$ & 0 & 1,456 & 0 \\
\hline
93 & 53.5 & 1,485 & $4.239\times10^{28}$ & 0 & 1,527 & 0 \\
\hline
95 & 59.3 & 1,365 & $2.121\times10^{31}$ & 0 & 1,569 & 0 \\
\hline
97 & 86.0 & 1,056 & $5.123\times10^{30}$ & 0 & 1,468 & 0 \\
\hline
99 & 74.6 & 1,831 & $4.239\times10^{28}$ & 0 & 1,788 & 0 \\
\hline
101 & 93.8 & 1,149 & $8.536\times10^{31}$ & 0 & 1,683 & 0 \\
\hline
103 & 101.2 & 1,198 & $3.482\times10^{32}$ & 0 & 1,708 & 0 \\
\hline
105 & 80.1 & 2,692 & $7.977\times10^{22}$ & 0 & 2,176 & 0 \\
\hline
107 & 122.7 & 1,297 & $5.787\times10^{33}$ & 0 & 1,910 & 0 \\
\hline
109 & 125.1 & 1,349 & $2.358\times10^{34}$ & 0 & 2,050 & 0 \\
\hline
111 & 116.3 & 2,160 & $2.253\times10^{34}$ & 0 & 2,407 & 0 \\
\hline
113 & 148.5 & 1,454 & $3.912\times10^{35}$ & 0 & 2,167 & 0 \\
\hline
115 & 131.3 & 2,007 & $1.941\times10^{38}$ & 0 & 2,620 & 0 \\
\hline
117 & 115.7 & 2,579 & $2.253\times10^{34}$ & 0 & 2,924 & 0 \\
\hline
119 & 165.0 & 2,040 & $4.540\times10^{39}$ & 0 & 2,742 & 0 \\
\hline
121 & 174.2 & 1,844 & $4.305\times10^{42}$ & 0 & 2,783 & 0 \\
\hline
123 & 150.0 & 2,673 & $1.478\times10^{38}$ & 0 & 3,207 & 0 \\
\hline
125 & 177.6 & 2,294 & $3.229\times10^{43}$ & 0 & 3,053 & 0 \\
\hline
127 & 262.1 & 1,856 & $7.203\times10^{39}$ & 0 & 2,927 & 0 \\
\hline
129 & 228.7 & 2,959 & $1.197\times10^{40}$ & 0 & 3,642 & 0 \\
\hline
131 & 320.3 & 1,979 & $1.189\times10^{41}$ & 0 & 3,092 & 0 \\
\hline
133 & 307.9 & 2,553 & $4.113\times10^{44}$ & 0 & 3,813 & 0 \\
\hline
135 & 206.3 & 4,109 & $2.253\times10^{34}$ & 0 & 4,146 & 0 \\
\hline
137 & 374.7 & 2,172 & $7.956\times10^{42}$ & 0 & 3,574 & 0 \\
\hline
139 & 390.4 & 2,239 & $3.229\times10^{43}$ & 0 & 3,712 & 0 \\
\hline
141 & 271.2 & 3,559 & $7.855\times10^{43}$ & 0 & 4,441 & 1 \\
\hline
\textbf{Total} & \textbf{4,826.0} & \textbf{69,202} & -- & \textbf{0} & \textbf{84,996} & \textbf{1} \\
\hline
\end{longtable}
The prime found containing an MSTD multiplicative subgroup of order 141 is $p=4231$, and this MSTD set is given below:
\begin{align*}
    \{&1, 44, 76, 90, 107, 113, 126, 144, 174, 267, 284, 306, 369, 429, 477, 531, 554, 559, 564, 579, 620, 659, \\
    &730, 741, 769, 771, 789, 797, 823, 827, 868, 929, 992, 996, 1018, 1022, 1027, 1114, 1168, 1210, 1220, \\
    &1249, 1268, 1313, 1338, 1404, 1514, 1545, 1694, 1708, 1754, 1814, 1842, 1894, 1936, 1952, 1962, \\
    &2004, 2090, 2101, 2105, 2108, 2143, 2154, 2163, 2182, 2209, 2235, 2277, 2284, 2293, 2364, 2404, \\
    &2468, 2472, 2475, 2482, 2503, 2540, 2542, 2609, 2615, 2645, 2658, 2715, 2719, 2738, 2769, 2797,\\
    &2817, 2875, 2878, 2908, 2926, 2947, 2967, 2987, 2993, 3109, 3125, 3129, 3151, 3183, 3221, 3225, \\
    &3252, 3282, 3286, 3314, 3319, 3344, 3368, 3425, 3433, 3441, 3465, 3543, 3556, 3576, 3579, 3593, \\
    &3610, 3618, 3629, 3658, 3661, 3703, 3769, 3801, 3812, 3869, 3901, 3933, 3960, 4025, 4034, 4064, \\
    &4114, 4148, 4184, 4219\}
\end{align*}